\documentclass[12pt]{amsart}
\usepackage{amscd,amssymb,amsthm,amsmath,amssymb,mathrsfs,enumerate,longtable,hyperref,url,mathrsfs,float}
\usepackage[all]{xy}
\usepackage[margin=2cm]{geometry}

\newtheorem{theorem}[equation]{Theorem}
\newtheorem*{theorem*}{Theorem}

\newtheorem{proposition}[equation]{Proposition}
\newtheorem*{proposition*}{Proposition}
\newtheorem{lemma}[equation]{Lemma}
\newtheorem{corollary}[equation]{Corollary}
\newtheorem*{corollary*}{Corollary}

\newtheorem*{problem*}{Problem}

\newtheorem*{question*}{Question}

\newtheorem*{construction*}{Construction}
\newtheorem*{theoremA*}{Theorem A}
\newtheorem*{theoremAprime*}{Theorem A$^\prime$}
\newtheorem*{theoremB*}{Theorem B}
\newtheorem*{maintheorem*}{Main Theorem}
\newtheorem*{conjecture*}{Conjecture}

\theoremstyle{definition}

\newtheorem*{example*}{Example}
\newtheorem{definition}[equation]{Definition}
\newtheorem*{definition*}{Definition}

\theoremstyle{remark}
\newtheorem{remark}[equation]{Remark}
\newtheorem*{remark*}{Remark}

\makeatletter\@addtoreset{equation}{section} \makeatother

\author{Ivan Cheltsov, Antoine Pinardin, Yuri Prokhorov}

\title{Simple subgroups of the real space Cremona group}

\address{\emph{Ivan Cheltsov}\newline\textnormal{University of Edinburgh, Edinburgh, Scotland
\newline
\texttt{i.cheltsov@ed.ac.uk}}}

\address{\emph{Antoine Pinardin}\newline\textnormal{University of Basel, Basel, Switzerland}
\newline
\textnormal{\texttt{antoine.pinardin@unibas.ch}}}

\address{\emph{Yuri Prokhorov}\newline\textnormal{Steklov Mathematical Institute, Moscow, Russia
\newline
Department
of Algebra, Moscow State
University, Russia
\newline
National Research University Higher School of Economics, Moscow, Russia }
\newline
\textnormal{\texttt{prokhoro@mi-ras.ru}}}

\newcommand{\mumu}{{\boldsymbol{\mu}}}
\newcommand\FF{\mathbf{F}}
\newcommand\CC{\mathbb{C}}
\newcommand\PP{\mathbb{P}}
\newcommand\QQ{\mathbb{Q}}
\newcommand\GG{\mathbb{G}}
\newcommand\RR{\mathbb{R}}
\newcommand\B{\mathbf{B}}
\newcommand\CCC{\mathcal{C}}
\newcommand{\OOO}{{\mathscr{O}}}
\newcommand{\EEE}{{\mathscr{E}}}
\newcommand{\PPP}{{\mathscr{P}}}
\newcommand{\MMM}{{\mathscr{M}}}

\newcommand\Aut{\operatorname{Aut}}
\newcommand\SL{\operatorname{SL}}
\newcommand\PSL{\operatorname{PSL}}
\newcommand\GL{\operatorname{GL}}
\newcommand\Sing{\operatorname{Sing}}
\newcommand\Gr{\operatorname{Gr}}
\newcommand\Bs{\operatorname{Bs}}
\newcommand\g{\operatorname{g}}
\newcommand\dd{\operatorname{d}}
\newcommand\rk{\operatorname{rk}}

\renewcommand{\theenumi}{\rm (\roman{enumi})}
\renewcommand{\labelenumi}{\rm (\roman{enumi})}

\newcounter{NN}\numberwithin{NN}{section}
\renewcommand{\theNN}{\arabic{NN}${}^o$}
\def\nr{\refstepcounter{NN}{\theNN}}%
\usepackage[dvipsnames]{xcolor}
\usepackage{graphicx}

\usepackage[normalem]{ulem}

\begin{document}

\begin{abstract}
We show that the alternating groups $\mathfrak{A}_5$ and $\mathfrak{A}_6$ are the only finite simple non-abelian subgroups of the group of birational selfmaps of the real three-dimensional projective space.
\end{abstract}

\maketitle
\setcounter{tocdepth}{1}
\tableofcontents

\section*{Introduction}

Every finite subgroup of the group $\mathrm{Bir}(\PP^1_{\CC})=\mathrm{Aut}(\PP^1_{\CC})\simeq\mathrm{PGL}_2(\CC)$ is isomorphic to one of the following groups: the cyclic group $\mumu_n$ of order $n$, the abelian non-cyclic group $\mumu_2^2$, the dihedral group $\mathfrak{D}_n$ of order $2n$, the alternating group $\mathfrak{A}_4$, the symmetric group $\mathfrak{S}_4$, the simple group $\mathfrak{A}_5$. Finite subgroups of the plane Cremona group $\mathrm{Cr}_2(\CC):=\mathrm{Bir}(\PP^2_{\CC})$ have been essentially classified by Blanc, Dolgachev and Iskovskikh \cite{Blanc1,Blanc2,IgorVasya}. In particular, we know that, up to isomorphism, the group $\mathrm{Cr}_2(\CC)$ contains exactly three finite \textit{simple} non-abelian subgroups, and they are isomorphic to $\mathfrak{A}_5$, $\PSL_2(\FF_7)$ or $\mathfrak{A}_6$. It seems to be not feasible to obtain a full classification of all finite subgroups of the space Cremona  group $\mathrm{Cr}_3(\CC):=\mathrm{Bir}(\PP^3_{\CC})$. However, finite \textit{simple} non-abelian subgroups of this group have been classified in \cite{P:JAG:simple}. Namely, up to isomorphism, the group $\mathrm{Cr}_3(\CC)$ contains exactly $6$ such subgroups, and they are isomorphic to 
\begin{center}
$\mathfrak{A}_5$, $\PSL_2(\FF_7)$, $\mathfrak{A}_6$, $\mathfrak{A}_7$, $\SL_2(\FF_8)$, $\mathrm{PSp}_{4}(\FF_3)$.
\end{center}
Our aim is to obtain a similar result over the field of real numbers: we aim to classify finite \textit{simple} non-abelian subgroups of the real space Cremona group $\mathrm{Cr}_3(\mathbb{R}):=\mathrm{Bir}(\PP^3_{\mathbb{R}})$. In dimension two, an analogous problem has been solved by Yasinsky in \cite{Egor-simple-groups}, who proved that every finite \textit{simple} non-abelian subgroup of the real plane Cremona group $\mathrm{Cr}_2(\mathbb{R}):=\mathrm{Bir}(\PP^2_{\mathbb{R}})$ is isomorphic to $\mathfrak{A}_5$. Our main result is the following theorem:

\begin{maintheorem*}
Let $G$ be a finite simple non-abelian  subgroup of the real Cremona group $\mathrm{Cr}_3(\mathbb{R})$. Then either $G\simeq\mathfrak{A}_5$ or $G\simeq\mathfrak{A}_6$.
\end{maintheorem*}

For any finite group $G$, its embedding $G\hookrightarrow\mathrm{Cr}_3(\mathbb{R})$ (if it exists) arises from a faithful $G$-action on a real geometrically irreducible 3-fold that is rational over $\mathbb{R}$. Thus, our Main Theorem can be geometrically restated as follows.

\begin{maintheorem*}
Let $X$ be a real geometrically irreducible 3-fold such that $X$ is rational over $\mathbb{R}$, and $\mathrm{Aut}(X)$ contains a simple non-abelian finite subgroup $G$. Then either $G\simeq\mathfrak{A}_5$ or $G\simeq\mathfrak{A}_6$.
\end{maintheorem*}

Note that the group $\mathrm{Cr}_3(\mathbb{R})$ contains a subgroup isomorphic to $\mathfrak{A}_6$.
 
\begin{example*}
\label{example:A6}
The Segre cubic 3-fold
$$
X_3:=\Big\{\sum_{i=1}^{6}x_i=\sum_{i=1}^{6}x^3_i=0\Big\}\subset\PP^5,
$$
admits a faithful action of the group $\mathfrak{A}_6$. It is known that  the Segre cubic is $\mathfrak{A}_6$-birationally superrigid over $\CC$ \cite{cheltsov2014five}.  It is easy to see that $X_3$ is $\mathbb{R}$-rational. Hence the Cremona group $\mathrm{Cr}_3(\mathbb{R})$ contains a subgroup isomorphic to $\mathfrak{A}_6$. 
\end{example*}
We believe that this subgroup is unique:

\begin{conjecture*}
Up to conjugation, $\mathrm{Cr}_3(\mathbb{R})$ contains a unique subgroup isomorphic to $\mathfrak{A}_6$.
\end{conjecture*}

On the other hand, the group $\mathrm{Cr}_3(\mathbb{R})$ contains a lot of subgroups isomorphic to $\mathfrak{A}_5$
(cf. \cite{Krylov}):
\begin{example*}
\label{example:A5-A6}
Consider  following real varieties: 
\begin{enumerate}
\renewcommand{\theenumi}{\rm (\arabic{enumi})}
\renewcommand{\labelenumi}{\rm (\arabic{enumi})}
\item 
$X_1:=\PP^3$, 
\item
$X_2:=\{x_1^2+x_2^2+x_3^2+x_4^2=x_5^2\}\subset\PP^4$, 
\item
the Segre cubic 3-fold $X_3$ (see above),
\item
$X_4:=\PP^1\times S$, where $S=\{x_1^2+x_2^2+x_3^2=x_4^2\}\subset\PP^3$, 
\item
$X_5:=\Gr(2,5)\cap \PP^6\subset \PP^9$, the intersection of the real Grassmannian $\Gr(2,5)$ embedded into $\PP^9$ with a real $\mathfrak{A}_5$-invariant  linear subspace of codimension $3$ such that $X_5$ is smooth, 
\item
the hypersurface $X_6$ in $\PP(1_{x_1},1_{x_2},1_{x_3},1_{x_4},3_{y})$ that is given the equation
$$
y^2=4(\tau^2x_1^2-x_2^2)(\tau^2x_2^2-x_3^2)(\tau^2x_3^2-x_1^2)-(1+2\tau)x_4^2(x_1^2+x_2^2+x_3^2-x_4^{2})^2
$$
where $\tau=\frac{1+\sqrt{5}}{2}$. This is  the real form of the Barth sextic double solid \cite{Barth,BuBa}.
\end{enumerate}
 Each of these 3-folds is rational over $\mathbb{R}$, and its automorphism group contains a subgroup isomorphic to $\mathfrak{A}_5$.  In fact, the groups $\mathrm{Aut}(X_2)\simeq\mathrm{SO}(4,1)$ and $\mathrm{Aut}(X_3)\simeq\mathfrak{S}_6$ contain two such subgroups. Fix $\mathfrak{A}_5$-actions on these 3-folds such that the actions on $X_2$ and $X_3$ leave invariant exactly one hyperplane section. Then it follows from \cite{Avilov,BuBa,Icosahedron,KollarSzabo,PinardingZhang} that any two 3-folds listed above are not $\mathfrak{A}_5$-birational over $\mathbb{C}$, so, in particular, they are not $\mathfrak{A}_5$-birational over $\mathbb{R}$.
\end{example*}

Let us say few words about the proof of Main Theorem. Since $\mathrm{Cr}_3(\mathbb{R})\subset\mathrm{Cr}_3(\mathbb{C})$, we know from \cite[Theorem 1.3]{P:JAG:simple} that any finite simple non-abelian subgroup of the group $\mathrm{Cr}_3(\mathbb{R})$ is isomorphic to one of the groups $\mathfrak{A}_5$, $\PSL_2(\FF_7)$, $\mathfrak{A}_6$, $\mathfrak{A}_7$, $\SL_2(\FF_8)$, $\mathfrak{A}_7$, $\mathrm{PSp}_{4}(\FF_3)$. We already know that $\mathrm{Cr}_3(\mathbb{R})$ contains subgroups isomorphic to $\mathfrak{A}_5$ and $\mathfrak{A}_6$. Thus, Main Theorem follows from the following two results, which have very different proofs.

\begin{theoremA*}
Let $X$ be a real geometrically irreducible 3-fold such that $X$ is rational over $\mathbb{R}$, and let $G$ be a subgroup of the group $\mathrm{Aut}(X)$.
Then $G$ is not isomorphic to $\SL_2(\FF_8)$, $\mathfrak{A}_7$ or $\mathrm{PSp}_{4}(\FF_3)$.
\end{theoremA*}

\begin{theoremB*}
Let $X$ be a real geometrically irreducible 3-fold such that $X$ is rational over $\mathbb{R}$, and let $G$ be a subgroup of the group $\mathrm{Aut}(X)$.
Then $G$ is not isomorphic to $\PSL_2(\FF_7)$.
\end{theoremB*}

The proof of Theorem A is short, and it based on the technique developed in \cite{cheltsov2014five,CheltsovShramovKlein,CheltsovShramovSelecta,Icosahedron,CheltsovSarikyan,CheltsovSarikyanZhuang,PinardingZhang}. This is done in Part \ref{part:big-groups}. In fact, we prove a stronger result:

\begin{theoremAprime*}
Let $X$ be a real geometrically irreducible 3-fold such that $X$ is rational over $\mathbb{C}$, and let $G$ be a subgroup of the group $\mathrm{Aut}(X)$.
Then $G$ is not isomorphic to $\SL_2(\FF_8)$, $\mathfrak{A}_7$ or $\mathrm{PSp}_{4}(\FF_3)$.
\end{theoremAprime*}

The proof of Theorem B is long, and it is based on the technique developed in \cite{Prokhorov-p-groups,P:JAG:simple,Prokhorov-2-groups,YuraCostya-p-groups,Kuznetsova,Loginov,YuraS6}. This is done in Part \ref{part:Klein}. The example below shows that there is a real geometrically irreducible 3-fold such that it is rational over $\mathbb{C}$, and its automorphism group contains a subgroup isomorphic to $\PSL_2(\FF_7)$, so we do need $\mathbb{R}$-rationality condition in Theorem~B.

\begin{example*}[{\cite[Example~8.2]{RonanSusanna}}]
Let $Y$ be the 3-fold $\{x_1y_1+x_2y_2+x_3y_3=0\}\subset\PP^2_{x_1,x_2,x_3}\times\PP^2_{y_1,y_2,y_3}$. Then $Y$ is rational over $\mathbb{R}$, and it has a real form $X$ such that $\mathrm{Aut}(X)\simeq\mathrm{PSU}_3(\mathbb{C})\times\mumu_2$ and $X(\mathbb{R})=\varnothing$, so $X$ is not rational over $\mathbb{R}$, and $\mathrm{Aut}(X)$ contains a subgroup isomorphic to $\PSL_2(\FF_7)$. 
\end{example*}

It would be interesting to compare our Main Theorem with \cite[Theorem 1.1]{topologists}.

\medskip
\noindent
\textbf{Acknowledgements.} We would like to thank Gavin Brown, Alex Degtyarev, Alex Duncan, Sebastian Fuentes Olguin, Ilya Itenberg, Takashi Kishimoto, Frederic Mangolte, Grisha Mikhalkin, Jennifer Paulhus, Miles Reid, Costya Shramov, Ronan Terpereau, Yuri Tschinkel, Zhijia Zhang, Egor Yasinsky for fruitful discussions. We started to work on this project at CIRM (Luminy) during a semester-long Morlet Chair (January--June 2025), and we finished the project during our visit to the G\"okova Geometry Topology Institute in September 2025. We are very grateful to both institutes for the provided hospitality. Ivan Cheltsov has been supported by Simons Collaboration grant \emph{Moduli of varieties} and by the Leverhulme Trust grant RPG-2021-229. Antoine Pinardin has been supported by the Leverhulme Trust grant RPG-2021-229.
Yuri Prokhorov  has been supported by  the HSE University Basic Research Program.

\medskip
\noindent
\textbf{Notations.}
Throughout this paper, we employ the following notations and assumptions:
\begin{itemize}
\item $\mumu_n$ denotes the cyclic group of order $n$;
\item
$\mathfrak{D}_n$ denotes the dihedral group of order $2n$;
\item
$\mathfrak{S}_n$ and $\mathfrak{A_n}$ are the symmetric and alternating groups, respectively;
\item
$\GL_n(\Bbbk)$, $\SL_n(\Bbbk)$, $\mathrm{PGL}_n(\Bbbk)$, $\PSL_n(\Bbbk)$ are linear groups over a field $\Bbbk$;
\item all varieties are assumed to be normal and projective unless mentioned otherwise;
\item all varieties are assumed to be defined over $\mathbb{C}$ unless stated otherwise;
\item an algebraic variety is said to be real if it is defined over $\mathbb{R}$;
\item a real variety $X$ is said to be \emph{pointless} if $X(\mathbb{R})=\varnothing$;
\item if $X$ is a real variety, its \emph{geometric model} $X_{\CC}$ is the complex variety $X\otimes_{\mathrm{Spec}(\mathbb{R})}\mathrm{Spec}(\CC)$;
\item if $X$ is a complex variety, and $P$ is a point in $X$, we denote by $T_{X,P}$ the Zariski tangent space of the variety $X$ at the point $P$;
\item if $X$ is a complex 3-fold with at most terminal singularities, and $P$ is its singular point, the index of $P$ is the smallest integer $r\geqslant 1$ such that $r(-K_X)$ is Cartier at $P$; 
\item if a variety $X$ is defined over a field $\Bbbk$, then the groups $\mathrm{Pic}(X)$, $\mathrm{Cl}(X)$, $\mathrm{Aut}(X)$ are assumed to be defined for the corresponding objects (line bundles, Weil divisors, automorphisms) that are also defined over $\Bbbk$;
\item a linear system on a variety $\mathcal{M}$ is said to be \emph{mobile} if it does not have fixed components;
\item a variety $X$ is said to be a \emph{Fano variety} if $-K_X$ is a $\mathbb{Q}$-Cartier ample divisor;
\item if $X$ is a Fano variety defined over a subfield $\Bbbk\subset\mathbb{C}$ such that $X_{\CC}$ has canonical Gorenstein singularities, then the \emph{Fano index} $\iota(X)$ is the largest integer $n$ such that $-K_{X_\CC}\sim nA$ for some $A\in\mathrm{Pic}(X_\CC)$;

\item if $X$ is a Fano 3-fold with only canonical Gorenstein singularities, then
the \emph{genus} of $X$ is the positive integer
$$
\g(X)=\dim\big(|-K_X|\big)-1=\frac12 (-K_X)^3+1;
$$
\item if $X$ is a Fano variety defined over a field $\Bbbk$ such that $X$ has Kawamata log terminal singularities, then $\rho(X)$ is the rank of $\mathrm{Pic}(X)$, and $\mathrm{r}(X)$ is the rank of $\mathrm{Cl}(X)$;
\item if $X$ is a Fano variety defined over a field $\Bbbk$ such that $X$ has Kawamata log terminal singularities, and $G$ is a finite subgroup in $\mathrm{Aut}(X)$, then 
\begin{itemize}
\item $\rho^G(X)$ is the rank of the subgroup $\mathrm{Pic}(X)^G\subset\mathrm{Pic}(X)$ that consists of all $G$-invariant classes of Cartier divisors on $X$ that are defined over $\Bbbk$, 
\item $\mathrm{r}^G(X)$ is the rank of the subgroup $\mathrm{Cl}(X)^G\subset\mathrm{Pic}(X)$ that consists of all $G$-invariant classes of Weil divisors on $X$ that are defined over $\Bbbk$;
\end{itemize}
\item a variety $X$ defined over a subfield $\Bbbk\subset\mathbb{C}$ is said to be a \emph{$G\mathbb{Q}$-Fano variety} for a finite subgroup $G\subset\mathrm{Aut}(X)$ if the following $3$ conditions are satisfied:
\begin{enumerate}
\item $X$ has terminal singularities,
\item the anticanonical divisor $-K_X$ is ample,
\item for every $G$-invariant Weil divisor $D$ on the variety $X$ that is defined over $\Bbbk$, one has
$$
D\sim_{\mathbb{Q}}\lambda(-K_{X})
$$
for some $\lambda\in\mathbb{Q}$, which is equivalent to $\mathrm{r}^G(X)=1$;
\end{enumerate}
\item if $C$ is a geometrically irreducible curve, its arithmetic genus is denoted by $\mathrm{p}_a(C)$;
\item if $C$ is a smooth geometrically irreducible curve, its genus is denoted by $\g(C)$.
\end{itemize}

\part{Big simple groups}
\label{part:big-groups}

In this part, we will prove Theorem A.

\begin{proof}[Proof of Theorem A]
Suppose that $\mathrm{Cr}_3(\mathbb{R})$ has a subgroup isomorphic to $\SL_2(\FF_8)$, $\mathfrak{A}_7$ or $\mathrm{PSp}_{4}(\FF_3)$. Then, arguing as in \cite{P:JAG:simple}, we see that there exists a real rational $G\mathbb{Q}$-Fano 3-fold $X$ such that its full automorphism group $\mathrm{Aut}(X)$ contains a subgroup $G$ that is isomorphic to one of these $3$ groups. Let us seek for a contradiction.

If $G\simeq \mathrm{PSp}_{4}(\FF_3)$, then it follows from Lemma~\ref{lemma:Burkhardt} that $X$ is either a real form of $\PP^3$ or a real form of the Burkhardt quartic in $\PP^4$. In the former case, $X\simeq\PP^3$, since $X(\mathbb{R})\ne\varnothing$, which leads to a contradiction, since $\mathrm{PGL}_4(\mathbb{R})$ has no subgroups isomorphic to $\mathrm{PSp}_{4}(\FF_3)$. In the latter case, the anticanonical embedding $X\hookrightarrow\PP^4$ is $G$-equivariant, which also leads to a contradiction, because $\mathrm{PSp}_{4}(\FF_3)$ has no real $5$-dimensional faithful representations.

Similarly, if $G\simeq\mathfrak{A}_7$, then it follows from \cite{P:JAG:simple,BeauvilleA7} that $X_{\CC}$ is $G$-birational to $\PP^3$, which implies that $X_{\CC}\simeq\PP^3$ by Lemma~\ref{lemma:A7-Fanos}. As above, this leads to a contradiction.

Thus, we may assume that $G\simeq\SL_2(\FF_8)$. Then it follows from \cite{P:JAG:simple} that $X_{\CC}$ is $G$-birational to the smooth Fano 3-fold of Picard rank $1$ and genus $7$ that is described in \cite[Example 2.11]{P:JAG:simple}, and $X_{\CC}$ is isomorphic to this smooth Fano 3-fold by Corollary~\ref{corollary:SL28-Fanos}. Note that the anticanonical embedding $X\hookrightarrow\PP^8$ is $G$-equivariant, and it follows from \cite[Example 2.11]{P:JAG:simple} that the action of the group $G$ on $\PP^8$ is induced by its irreducible $9$-dimensional representation. But $G$ has no real irreducible $9$-dimensional representations, which is a contradiction.
\end{proof}

\begin{proof}[Proof of Theorem A$^\prime$]
Let $X$ be a real 3-fold such that $X_{\CC}$ is rational. Suppose that $\mathrm{Aut}(X)$ has a subgroup $G$ isomorphic to $\SL_2(\FF_8)$, $\mathfrak{A}_7$ or $\mathrm{PSp}_{4}(\FF_3)$. Arguing as in the proof of Theorem A, we see that $X$ is $G$-birational to a form of $\PP^3$, and either $G\simeq\mathfrak{A}_7$ or $G\simeq \mathrm{PSp}_{4}(\FF_3)$. Since $G$ has a subgroup isomorphic to $\PSL_2(\FF_7)$, we obtain a contradiction with Lemma~\ref{lemma:SB} below.
\end{proof}

The proof of Theorem A relies on Corollary~\ref{corollary:SL28-Fanos} and Lemmas \ref{lemma:A7-Fanos} and \ref{lemma:Burkhardt}. We will prove them below in Sections \ref{section:SL28}, \ref{section:A7} and \ref{section:Burkhardt}, respectively. In these sections, we assume that all varieties are defined over the field of complex numbers.

\section{$\mathrm{SL}_2(\mathbf{F}_8)$}
\label{section:SL28}

Recall that $\mathrm{SL}_2(\mathbf{F}_8)$ is the unique simple group of order $504$. Moreover, it has been proved in  \cite{P:JAG:simple} that, up to conjugation, $\mathrm{Cr}_3(\CC)$ contains one subgroup isomorphic to $\mathrm{SL}_2(\mathbf{F}_8)$, and there is a unique smooth Fano 3-fold $X$ of Picard rank $1$ such that $\mathrm{Aut}(X)$ contains a subgroup $G\simeq\mathrm{SL}_2(\mathbf{F}_8)$. The explicit construction of this Fano 3-fold is given in \cite[Example 2.11]{P:JAG:simple}. The goal of this section is to prove the following inequality:

\begin{equation}
\label{equation:SL28-alpha}
\alpha_G(X)\geqslant 2.001,
\end{equation}
where $\alpha_G(X)$ is the $G$-invariant $\alpha$-invariant of Tian of the Fano 3-fold $X$. For the precise algebraic definition of $\alpha_G(X)$, see \cite{CalabiBook}.
The inequality \eqref{equation:SL28-alpha} implies the following two corollaries

\begin{corollary}[{\cite[Example 4.6]{CalabiBook}}]
\label{corollary:SL28-KE}
The 3-fold $X$ admits a Kahler--Einstein metric.
\end{corollary}

\begin{proof}
The assertion follows from \eqref{equation:SL28-alpha} and \cite{Tian}.
\end{proof}

\begin{corollary}
\label{corollary:SL28-Fanos}
Let $\chi\colon X\dashrightarrow X^\prime$ be a $G$-birational map such that $X^\prime$ is a Fano 3-fold with canonical singularities. Then $\chi$ is an isomorphism.
\end{corollary}

\begin{proof}
If $\chi$ is not biregular, then it follows from  \cite{CheltsovUMN,Icosahedron} that there exists a $G$-invariant non-empty mobile linear system $\mathcal{M}\subset|-nK_{X}|$ such that $(X,\frac{1}{n}\mathcal{M})$ has non-terminal singularities. By \cite[Corollary 2.3]{cheltsov2014five}, the singularities of the log pair $(X,\frac{2}{n}\mathcal{M})$ are not Kawamata log terminal, which contradicts \eqref{equation:SL28-alpha}.
\end{proof}

The proof of \eqref{equation:SL28-alpha} is very similar to the proof of \cite[Theorem 4.5]{cheltsov2011exceptional}. First, let us recall from \cite{P:JAG:simple} basic facts about $X$ and $G$. We know that $(-K_X)^3=12$, the divisor $-K_X$ is very ample, the linear system $|-K_X|$ gives a $G$-equivariant embedding $X\hookrightarrow\PP^8$ such that $X$ is an intersection of quadrics in $\PP^8$, and the $G$-action on $\PP^8$ is induced by an irreducible $9$-dimensional representation of the group $G$. Moreover, it follows from \cite[Corollary 5.5]{P:JAG:simple} that $H^0(\PP^8,\, \mathcal{I}_{X}(2))$ is a $10$-dimensional representation of the group $G$, where $\mathcal{I}_X$ is the ideal sheaf of $X$. On the other hand, using GAP \cite{GAP4}, we see that $H^0(\PP^8,\mathcal{O}_{\PP^8}(2))$ is a $45$-dimensional representation, which splits as a sum of one trivial $1$-dimensional representation, one irreducible $8$-dimensional representation, and four irreducible $9$-dimensional representations. Hence, since $X$ is projectively normal \cite{IP99}, we see that $H^0(\PP^8,\, \mathcal{I}_{X}(2))$ is a sum of a trivial $1$-dimensional representation and an irreducible $9$-dimensional representation, and $H^0(X,\mathcal{O}_{X}(-2K_X))$ splits as a sum of one irreducible $8$-dimensional representation, and three $9$-dimensional irreducible representations.

Recall from \cite{atlas} that the maximal proper subgroups of the group $G$ are isomorphic to $\mumu_2^3\rtimes\mumu_7$, $\mathfrak{D}_{9}$, $\mathfrak{D}_{7}$, and recall from \cite{GroupNames} that the proper subgroups of the group $\mumu_2^3\rtimes\mumu_7$ are isomorphic to $\mumu_2$, $\mumu_2^2$, $\mumu_2^3$, $\mumu_7$. Thus, if $\Sigma_n$ is a $G$-orbit in $X$ of length $n\leqslant 35$, then either $n=9$ or $n=28$. Moreover, all subgroup of the group $G$ isomorphic to  $\mumu_2^3$ are conjugated, and it follows from the proof of Proposition 16.7 in \cite{beauville2014finite} that these abelian subgroups do not fix points in $X$. Thus, we see that $X$ has no $G$-orbits of length $9$.

Recall from \cite{lmfdb,Breuer} that the group $G$ cannot faithfully act on a smooth irreducible curve of genus $\leqslant 6$. Moreover, if $G$ acts faithfully on a smooth irreducible curve of genus $\leqslant 15$, then its genus is either $7$ or $15$. Furthermore, there exists exactly one such curve of genus $7$, which is commonly known as the Macbeath curve \cite{hidalgo2018fricke}.

Now, we are ready to prove \eqref{equation:SL28-alpha}.
Suppose that $\alpha_G(X)<2.001$. Let us seek for a contradiction. By definition, there is a positive rational number $\lambda<2.001$
and an effective $G$-invariant $\mathbb{Q}$-divisor $D$ on the 3-fold $X$ such that $D\sim_{\mathbb{Q}}-K_{X}$, and the log pair $(X,\lambda D)$ is strictly log canonical. Let $C$ be an irreducible subvariety in $X$ that is a minimal center of log canonical singularities of this log pair, and let $Z$ be its $G$-orbit in $X$. Then every irreducible component of $Z$ is also a minimal center of log canonical singularities of the log pair $(X,\lambda D)$, which implies that these irreducible components are disjoint, because intersection of two log canonical centers is a log canonical center by \cite[Proposition 1.5]{Kawamata-1}. In particular, if $Z$ is a surface, then $Z=C$, since distinct irreducible components of $Z$ must intersect in this case.

\begin{lemma}
\label{lemma:SL28-surface}
The subvariety $Z$ is not a surface.
\end{lemma}

\begin{proof}
Suppose that $Z$ is a surface. Then, by definition, we have
$$
\lambda D=Z+\Delta\sim_{\mathbb{Q}} \lambda (-K_X),
$$
where $\Delta$ is a $G$-invariant effective $\mathbb{Q}$-divisor on $X$. Since $\mathrm{Pic}(X)=\mathbb{Z}[-K_X]$, $Z\sim d(-K_X)$ for some integer $d\geqslant 1$. Then $d\leqslant 2$, since $Z+\Delta\sim_{\mathbb{Q}} \lambda (-K_X)$ and $\lambda<3$, which implies that $d\leqslant 2$. This is impossible, since $|-K_X|$ and  $|-2K_X|$ contain no $G$-invariant surfaces -- see \cite{P:JAG:simple}.
\end{proof}

Thus, we see that $Z$ is either a curve or the $G$-orbit of a point. Now, we use \cite[Lemma 2.8]{cheltsov2011exceptional}, which is an equivariant version of what is known as \emph{Kawamata--Shokurov trick} or \emph{Tie Breaking}. By this lemma, we can replace $D$ by an effective $G$-invariant $\mathbb{Q}$-divisor $D^\prime\sim_{\mathbb{Q}}-K_{X}$ such that the singularities of the pair $(X,\lambda^\prime D^\prime)$ are strictly log canonical for some rational number $\lambda^\prime<2.001$, $C$ is a minimal center of log canonical singularities of this pair, and $\mathrm{Nklt}\big(X,\lambda^\prime D^\prime\big)=Z$.
Hence, without loss of generality, we may assume that $D^\prime=D$ and $\lambda^\prime=\lambda$.

Let $\mathcal{I}_Z$ be the ideal sheaf of the subvariety $Z\subset X$. Then $\mathcal{I}_Z$ is the multiplier ideal sheaf of the log pair $(X,\lambda D)$, so applying Nadel vanishing theorem \cite[Theorem 9.4.8]{lazarsfeld2003positivity}, we see that
$$
H^1\big(X,\mathcal{O}_{X}(-2K_X)\otimes\mathcal{I}_Z\big)=0.
$$
Hence, we have the following exact sequence of $G$-representations:
$$
0\longrightarrow H^0\big(X,\mathcal{O}_{X}(-2K_X)\otimes\mathcal{I}_Z\big)\longrightarrow H^0\big(X,\mathcal{O}_{X}(-2K_X)\big)\longrightarrow H^0\big(Z,\mathcal{O}_{Z}(-2K_X\vert_{Z})\big)\longrightarrow 0.
$$
Set $q=h^0(X,\mathcal{O}_{X}(-2K_X)\otimes\mathcal{I}_Z)$. Then $h^0(Z,\mathcal{O}_{Z}(-2K_X\vert_{Z}))=35-q$ and $q\in\{0,8,9,17,18,25,27\}$, because $H^0(X,\mathcal{O}_{X}(-2K_X))$ splits as a sum of one irreducible $8$-dimensional representation, and three $9$-dimensional irreducible representations.

\begin{corollary}
\label{corollary:SL28-point}
The subvariety $Z$ is not the $G$-orbit of a point.
\end{corollary}

\begin{proof}
If $C$ is a point, then $|Z|=h^0(Z,\mathcal{O}_{Z})=35-q\leqslant 35$, so $Z$ is a $G$-orbit of length $28$. This gives $q=7$, which is a contradiction.
\end{proof}

Hence, $Z$ is a curve. Let $n$ be the number of irreducible components of this curve, and let $d$ be the degree of the curve $C$. Then $Z$ is a curve of degree $nd$, and it follows from \cite{Kawamata-2} that the curve $C$ is smooth, so the curve $Z$ is also smooth. We let $g=\g(C)$, and we let $A=\epsilon (-K_X)$ for any rational number $\epsilon>0$. Then it follows from \cite[Theorem 1]{Kawamata-2} that $(K_{X}+\lambda D+A)\big\vert_{C}\sim_{\mathbb{Q}} K_C+\Delta_C$ for some effective $\mathbb{Q}$-divisor $\Delta_C$ on the smooth curve $C$, which gives  $(\lambda-1+\epsilon)d\geqslant 2g-2$. Since this holds for any $\epsilon>0$, we conclude that  $1.001d>(\lambda-1)d\geqslant 2g-2$. Then $-2K_X\vert_{C}$ is not special, so the Riemann--Roch theorem gives
$$
35-q=h^0\big(Z,\mathcal{O}_{Z}(-2K_X\vert_{Z})\big)=n(2d-g+1).
$$
On the other hand, going through subgroups of $G$, we see that either $n\in\{1,9,28,36\}$ or $n\geqslant 63$.

\begin{lemma}
\label{lemma:SL28-g-0-1}
One has $g\ne 0$ and $g\ne 1$.
\end{lemma}

\begin{proof}
Suppose that $g=0$ or $g=1$. Then $n\ne 1$. Indeed, if $n=1$, then $Z=C$ and $G$ acts faithfully on $C$, because $G$ is simple and $G$ does not fix points in $X$. But $G$ cannot act faithfully on a smooth curve of genus less then $7$, so either $n\in\{9,28,36\}$ or $n\geqslant 63$.

If $g=1$, then $35-q=2nd$ for $q\in\{0,8,9,17,18,25,27\}$, which gives $n=9$, $q=17$ and $d=1$, so that $C$ is a line, which is absurd. Thus, $g=0$. Then $35-q=n(2d+1)$, which gives $n=9$.

Let $G_C$ be the stabilizer of the curve $C$ in $G$. Then $G_C\simeq\mumu_2^3\rtimes\mumu_7$, and we have natural restriction homomorphism $G_C\to\mathrm{Aut}(C)\simeq\mathrm{PGL}_2(\mathbb{C})$. Taking into account the classification of finite subgroups in $\mathrm{PGL}_2(\mathbb{C})$, we see that the kernel of this homomorphism contains a subgroup isomorphic to $\mumu_2^3$, since the only proper normal subgroup of $G_C$ is isomorphic to $\mumu_2^3$. However, as we mentioned earlier, any subgroup of $G$ isomorphic to $\mumu_2^3$ does not fix points in $X$.
\end{proof}

Thus,  $g\geqslant 2$. Then $d\geqslant 4$. Now, using $1.001d>2g-2$ and $35-q=n(2d-g+1)$, we get $n=1$. Thus, since $35-q=2d-g+1$ and $1.001d>2g-2$, we get $g\leqslant 12$, so it follows from \cite{lmfdb} that $g=7$, and $C$ is unique up to isomorphism. But we have $41-2d=q\in\{0,8,9,17,18,25,27\}$, where $d\geqslant 12$, since $1.001d>12$. Then $d=12$ or $d=16$.

Since $H^0(X,\mathcal{O}_{X}(-K_X))$ is an irreducible $9$-dimensional representation of the group $G$,
the curve  $C$ is not contained in a hyperplane in $\PP^8$. Thus, the restriction homomorphism of $G$-representations
$$
H^0\big(X,\mathcal{O}_{X}(-K_X)\big)\longrightarrow H^0\big(C,\mathcal{O}_{C}(-K_X\vert_{C})\big)
$$
must be injective. On the other hand, the Riemann--Roch theorem gives
$$
h^0\big(C,\mathcal{O}_{C}(-K_X\vert_{C})\big)=d-6+h^0\big(C,\mathcal{O}_{C}(K_C+K_X\vert_{C})\big).
$$
Thus, if $d=12$, we have
$$
9=h^0(X,\mathcal{O}_{X}(-K_X))\leqslant h^0\big(C,\mathcal{O}_{C}(-K_X\vert_{C})\big)=6+h^0\big(C,\mathcal{O}_{C}(K_C+K_X\vert_{C})\big)\leqslant 7,
$$
because $K_C+K_X\vert_{C}$ is a divisor of degree $0$. This shows that $d=16$, so that
$$
h^0\big(C,\mathcal{O}_{C}(-K_X\vert_{C})\big)=10+h^0\big(C,\mathcal{O}_{C}(K_C+K_X\vert_{C})\big)=10,
$$
so $H^0(C,\mathcal{O}_{C}(-K_X\vert_{C}))$ has a $1$-dimensional subrepresentation of the group $G$, and, therefore, the linear system $|-K_X\vert_{C}|$ contains a $G$-invariant effective divisor. This is impossible, since this divisor has degree $16$, but $C$ has no $G$-orbits of length less than $72$.

The obtained contradiction completes the proof of \eqref{equation:SL28-alpha}.

\section{$\mathfrak{A}_7$}
\label{section:A7}

Recall that $\mathfrak{A}_7$ is the unique simple group of order $2520$.
Moreover, up to conjugation, the group $\mathrm{PGL}_4(\CC)$ has a unique subgroup isomorphic to $\mathfrak{A}_7$. Let $G$ be one of such subgroups. Then $\PP^3$ is $G$-birationally rigid \cite{cheltsov2014five,CheltsovShramovSelecta,P:JAG:simple}. The goal of this section is to prove the followoing result.

\begin{lemma}
\label{lemma:A7-Fanos}
Let $\chi\colon \PP^3\to X$ be a $G$-equivariant birational map such that $X$ is a Fano 3-fold with canonical singularities. Then $\chi$ is an isomorphism.
\end{lemma}

\begin{proof}
Suppose that the required assertion is not true. Then it follows from \cite{CheltsovUMN,Icosahedron} that there exists a non-empty mobile linear subsystem $\mathcal{M}\subset|\mathcal{O}_{\PP^3}(n)|$ such that the singularities of the log pair $(\PP^3,\frac{4}{n}\mathcal{M})$ are not terminal. Let us show that this leads to a contradiction.

The singularities of the log pair $(\PP^3,\frac{8}{n}\mathcal{M})$ are not Kawamata log terminal. Choose a positive rational number $\mu\leqslant \frac{8}{n}$ such that the singularities of the log pair $(\PP^3,\mu\mathcal{M})$ are strictly log canllonical. Let $C$ be an irreducible subvariety in $\PP^3$ that is a minimal center of log canonical singularities of the log pair $(\PP^3,\mu\mathcal{M})$, and let $Z$ be its $G$-orbit in $\PP^3$. Then every irreducible component of $Z$ is also a minimal center of log canonical singularities of the log pair $(\PP^3,\mu\mathcal{M})$, which implies that these irreducible components are disjoint by \cite[Proposition 1.5]{Kawamata-1}.

Note also that $Z$ is not a surface, because $\mathcal{M}$ is mobile. Thus, either $Z$ is a $G$-irreducible curve or $Z$ is the $G$-orbit of a point. Moreover, if $Z$ is a curve, then it follows from \cite[Lemma 1.8]{Corti} that $\deg(Z)\leqslant 16$, which implies that $Z$ is irreducible and $Z=C$. Indeed, if $Z$ is a reducible curve, then it follows from $\deg(Z)\leqslant 16$ that either $Z$ is a union of $15$ lines, or $Z$ is a union of $7$ conics, because the only subgroups of $G$ of index $\leqslant 16$ are subgroups isomorphic to $\mathfrak{A}_6$ and $\PSL_2(\mathbf{F}_7)$. But these subgroups do not leave invariant lines and conics in $\PP^3$. Thus, if $Z$ is a curve, then $Z=C$.

Now, applying \cite[Lemma 2.8]{cheltsov2011exceptional}, we see that there exists a $G$-invariant $\mathbb{Q}$-divisor $D$ on $\PP^3$ such that $D\sim_{\mathbb{Q}}\mathcal{O}_{\PP^3}(4)$, the singularities of the log pair $(\PP^3,\lambda D)$ are strictly log canonical for some positive rational number $\lambda<2.001$, the subvariety $C$ is a minimal center of log canonical singularities of the log pair $(\PP^3,\lambda D)$, and
$\mathrm{Nklt}(\PP^3,\lambda D)=Z$.

Let $\mathcal{I}_Z$ be the ideal sheaf of the subvariety $Z$. Then, applying \cite[Theorem 9.4.8]{lazarsfeld2003positivity}, we obtain the following exact sequence of $G$-representations:
$$
0\longrightarrow H^0\big(\PP^3,\, \mathcal{I}_Z(5)\big)\longrightarrow H^0\big(\PP^3,\mathcal{O}_{\PP^3}(5)\big)\longrightarrow H^0\big(Z,\mathcal{O}_{Z}(\mathcal{O}_{\PP^3}(5)\vert_{Z})\big)\longrightarrow 0.
$$
Thus, setting $q=h^0(\PP^3,\, \mathcal{I}_Z(5))$, we obtain $h^0(Z,\mathcal{O}_{Z}(\mathcal{O}_{\PP^3}(5)\vert_{Z}))=56-q$. Moreover, the representation $H^0\big(\PP^3,\mathcal{O}_{\PP^3}(5)\big)$ splits as a sum of a $20$-dimensional irreducible representation and a $36$-dimensional irreducible representation. Then $q\in\{0,20,36\}$. Thus, if $C$ is a point, then
$$
|Z|=h^0\big(Z,\mathcal{O}_{Z}\big)=56-q\in\{20,36,56\},
$$
which is impossible, since $\PP^3$ has no $G$-orbits of length $20$, $36$ or $56$, because $G$ has no subgroups of indices $20$, $36$ or $56$. In fact, $\PP^3$ has no $G$-orbits of length $<120$.

Hence, we see that $Z$ is an irreducible curve, so $Z=C$. Then $C$ is smooth by \cite{Kawamata-2}, and $G$ acts faithfully on $C$. Let $g=\g(C)$. Then $g\ne 0$ and $g\ne 1$, so it follows from the Hurwitz bound that $g\geqslant 31$. On the other hand, it follows from \cite[Theorem 1]{Kawamata-2} that $4(\lambda-1)\deg(C)\geqslant 2g-2$, so the Riemann--Roch theorem gives $56\geqslant 56-q=h^0(Z,\mathcal{O}_{Z}(\mathcal{O}_{\PP^3}(5)\vert_{Z}))=5\deg(C)-g+1$.
Combining these inequalities, we obtain a contradiction with $g\geqslant 31$.
\end{proof}

Arguing as in the proof of Lemma~\ref{lemma:A7-Fanos}, we get an alternative proof of \cite[Proposition 1.1]{YuraS6}.

\section{$\mathrm{PSp}_{4}(\FF_3)$}
\label{section:Burkhardt}

Recall that $\mathrm{PSp}_{4}(\FF_3)$ is the unique simple group of order $25920$, and it can be geometrically described as the automorphism group of the Burkhardt quartic
$$
\big\{x_5^4-x_5\big(x_1^3+x_2^3+x_3^3+x_4^3\big)+3x_1x_2x_3x_4=0\big\}\subset\PP^4.
$$
Note that the Burkhardt quartic is a $\mathrm{PSp}_{4}(\FF_3)\mathbb{Q}$-Fano 3-fold \cite{BuBa,VanyaYuraZhijia}.
The goal of this section is to show the following result.

\begin{lemma}
\label{lemma:Burkhardt}
Let $X$ be a Fano variety with canonical singularities such that $\mathrm{Aut}(X)$ contains a subgroup $G\simeq\mathrm{PSp}_{4}(\FF_3)$. Then either $X\simeq\PP^3$ or $X$ is isomorphic to the Burkhartdt quartic.
\end{lemma}

\begin{proof}
It has been proven in \cite{P:JAG:simple} that there exists a $G$-birational map $\chi\colon X\dasharrow Y$ such that either $Y=\PP^3$ or $Y$ is the Burkhardt quartic. Let us show that $\chi$ is an isomorphism. Suppose that this is not true. Then it follows from \cite{CheltsovUMN,Icosahedron} that there exists a non-empty mobile linear system $\mathcal{M}$ on $Y$ such that the singularities of the log pair $(Y,\lambda\mathcal{M})$ are not terminal for $\lambda\in\mathbb{Q}_{>0}$ such that $\lambda\mathcal{M}\sim_{\mathbb{Q}}-K_Y$. Let us seek for a contradiction.

The singularities of the log pair $(Y,2\lambda\mathcal{M})$ are not Kawamata log terminal. Indeed, if $Y=\PP^3$, this is obvious. If $Y$ is the Burkhartdt quartic in $\PP^4$, this follows from \cite[Theorem 3.10]{Corti}. Choose a positive rational number $\lambda^\prime\leqslant \lambda$ such that the singularities of the log pair $(Y,2\lambda^\prime\mathcal{M})$ are strictly log canonical. Let $C$ be an irreducible subvariety in $Y$ that is a minimal center of log canonical singularities of the log pair $(Y,2\lambda^\prime\mathcal{M})$, and let $Z$ be its $G$-orbit in $Y$. Then every irreducible component of $Z$ is also a minimal center of log canonical singularities of this log pair, which implies that these components are disjoint by \cite[Proposition 1.5]{Kawamata-1}.

Since $\mathcal{M}$ is mobile, $Z$ is either a $G$-irreducible curve or the $G$-orbit of a point. If $Z$ is a curve, then it follows from \cite[Lemma 1.8]{Corti} that $-K_{Y}\cdot Z\leqslant (-K_Y)^3$. So, if $Y=\PP^3$, then $\deg(Z)\leqslant 16$. Similarly, if $Y$ is the Burkhartdt quartic in $\PP^4$, then $\deg(Z)\leqslant 4$.
In both cases, the curve $Z$ must be irreducible, because $G$ cannot act non-trivially on the set of less than $27$ elements. In particular, if $Y$ is the Burkhartdt quartic in $\PP^4$, then $Z$ is not a curve, because $G$ cannot act faithfully on a curve of genus less that $310$ by the Hurwitz bound.

Now, using \cite[Lemma 2.8]{cheltsov2011exceptional}, we can find a $G$-invariant effective $\mathbb{Q}$-divisor $D$ on $Y$ such that $D\sim_{\mathbb{Q}}-K_{Y}$, the singularities of the log pair $(Y,\mu D)$ are strictly log canonical for some positive rational number $\mu<2.001$, the subvariety $C$ is a minimal center of log canonical singularities of the log pair $(Y,\mu D)$, and $\mathrm{Nklt}(Y,\mu D)=Z$.

Let $\mathcal{I}_Z$ be the ideal sheaf of the subvariety $Z$, and let $H$ be any Cartier divisor on $Y$ such that the divisor $H-(K_Y+\mu D)$ is ample. Then $H^1(Y,\mathcal{O}_{Y}(H)\otimes\mathcal{I}_Z)=0$ by \cite[Theorem 9.4.8]{lazarsfeld2003positivity}. If $Y$ is the Burkhartdt quartic in $\PP^4$, we chose $H=-2K_Y$, which gives us the following exact sequence of $G$-representations:
$$
0\longrightarrow H^0\big(Y,\mathcal{O}_{Y}(-2K_Y)\otimes\mathcal{I}_Z\big)\longrightarrow H^0\big(Y,\mathcal{O}_{Y}(-2K_Y)\big)\longrightarrow H^0\big(Z,\mathcal{O}_{Z}(\mathcal{O}_{Y}(-2K_Y)\vert_{Z})\big)\longrightarrow 0.
$$
In this case, $Z$ is the $G$-orbit of a point, so the exact sequence gives
$$
|Z|=h^0\big(Y,\mathcal{O}_{Y}(-2K_Y)\big)-h^0\big(Y,\mathcal{O}_{Y}(-2K_Y)\otimes\mathcal{I}_Z\big)=15-h^0\big(Y,\mathcal{O}_{Y}(-2K_Y)\otimes\mathcal{I}_Z\big)\leqslant 15,
$$
which is a contradiction, since $G$ cannot act non-trivially on the set of less than $27$ elements.

Hence, we see that $Y=\PP^3$. In this case the $G$-action on $Y$ is induced by an irreducible representation of the central extension $\widehat{G}\simeq 2.\mathrm{PSp}_{4}(\FF_3)$ of the group $G$. Then we choose $H=\mathcal{O}_{\PP^3}(5)$, which gives us the following exact sequence of $\widehat{G}$-representations:
$$
0\longrightarrow H^0\big(Y,\, \mathcal{I}_Z(5)\big)\longrightarrow H^0\big(\PP^3,\mathcal{O}_{\PP^3}(5)\big)\longrightarrow H^0\big(Z,\mathcal{O}_{Z}(\mathcal{O}_{\PP^3}(5)\vert_{Z})\big)\longrightarrow 0.
$$
Set $q=h^0(\PP^3,\, \mathcal{I}_Z(5))$. Then $q\in\{0,20,36\}$,
since $H^0\big(\PP^3,\mathcal{O}_{\PP^3}(5)\big)$ is a sum of a $20$-dimensional irreducible representation and a $36$-dimensional irreducible representation of the group $\widehat{G}$. Then
$$
h^0(Z,\mathcal{O}_{Z}\big(\mathcal{O}_{\PP^3}(5)\vert_{Z})\big)=56-q\in\{20,36,56\}.
$$
In particular, if $Z$ is the $G$-orbit of a point, then $|Z|=56-q\in\{20,36,56\}$, so $|Z|=36$, since $G$ has no subgroups of indices $20$ or $56$, which implies that the stabilizer of a point in $Z$ is  a subgroup isomorphic to $\mathfrak{S}_6$, but none of them fixes a point in $\PP^3$. Hence,  $Z$ is a curve.

Recall that $Z$ is an irreducible curve. Moreover, it follows from \cite{Kawamata-2} that $Z$ is smooth and $G$ acts faithfully on it. Let $g=\g(Z)$. Then $g\geqslant 310$ by the Hurwitz bound. On the other hand, $4(\mu-1)\deg(Z)\geqslant 2g-2$ by \cite[Theorem 1]{Kawamata-2}, which contradicts to $g\geqslant 310$, since $\deg(Z)\leqslant 16$ and $\mu<2.001$.
\end{proof}

We know from Lemma~\ref{lemma:Burkhardt} that the Burkhardt quartic 3-fold is not $\mathrm{PSp}_{4}(\FF_3)$-birational to any other Fano 3-fold with canonical singularities. One can also prove this result arguing as in the proofs of \cite[Theorem 1.1]{BuBa} and \cite[Proposition 7.4]{VanyaYuraZhijia}.

\part{The Klein group}
\label{part:Klein}

In this part, we will prove Theorem B.

\begin{proof}[Proof of Theorem B]
Suppose that $\mathrm{Cr}_3(\mathbb{R})$ has a subgroup isomorphic to $\PSL_2(\FF_7)$. Then, arguing as in \cite{P:JAG:simple}, we see that there exists a real rational 3-fold with terminal singularities such that the group $\mathrm{Aut}(X)$ has a subgroup $G\simeq \PSL_2(\FF_7)$, and $X$ is $G\mathbb{Q}$-factorial, i.e., every Weil divisor defined over $\mathbb{R}$ whose class in $\mathrm{Cl}(X)$ is $G$-invariant is $\mathbb{Q}$-Cartier. Then, applying $G$-equivariant Minimal Model Program over $\mathbb{R}$, we may further assume that
\begin{itemize}
\item either $X$ is a real $G\mathbb{Q}$-Fano 3-fold,
\item or there exists a $G$-equivariant conic bundle $X\to S$ such that $S$ is a real normal surface,
\item or there exists a $G$-equivariant del Pezzo fibration $\pi\colon X\to\PP^1$.
\end{itemize}
The last two cases are ruled out by Lemma~\ref{lemma:fibration} below, so $X$ is a real $G\mathbb{Q}$-Fano 3-fold. Then either $X$ has terminal Gorenstein singularities, or the singularities of the 3-fold $X$ are not Gorenstein. In the former case, we know a lot about $X$, and we use this in Section~\ref{section:Gorenstein} to obtain a contradiction. To be precise, this follows from  Theorem~\ref{theorem:PSL27-real-Gorenstein}.  Then, in Section~\ref{section:non-Gorenstein}, we show that the latter (non-Gorenstein) case is also impossible by Theorem~\ref{theorem:PSL27-real-non-Gorenstein}.
\end{proof}

\section{Preliminaries}
\label{section:preliminaries}
In this section, we list known results that are used 
in the proof of Theorem~A.

\subsection{The orbifold Riemann-Roch formula}
\label{subsection:sing}

First, we present the Riemann-Roch formula for (complex) Fano 3-folds with terminal singularities, and some of its applications. To do this, we have to remind the reader what is the \textit{basket} of a three-dimensional terminal singularity \cite{Reid:YPG}.

\begin{construction*}
\label{construction:basket}
Let $(P\in X)$ be a germ of a three-dimensional terminal singularity of index $r>1$. Then it follows from \cite{Mori-singularities} that there is a one-parameter deformation  $\mathfrak X\to \mathbf{D}\ni 0$ over a small disk $\mathbf{D}\subset \CC$ such that the central fiber $\mathfrak X_0$ is isomorphic to $X$ and the general fiber $\mathfrak X_\lambda$ has only cyclic quotient singularities $P_{\lambda,1},\ldots P_{\lambda,k}$. For the explicit construction of this deformation, see \cite{Reid:YPG}. Thus, given 3-fold $X$ with terminal singularities, one can associate
a collection
$$
\B(X)=\bigl\{ P_{\lambda,1},\ldots P_{\lambda,k}\bigr\},
$$
where each $P_{\lambda,j}\in \mathfrak X_\lambda$ is a singularity of type
$$
\dfrac 1{r_{j}}(1,b_{j},-b_{j}).
$$
This collection is uniquely determined by the variety $X$ and called the \textit{basket} of singularities of $X$.
Note that $k=1$ if and only if $(P\in X)$ is a cyclic quotient singularity. Moreover, $r$ is the least common multiple of the indices $r_1,\ldots,r_k$, and $r=r_{1}=\cdots=r_{k}$ except for one  case:
\begin{itemize}
\item[$\mathrm{cAx/4}$:]
 $r=4$,   $r_{1}=4$,  $r_{2}=\cdots=r_{k}=2$.
\end{itemize}
In the following, we will always assume that $r_{1}=r$, and we will write
$$
\B(X)=(r_{1},\dots,r_{k})
$$
for the basket $\B(X)=\{P_{\lambda,1},\dots,P_{\lambda,k}\}$.
\end{construction*}

\begin{definition}
\label{def:simple:HQ}
A three-dimensional terminal singularity $(P\in X)$ of index $r>1$ is \emph{moderate} if it is analytically isomorphic to the quotient
$$
\{x_1x_2+x_3^r+x_4^n=0\}/\mumu_r(1,-1,a,0),\qquad \gcd(r,a)=1.
$$
The number $n$ is called the \emph{axial weight} of $(P\in X)$. In this case $|\B(X)|=n$.
\end{definition}

Now, we assume that $X$ is a (complex) Fano 3-fold that has at most terminal singularities. Then it follows from \cite[Corollary~10.3]{Reid:YPG} and Kawamata-Viehweg vanishing that
\begin{equation}
\label{equation:RR}
\chi\big(mK_X\big)=\frac1{12}m(m-1)(2m-1)K_X^3+\frac{m}{12}\,K_X\cdot c_2(X)+1
+\sum_{P \in\B(X)}c_P\big(mK_X\big)
\end{equation}
for any $m\in\mathbb{Z}$, where at a point $P \in\B(X)$ which is 
a cyclic quotient of type $\frac{1}{r_P}(1,b_P,-b_P)$ the 
local contribution $c_P$ has the form 
$$
c_P\big(mK_X\big)=-\overline{m}\frac{r_P^2-1}{12r_P}+\sum_{j=0}^{\overline{m}-1}\frac{\overline{b_Pm}(r_P-\overline{b_Pm}\,)}{2r_P},
$$
Here $\overline {\phantom{xx}}$ denotes the smallest residue modulo $r_P$.
In particular, if $m=1$, then \eqref{equation:RR}, Serre duality and Kawamata--Viehweg vanishing theorem give
\begin{equation}
\label{eq2.4}
24=-K_X\cdot c_2(X)+\sum_{P\in\B(X)}\biggl(r_P-\frac1{r_P}\biggr).
\end{equation}
Similarly, if $m=-1$, then $H^i(X,\mathcal{O}_X(-K_X))=0$ for $i>0$, so
$$
c_P(-K_X)=\frac{r_P^2-1}{12r_P}-\frac{b_P(r-b_P)}{2r_P}.
$$
Combining the last equality with \eqref{eq2.4}, we obtain
\begin{equation}
\label{equation:RR-K}
\dim\big(\lvert-K_X\rvert\big)= -\frac12\,K_X^3+2-\sum_{P\in\B(X)} \frac{b_P(r_P-b_P)}{2r_P}.
\end{equation}

The following result follows from the Bogomolov--Miyaoka inequality:

\begin{theorem}[{\cite[Theorem~1.6]{IwaiJiangLiu}}]
\label{theorem:BM}
One has $-K_X\cdot c_2(X)>0$.
\end{theorem}

\begin{corollary}
\label{corollary:BM-24}
One has 
$$
\sum_{P\in \B(X)}\biggl(r_P-\frac1{r_P}\biggr)<24.
$$
\end{corollary}

\begin{corollary}
\label{corollary:BM-RR}
Let $N$ be the length of the basket $\B(X)$. If $2K_X$ is Cartier, then $N\leqslant 15$ and
\begin{eqnarray*}
\dim\lvert-K_X\rvert&=&\frac12\,(-K_X)^3+2-\frac{N}{4},
\\
\dim\lvert-2K_X\rvert&=&\frac52\,(-K_X)^3+4-\frac{N}{4}.
\end{eqnarray*}
\end{corollary}

Theorem~\ref{theorem:BM} originated in \cite{Kawamata:bF,KMMT-2000}. In \cite[Theorem 12.1.3]{P:G-MMP}, it has been proven under an additional assumption that $X$ is a $G\mathbb{Q}$-Fano 3-fold for some finite subgroup~\mbox{$G\subset\mathrm{Aut}(X)$}. If~$\mathrm{r}(X)=1$, it follows from \cite{LiuLiu1,LiuLiu2} that $-K_X\cdot c_2(X)\geqslant\frac{1}{3}(-K_X)^3$, which does not always hold without condition $\mathrm{r}(X)=1$. The database \cite{GRDB} and recent results in \cite{LiuLiu1} suggest that probably the following inequality always hold:
\begin{equation}
\label{equation:BM}
-K_X\cdot c_2(X)\geqslant\frac{1}{4}(-K_X)^3.
\end{equation}
However, \cite[\S~2.4]{Prokhorov2007} leaves room for possible exceptions.

\subsection{Almost Fano 3-folds}
\label{subsection:almost-Fanos}

Let $X$ be a (complex) projective 3-fold. Recall from \cite{Prokhorov-g-12} that $X$ is said to be a \emph{weak Fano 3-fold} if $X$ has at most terminal Gorenstein singularities, and the divisor $-K_X$ is nef and big.
If furthermore the natural morphism
$$
\phi\colon X\longrightarrow\overline{X}:=\mathrm{Proj}\Big(\bigoplus_{n\geqslant 0} H^0\big(X,\mathcal{O}_X(-nK_X)\big)\Big)
$$
does not contract divisors, we say that $X$ is an \emph{almost Fano 3-fold}. In this case, we say that $\overline{X}$ is the anticanonical model of the 3-fold $X$, and $\phi$ is the (pluri) anticanonical morphism.  Now, we suppose that $X$ is an almost weak Fano 3-fold. Set $g=\g(X)=\frac12 (-K_X)^3+1$.

\begin{remark}
\label{remark:smoothing}
It follows from \cite{Namikawa} that the anticanonical model $\overline{X}$ admits a $\mathbb{Q}$-Gorenstein smoothing to a smooth complex Fano 3-fold $V$ such that $(-K_{V})^3=(-K_{\overline{X}})^3=(-K_X)^3$. Hence, it follows from the classification of smooth Fano 3-folds that $(-K_{X})^3\leqslant 64$, so $g\leqslant 33$. Moreover, it follows from \cite{SmoothingPic} (see also \cite[Proposition 2.5]{KP:1nodal})   that $\rho(V)=\rho(\overline{X})$ and $\iota(V)=\iota(\overline{X})$. If $(-K_X)^3=64$, then $V\simeq\PP^3$, and the Fano index $\iota(\overline{X})=4$, which gives $X\simeq\overline{X}\simeq\PP^3$. If $\rho(\overline{X})=1$, then
$$
(-K_{\overline{X}})^3=(-K_{V})^3\in\big\{2,4,6,8,10,12,14,16,18,22,24,32,40,54,64\big\}.
$$
In this case, if $64>(-K_{\overline{X}})^3>40$, then $(-K_{\overline{X}})^3=54$ and $\overline{X}$ is a quadric 3-fold in $\PP^3$.
\end{remark}

By the Riemann--Roch theorem and the Kawamata-Viehweg vanishing, we have
$$
\dim(|-K_X|)=\g(X)+1.
$$
The Picard group $\mathrm{Pic}(X)$ and the Weil divisor class group $\mathrm{Cl}(X)$ are known to be finitely generated and torsion free \cite[Proposition 2.1.2]{IP99}, there exists a natural embedding $\mathrm{Pic}(X) \hookrightarrow \mathrm{Cl}(X)$ as a primitive sublattice, and every $\mathbb{Q}$-Cartier divisor on $X$ is Cartier by \cite[Lemma 5.1]{Kawamata-88}.

\begin{remark}
If $X$ is not $\mathbb{Q}$-factorial, then it follows from \cite[Corollary 4.5]{Kawamata-88} that there exists a small birational morphism $f\colon\widetilde{X}\to X$ such that $\widetilde{X}$ is a $\mathbb{Q}$-factorial almost Fano 3-fold, and
$$
-K_{\widetilde{X}}\sim f^*(-K_X).
$$
We say that $f$ is a \emph{factorialization} of $X$. Note that $f$ is not uniquely determined by $X$.
\end{remark}

Starting from now, we assume further that $X$ is $\mathbb{Q}$-factorial and $\rho(X)>1$. Then, applying Minimal Model Program, we obtain an extremal contraction $h\colon X\to Z$ such that $Z$ is not a point, and $-K_X$ is $h$-ample. A priori, we have the following possible cases:
\begin{enumerate}
\item either $h$ is birational and $\dim(Z)=3$,
\item or $Z$ is a smooth rational surface, and $h$ is a conic bundle,
\item or $Z=\PP^1$, and $h$ is a fibration into del Pezzo surfaces of degree $d$ such that
$$
d=(-K_S)^2=(-K_X)^2\cdot S\in\{1,2,3,4,5,6,8,9\},
$$
where $S$ is a smooth del Pezzo surface that is a general fiber of $h$.
\end{enumerate}

\begin{remark}
If $h$ is del Pezzo fibration and $d=9$, then $h$ is a locally trivial $\PP^2$-bundle \cite{Cutkosky}.
\end{remark}

\begin{remark}
If $h$ is a conic bundle, there is a reduced curve $\Delta\subset Z$ such that $h$ is a smooth morphism over $Z\setminus \Delta$, every scheme fiber over $P\in \Delta$ is a singular (reducible or non-reduced) conic. We say that $\Delta$ is the \emph{discriminant curve} of the conic bundle $h$. If $\Delta=\emptyset$, then $X$ is smooth, and the conic bundle $h$ is a locally trivial $\PP^1$-bundle \cite[Proposition 5.2]{Prokhorov-G-Fanos}.
\end{remark}

If $h$ is birational, then $h$ contracts an irreducible surface $S$, and either $h(S)$ is a curve or a point. If $h(S)$ is a curve $C$, then it follows from \cite{Cutkosky} that $C$ is contained in the smooth locus of $Z$, the curve $C$ has planar singularities, and $h$ is the blowup of its ideal sheaf. In this case, we have
\begin{eqnarray*}
(-K_X)^3 &=&(-K_Z)^3 + 2K_Z \cdot C + 2\mathrm{p}_a(C)-2,\\
(-K_X)^2\cdot S &=&-K_Z\cdot C-2\mathrm{p}_a(C)+2,\\
(-K_X) \cdot S^2 &=&2\mathrm{p}_a(C)-2.
\end{eqnarray*}
This gives
\begin{equation}
\label{equation:degrees}
(-K_Z)^3=(-K_X)^3+2K_X^2\cdot S+2\mathrm{p}_a(C)-2\geqslant (-K_X)^3+2(-K_X)^2\cdot S-2.
\end{equation}
Similarly, if $h(S)$ is a point $P$, then $h$ is the blowup of its maximal ideal, and one of the following three cases holds:
\begin{itemize}
\item $Z$ is smooth at $P$, $S\simeq\PP^2$, $S\vert_{S}\sim\mathcal{O}_S(-1)$, $(-K_X)^2\cdot S=4$, and
$$
(-K_Z)^3=(-K_X)^3+8,
$$
\item $(P\in Z)$ is a $\text{cA}$-singularity, $S$ is a quadric in $\PP^3$, $S\vert_{S}\sim\mathcal{O}_S(-1)$, $(-K_X)^2\cdot S=2$, and
$$
(-K_Z)^3=(-K_X)^3+2,
$$
\item $(P\in Z)$ is a quotient singularity of type $\frac{1}{2}(1,1,1)$, $S\simeq\PP^2$, $S\vert_{S}\sim\mathcal{O}_S(-2)$, $(-K_X)^2\cdot S=1$.
\end{itemize}

\begin{lemma}[{\cite{Prokhorov-72,Prokhorov-g-12}}]
\label{lemma:MMP-no-planes}
Suppose that $h$ is divisorial. Let $S$ be the $h$-exceptional surface. Suppose that $(-K_X)^2\cdot S\ne 1$. Then $Z$ is an almost Fano 3-fold, and
$$
(-K_{X})^2\cdot S^\prime \leqslant (-K_{Z})^2\cdot f(S^\prime)
$$
for every irreducible surface $S^\prime\subset X$ such that $S^\prime\ne S$.
\end{lemma}

If $S$ is an irreducible surface in $X$, we say that $S$ is a \textit{plane} if $(-K_X)^2\cdot S=1$. By Lemma~\ref{lemma:MMP-no-planes}, if $X$ does not contains planes, then applying Minimal Model Program to $X$ we either obtain a fibration (a conic bundle or a del Pezzo fibration) or an almost Fano 3-fold that also does not contain planes. More generally, if $S$ is a surface in $X$, we say that its degree is $(-K_X)^2\cdot S\in\mathbb{Z}_{>0}$.
By Lemma~\ref{lemma:MMP-no-planes}, if the almost Fano 3-fold $X$ does not contains surfaces of degree $\geqslant d\geqslant 2$, then, applying Minimal Model Program, we either obtain a fibration or an almost Fano 3-fold that also does not contain surfaces of degree $\geqslant d$.

If $h$ is a fibration into del Pezzo surfaces, then $\rho(X)=2$, and the 3-fold $X$ uniquely determines the following Sarkisov link:
$$
\xymatrix{
X\ar[d]_{h}\ar[rrd]_{\phi}\ar@{-->}[rrrr]^{\chi}&&&&X^\prime\ar[d]^{h^\prime}\ar[lld]^{\phi^\prime}\\
\PP^1&&\overline{X}&&Z^\prime}
$$
where $\chi$ is a composition of flops, $X^\prime$ is an almost Fano 3-fold, $\phi$ and $\phi^\prime$ are pluri-anticanonical morphisms, and $h^\prime$ is an extremal contraction. If the almost Fano 3-fold $X$ is smooth, such links has been studied in \cite{Takeuchi:DP},  \cite{Fukuoka:DP6}, \cite{Fukuoka:DP:Ref}. If $h$ is a conic bundle, we have the following result:

\begin{lemma}[\cite{Prokhorov-g-12}]
\label{MMp-to-surface}
Suppose that the extremal contraction $h\colon X\to Z$ is a conic bundle, and almost Fano 3-fold $X$ does not contain planes. Then $Z$ is a smooth del Pezzo surface. Moreover, if $Z$ contains a $(-1)$-curve $C$, then there exists the following commutative diagram:
$$
\xymatrix{
X\ar@{-->}[rr]^{\chi}\ar[d]_{h}&&\widehat{X}\ar[rr]^{\varphi}&&X^\prime\ar[d]^{h'}\\
Z\ar[rrrr]&&&&Z^\prime}
$$
where $\chi$ is either an isomorphism or a composition of flops, $\widehat{X}$ is a $\mathbb{Q}$-factorial almost Fano 3-fold that does not contain planes, $\varphi$ is an extremal divisorial contraction, $h^\prime$ is a conic bundle that is an extremal contraction, and $Z\to Z^\prime$ is the contraction of the curve $C$.
\end{lemma}

\subsection{One nice lemma}
\label{subsection:nice-lemma}

Let $X$ be a Fano 3-fold with at most terminal singularities. Suppose further that $X$ is a $G\QQ$-Fano $3$-fold for a finite subgroup $G\subset\mathrm{Aut}(X)$.
The goal of this section is to prove the following result inspired by \cite{CampanaFlenner,Sano}.

\begin{lemma}
\label{lemma:MMP}
Let $\MMM$ be a non-empty  $G$-invariant mobile linear system on $X$. Suppose that
$$
-K_X\sim_{\QQ} \lambda \MMM
$$ 
for some positive rational number $\lambda \geqslant 1$. Then one of the following holds:
\begin{enumerate}
\item \label{prop:MMP-1}
$\lambda=1$ and  $(X,\MMM)$ has canonical singularities;
\item \label{prop:MMP-2}
there is a $G$-equivariant birational map $X\dashrightarrow\widehat{X}$ such that $\widehat{X}$ has terminal $G\mathbb{Q}$-factorial singularities, and $\widehat{X}$ has a structure of a $G$-Mori fiber space  $h\colon \widehat{X}\to Z$ with $\dim(Z)\geqslant 1$;
\item \label{prop:MMP-3} 
there is a $G$-equivariant birational map $X\dashrightarrow \widehat{X}$ such that 
$\widehat{X}$ is a $G\mathbb{Q}$-Fano 3-fold with at most Gorenstein terminal singularities and $\iota(\widehat{X})\geqslant 2$.
\end{enumerate}
\end{lemma}

\begin{proof}
Let $c:=\mathrm{ct}(X,\MMM)$ be the canonical threshold of the pair $(X,\MMM)$.
Then the singularities of the log pair  $(X,c\MMM)$ are canonical, and the log pair $(X,(c-\epsilon)\MMM)$  is terminal for any $\epsilon>0$.

Suppose that $c<\lambda$. Let $f: (\widetilde{X}, c\widetilde{\MMM})\to(X, c\MMM)$ be a $G$-equivariant  terminal modification of the log pair $(X,c\MMM)$
(see e.g. \cite[Corollary~4.4.4]{P:G-MMP}). Then $\widetilde{X}$ is $G\QQ$-factorial, the singularities of the log pair  $(\widetilde{X}, c\widetilde{\MMM})$ are terminal, and  $K_{\widetilde{X}}+ c\widetilde{\MMM}\sim_{\QQ}f^*(K_X+ c\MMM)$. Moreover, we have 
$$
K_{\widetilde{X}}+\lambda\widetilde{\MMM}+E\sim_{\QQ}f^*\big(K_X+\lambda\MMM\big) \sim_{\QQ} 0,$$
where $E$ is a non-zero $f$-exceptional effective $G$-invariant  $\QQ$-divisor.
Now, we run $G$-equivariant $(K_{\widetilde X}+ c\widetilde \MMM)$-Minimal Model Program.
Since the divisor $K_{\widetilde X}+c\widetilde \MMM\sim_{\QQ} -(\lambda-c)\MMM-E$ is not pseudoeffective, at the end we obtain a $G$-Mori fiber space 
$h\colon (\widehat X, c\widehat \MMM)\to Z$. We have
$$
K_{\widehat X}+ \lambda\widehat \MMM+\widehat E \sim_{\QQ} 0,
$$
where $\widehat \MMM$ and $\widehat E$ are proper transforms on $\widehat{X}$ of $\MMM$ and $E$, respectively. 
If $Z$ is not a point, we obtain case~\ref{prop:MMP-2}.
If $Z$ is a point, then $\widehat X$ is a $G\QQ$-Fano 3-fold such that 
$-K_{\widehat X}\sim_{\QQ} \widehat \lambda \widehat \MMM$ with $\widehat \lambda>\lambda$.
Then we repeat our procedure. Since terminal Fano 3-folds are bounded,
the process terminates.

Therefore, we may assume that $c\geqslant \lambda$. If $\lambda=1$, we get case \ref{prop:MMP-1}.
It remains to consider the case   $c\geqslant \lambda>1$.
In this  case, the log pair $(X,\MMM)$ has terminal singularities. Then it follows from \cite[Lemma 1.22]{Alexeev:ge} or \cite[Corollary~7.2.1]{P:G-MMP} that $\MMM$ is a linear system of Cartier divisors, and a general member $S\in \MMM$ is smooth.  
Thus, by the adjunction formula $S$ is a del Pezzo surface. 

We claim that $\operatorname{Cl}(X)$ is torsion free. Indeed, suppose that $\operatorname{Cl}(X)$ has a non-trivial torsion element. Then it defines a cyclic cover $\pi\colon X^{\sharp}\to X$ that is \'etale in codimension two. Let $S^{\sharp}:= \pi^{-1}(S)$. 
Then, since $S$ is a smooth rational surface, the restriction $\pi_S: S^{\sharp}\to S$ splits.
On the other hand, $S^{\sharp}$ is an ample divisor, hence it is connected, a contradiction.

Thus, $\operatorname{Cl}(X)^G$ is a cyclic group. Let $A$ be its ample generator. Then $-K_X\sim qA$ and $S\sim aA$, where $q$ and $a$ are positive integers such that $\lambda=q/a$, so $a<q$. Set $A_S=A\vert_{S}$. Then it follows from the adjunction formula that $-K_S\sim (q-a)A_S$, and it follows from properties of del Pezzo surfaces that $\dim(|A_S|)\geqslant 1$.
Now, using the exact sequence
$$
0 \longrightarrow \OOO_X((1-a)A) \longrightarrow \OOO_X(A) \longrightarrow \OOO_S(A) \longrightarrow 0 
$$
and Kawamata--Viehweg vanishing, we obtain $\dim(|A|)\geqslant\dim(|A_S|)\geqslant 1$. Thus, $|A|$ is a mobile $G$-invariant linear system.
Hence, if $a>1$, we can replace $\MMM$ with $|A|$ and apply the modifications as above. Again the process terminates because terminal Fano 3-folds are bounded. Thus, we may assume that  $a=1$. Then $-K_X\sim q S$, so $-K_X$ is a Cartier divisor, so we obtain case \ref{prop:MMP-3}.
\end{proof}

\subsection{The Klein group and its subgroups}
\label{subsection:Klein}

Let $G=\PSL_2(\FF_7)$. In this section, we present some results about the group $G$ and its subgroups. First, we recall that $\mathrm{Aut}(G)\simeq\mathrm{PGL}_2(\FF_7)$ and 
$$
\mathrm{Out}(G):=\mathrm{Aut}(G)/\mathrm{Inn}(G)\simeq\mumu_2, 
$$
where $\operatorname{Inn}(G)$ is the subgroup in $\mathrm{Aut}(G)$ that consists of all inner automorphisms of the group $G$. 

\begin{lemma}
\label{lemma:Klein-subgroups}
Conjugacy classes of non-trivial subgroups of $G$ are described as follows:
\begin{center}\renewcommand{\arraystretch}{1.5}
\begin{tabular}{|c||c|c|c|c|c|c|c|c|c|c|c|c|c|}
\hline
$\CCC$&$\CCC_2$&$\CCC_3$&$\CCC_4$
&
$\CCC_7$&$\CCC_4'$&$\CCC_4''$
&
$\CCC_6$&$\CCC_8$&$\CCC_{12}$&$\CCC_{12}'$
&$\CCC_{24}$&$\CCC_{24}'$&$\CCC_{21}$
\\\hline
$H\in \CCC$ &$\mumu_2$&$\mumu_3$&$\mumu_4$&$\mumu_7$&
$\mumu_2\times\mumu_2$&$\mumu_2\times\mumu_2$&$\mathfrak{S}_3$
&$\mathfrak{D}_4$&$\mathfrak{A}_4$&$\mathfrak{A}_4$
&$\mathfrak{S}_4$&$\mathfrak{S}_4$
&$\mumu_7 \rtimes\mumu_3$\\
\hline
$[G:H]$ & $84$ & $56$ & $42$ & $24$ & $42$ & $42$ & $28$ & $21$ & $14$ & $14$ & $7$ & $7$ & $8$\\
\hline
\end{tabular}
\end{center}
\end{lemma}

\begin{corollary}
\label{cor:act}
Let $\Sigma$ be a set such that $1<|\Sigma|\leqslant 42$ and $G$ acts transitively on $\Sigma$. Then
$$
|\Sigma|\in \{7,8,14,21,24,28,42\}.
$$
Moreover, if $|\Sigma|\in \{7,8\}$, then the $G$-action on $\Sigma$ is doubly transitive and primitive.
\end{corollary}

\begin{lemma}
\label{lemma:reps}
The group $G$ has $6$ irreducible complex representations, which can be described as follows:
\begin{enumerate}
\item $\mathbb{V}_1$ is the trivial $1$-dimensional representation,
\item $\mathbb{V}_3$ is a $3$-dimensional faithful representation,
\item $\mathbb{V}_3^\prime$ is a $3$-dimensional faithful representation that is a complex conjugate of $\mathbb{V}_3$,
\item $\mathbb{V}_6$ is a real faithful $6$-dimensional  representation,
\item $\mathbb{V}_7$ is a real faithful $7$-dimensional  representation,
\item $\mathbb{V}_8$ is a real faithful $8$-dimensional  representation.
\end{enumerate}
\end{lemma}

\begin{remark}
\label{remark:invariants}
In the notations of Lemma~\ref{lemma:reps}, the ring of invariants in $\CC[\mathbb{V}_3]=\CC[x_1,x_2,x_3]$ is generated by polynomials $\phi_4$, $\phi_6$, $\phi_{14}$, $\phi_{21}$ of degree $4$, $6$, $14$, $21$, respectively. In suitable coordinates, one has
\begin{equation}
\label{eq:INVARIANT}
\phi_4= x_1x_2^3+ x_2x_3^3+x_3x_1^3
\end{equation}
and then $\phi_6$ is the Hessian of the polynomial $\phi_4$, $\phi_{14}$ is so-called  bordered Hessian of $\phi_4$ and $\phi_6$, and $\phi_{21}$ is the Jacobian of $\phi_4$, $\phi_6$, $\phi_{14}$. 
The curves $\{\phi_4=0\}$, $\{\phi_6=0\}$, $\{\phi_{14}=0\}$ in  $\PP^2$ are smooth.
\end{remark}

The curve $\{x_1x_2^3+ x_2x_3^3+x_3x_1^3=0\}\subset\PP^2$ is called \emph{plane Klein quartic curve}.

\begin{lemma}
\label{lemma:ext}
Let $G^\prime$ be a non-trivial central extension of the group $G$. Then $G^\prime\simeq \SL_2(\FF_7)$, and $G^\prime$ has $11$ irreducible complex representations, which can be described as follows:
\begin{enumerate}
\item $\mathbb{U}_1$ is the trivial $1$-dimensional representation,
\item $\mathbb{U}_3$ is the lift of the $3$-dimensional representation $\mathbb{V}_3$,
\item $\mathbb{U}_3^\prime$ is the lift of the $3$-dimensional representation $\mathbb{V}_3^\prime$,
\item $\mathbb{U}_4$ is a $3$-dimensional faithful representation,
\item $\mathbb{U}_4^\prime$ is a $3$-dimensional faithful representation that is a complex conjugate of $\mathbb{U}_4$,
\item $\mathbb{U}_6$ is the lift of the $3$-dimensional representation $\mathbb{V}_6$,
\item $\mathbb{U}_6^\prime$ is a faithful $6$-dimensional faithful representation,
\item $\mathbb{U}_6^{\prime\prime}$ is a faithful $6$-dimensional faithful representation that is a complex conjugate of $\mathbb{U}_6^\prime$,
\item $\mathbb{U}_7$ is the lift of the $7$-dimensional representation $\mathbb{V}_7$,
\item $\mathbb{U}_8$ is the lift of the $8$-dimensional representation $\mathbb{V}_8$,
\item $\mathbb{U}_8^\prime$ is a real faithful $8$-dimensional faithful representation,
\end{enumerate}
where $\mathbb{V}_3$, $\mathbb{V}_3^\prime$, $\mathbb{V}_6$, $\mathbb{V}_7$, $\mathbb{V}_8$ are representations of the group $G$ described in Lemma~\ref{lemma:reps}.
\end{lemma}

\begin{corollary}
\label{corollary:Pn-PSL27-complex}
The group $\mathrm{PGL}_2(\mathbb{C})$ has no subgroups isomorphic to $G$.
Up to conjugation, $\mathrm{PGL}_3(\mathbb{C})$ has one subgroup isomorphic to $G$,
and $\mathrm{PGL}_4(\mathbb{C})$ has two subgroups isomorphic to $G$.
\end{corollary}

\begin{corollary}
\label{corollary:Pn-PSL27}
If $n\in\{2,3,4,5\}$, then $\mathrm{PGL}_n(\mathbb{R})$ has no  subgroups isomorphic to $G$.
\end{corollary}

\begin{remark}
\label{remark:invariants:P3}
In the notations of Lemma~\ref{lemma:ext}, the ring of invariants in $\CC[\mathbb{U}_4]=\CC[x_1,x_2,x_3,x_4]$ does not contain polynomials of odd degree, and it also does not contain polynomials of degree $2$. Up to scaling, it has exactly one polynomial $\phi_4$ of degree $4$, and one polynomial $\phi_6$ of degree $6$. The surfaces in $\PP^3$ given by $\phi_4=0$ and $\phi_6=0$ are smooth \cite{Maschke,Edge1947}.
\end{remark}

Let $\Gamma=\mumu_7\rtimes\mumu_3$. Up to conjugation, $G$ contains a unique subgroup isomorphic to $\Gamma$.

\begin{lemma}
\label{lemma:reps:mu7mu3}
The group $\Gamma$ has $5$ irreducible complex representations, which can be described as follows:
\begin{enumerate}
\item $\mathbb{W}_1$ is the trivial $1$-dimensional representation,
\item $\mathbb{W}_1^\prime$ is a non-trivial $1$-dimensional representation,
\item $\mathbb{W}_1^{\prime\prime}$ is a non-trivial $1$-dimensional representation that is a complex conjugate of $\mathbb{W}_1^\prime$,
\item $\mathbb{W}_3$ is a $3$-dimensional faithful representation,
\item $\mathbb{W}_3^\prime$ is a $3$-dimensional faithful representation that is a complex conjugate of $\mathbb{W}_3$.
\end{enumerate}
\end{lemma}

\begin{lemma}
\label{lemma:mu7mu3}
Let $\Gamma^\prime$ be an extension  of the group $\Gamma$ by $\mumu_r$ such that $r\leqslant 3$. Then $\Gamma^\prime$ is a central extension of the group $\Gamma$, and one of the following three cases holds:
\begin{enumerate}
\item $\Gamma^\prime\simeq\mumu_{14}\rtimes\mumu_3\simeq (\mumu_{7}\rtimes\mumu_3)\times \mumu_2$, and its GAP ID is [42,2];
\item $\Gamma^\prime\simeq\mumu_{7}\rtimes\mumu_9\simeq\mumu_3.(\mumu_{7}\rtimes\mumu_3)$, and its GAP ID is [63,1]; 
\item $\Gamma^\prime\simeq(\mumu_{7}\rtimes\mumu_3)\times \mumu_3$, and its GAP ID is [63,3].
\end{enumerate}
Moreover, any faithful complex representation of the group $\Gamma^\prime$ has dimension $\geqslant 3$, and any faithful real representation of the group $\Gamma^\prime$ has dimension $\geqslant 6$. 
\end{lemma}

\section{Actions on  curves, surfaces and 3-folds}
\label{section:PSL2-surfaces}

In this section, we present results about the action of the group $\PSL_2(\FF_7)$ and some its subgroups on real and complex curves, surfaces and 3-folds. 

\subsection{Actions on  curves and surfaces}
\label{subsection:PSL2-surfaces}

We start with actions on curves.

\begin{lemma}
\label{lemma:PSL27-curves}
Let $C$ be a smooth complex irreducible curve such that the group $\mathrm{Aut}(C)$ contains a subgroup $G\simeq\PSL_2(\FF_7)$. Then the $G$-orbits in $C$ are of length $24$, $42$, $56$, $84$ or $168$. Moreover, the curve $C$ is not hyperelliptic. Furthermore, if $\g(C)\leqslant 45$,  then
$$
\g(C)\in\{3,8,10,15,17,19,22,24,29,31,33,36,40,43,45\}.
$$
Furthermore, if $\g(C)=3$, then $C\simeq\{x_1x_2^3+ x_2x_3^3+x_3x_1^3=0\}\subset\PP^2$. 
\end{lemma}

\begin{proof}
If $P$ is a point in $C$, its stablizer in $G$ is a cyclic subgroup, so it follows from Lemma~\ref{lemma:Klein-subgroups} that the $G$-orbit of the point $P$ is of length $24$, $42$, $56$, $84$ or $168$. The curve $C$ is not hyperelliptic, since $G$ is not contained in $\mathrm{PGL}_2(\CC)$. In particular, if $\g(C)=3$, then it is the plane Klein quartic by Remark~\ref{remark:invariants}. The genus bound follows from \cite{Breuer,Jen,lmfdb}.
\end{proof}

\begin{corollary}
\label{corollary:PSL27-real-curves}
Let $C$ be a smooth real geometrically irreducible curve such that $\mathrm{Aut}(C)$ contains a subgroup $G\simeq\PSL_2(\FF_7)$. Then either $\g(C)=8$ or $\g(C)\geqslant 15$.
\end{corollary}

\begin{proof}
Suppose that $\g(C)\ne 8$ and $\g(C)<15$. Then it follows from Lemma~\ref{lemma:PSL27-curves} that $C$ is not hyperelliptic, and either $\g(C)=3$ or $\g(C)=10$. Moreover, if $\g(C)=3$, then $|K_C|$ gives the canonical embedding $C\hookrightarrow\PP^2$, which is $G$-equivariant, so we obtain a faithful $G$-action on $\PP^2$, which contradicts Corollary~\ref{corollary:Pn-PSL27}.

Thus, $\g(C)=10$. Then it follows from \cite{lmfdb} that $C_{\CC}$ has a unique $G$-orbit of length $24$, so the curve $C$ has a unique effective real $G$-invariant divisor $D$ such that $\deg(D)=24$. Then, using the Riemann--Roch theorem, we get $h^0(C,\mathcal{O}_C(D-K_C))+3=h^0(C,\mathcal{O}_C(2K_C-D))$,
so $D-K_C$ is special. Hence, if $h^0\big(C,\mathcal{O}_C(D-K_C)\big)>0$, then $h^0(C,\mathcal{O}_C(D-K_C))\leqslant 4$ by the Clifford theorem. In this case, $|D-K_C|$ is base point free and does not contain $G$-invariant divisors, since $C_{\CC}$ does not contain $G$-orbits of length $<24$ by Lemma~\ref{lemma:PSL27-curves}. Therefore, this linear system gives a $G$-equivariant morphism $C\to\PP^n$ for $n=h^0\big(C,\mathcal{O}_C(D-K_C)\big)-1\leqslant 3$ such that the $G$-action on $\PP^n$ is faithful, which is impossible by Corollary~\ref{corollary:Pn-PSL27}.

Hence, we see that $h^0(C,\mathcal{O}_C(D-K_C))=0$. Then $h^0(C,\mathcal{O}_C(2K_C-D))=3$. Now, arguing as above, we obtain a $G$-equivariant morphism $C\to\PP^2$ such that the $G$-action on $\PP^2$ is faithful, which contradicts Corollary~\ref{corollary:Pn-PSL27}.
\end{proof}

\begin{corollary}
\label{corolarry:Klein-action-on-singular-curves}
Let $C$ be a singular irreducible complex curve with at most locally planar singularities. Suppose that the group $\mathrm{Aut}(C)$ contains a subgroup \mbox{$G\simeq\PSL_2(\FF_7)$}. Then \mbox{$\mathrm{p}_a(C)\geqslant 24$}.
\end{corollary}

\begin{proof}
Let $\pi\colon\widetilde{C}\to C$ be the normalization. Then $\g(\widetilde{C})\geqslant 3$ by Lemma~\ref{lemma:PSL27-curves}, $\pi$ is $G$-equivariant, and
$\mathrm{p}_a(C)\geqslant \g(\widetilde{C})+|\mathrm{Sing}(C)|$. Let $P$ be a singular point of the curve $C$, and let $G_P$ be its stabilizer in $G$. Then $G_P$ faithfully acts on the Zariski tangent space to $C$ at the point $P$, which is two-dimensional by assumption. Since $\mathfrak{A}_4$, $\mathfrak{S}_4$, $\mumu_7 \rtimes\mumu_3$ and $G$ do not have faithful two-dimensional representations, we conclude that $|\mathrm{Sing}(C)|\geqslant 21$ by Lemma~\ref{lemma:Klein-subgroups}, so the assertion follows.
\end{proof}

Now, we move to $\PSL_2(\FF_7)$-actions on surfaces. We start with 

\begin{lemma}
\label{lemma:Klein-action-on-surfaces}
Let $S$ be a smooth complex surface such that its automorphism group contains a subgroup \mbox{$G\simeq\PSL_2(\FF_7)$}. Then $S$ has no $G$-orbits of length less than $21$.
\end{lemma}

\begin{proof}
Let $P$ be a point in $S$, and let $G_P$ be its stabilizer in $G$. Then $G_P$ faithfully acts on the Zariski tangent space to $S$ at the point~$P$. Since $\mathfrak{A}_4$, $\mathfrak{S}_4$, $\mumu_7 \rtimes\mumu_3$ and $G$ do not have faithful two-dimensional representations, the assertion follows from Lemma~\ref{lemma:Klein-subgroups}.
\end{proof}

Now, we recall from \cite{Belousov,IgorVasya} actions of the group $\PSL_2(\FF_7)$ on del Pezzo surfaces.

\begin{lemma}[\cite{Belousov,IgorVasya}]
\label{lemma:PSL2-del-Pezzos}
Let $S$ be a complex del Pezzo surface with Du Val singularities such that $\mathrm{Aut}(S)$ contains a subgroup $G\simeq\PSL_2(\FF_7)$. Then one of the following holds:
\begin{enumerate}
\item
$S\simeq \PP^2$,
\item
$S$ is a smooth del Pezzo surface of degree $2$ representable as a double cover of $\PP^2$ whose branch curve is the plane Klein quartic.
\end{enumerate}
\end{lemma}

\begin{proof}
Taking minimal resolution of singularities of the surface $S$, applying $G$-equivariant Minimal Model Program and using Lemma~\ref{lemma:PSL27-curves}, we obtain a $G$-equivariant birational map $\chi\colon S\dasharrow Y$, where $Y$ is a smooth $G\mathbb{Q}$-Fano variety. Then it follows from \cite{IgorVasya} that either $Y\simeq \PP^2$, or $Y$ is a smooth del Pezzo surface of degree $2$ representable as a double cover of $\PP^2$ whose branch curve is the plane Klein quartic.

We claim that $\chi$ is an isomorphism. Indeed, let $\mathcal{M}_Y$ be the strict transform on $Y$ of the linear system $|-nK_S|$ for $n\gg 0$, and let $\lambda\in\mathbb{Q}_{>0}$ such that $K_Y+\lambda\mathcal{M}_Y\sim_{\mathbb{Q}} 0$.  If $\chi$ is not an isomorphism, the singularities of the log pair $(Y,\lambda\mathcal{M}_Y)$ are not terminal \cite{CheltsovUMN,Icosahedron}, so there is $O\in Y$ such that  $\mathrm{mult}_{O}(\lambda\mathcal{M}_Y)\geqslant 1$.  Recall that $\mathcal{M}_Y$ is $G$-invariant. Let $\Sigma$ be the $G$-orbit of the point $O$, let $M_1$ and $M_2$ be two general curves in $\mathcal{M}_Y$. Then, if $\chi$ is not an isomorphism, we get 
$$
\frac1{\lambda^2} (-K_Y)^2=M_1\cdot M_2\geqslant\sum_{P\in\Sigma}\mathrm{mult}_{P}\big(M_1\cdot M_2\big)\geqslant\sum_{P\in\Sigma}\mathrm{mult}^2_{O}\big(\mathcal{M}_Y\big)\geqslant\sum_{P\in\Sigma}\frac{1}{\lambda^2}=
\frac1{\lambda^2}|\Sigma|,
$$
which gives $|\Sigma|\leqslant (-K_Y)^2\leqslant 9$. But $|\Sigma|\geqslant 21$ by Lemma~\ref{lemma:Klein-action-on-surfaces}.
\end{proof}

\begin{corollary}
\label{corollary:PSL27-ruled-surface}
Let $S$ be a smooth complex surface such that its automorphism group has a subgroup $G\simeq\PSL_2(\FF_7)$, and $|-K_S|$ contains a non-zero $G$-invariant divisor $D$. Then $D$ is not reduced. 
\end{corollary}

\begin{proof}
Suppose that $D$ is reduced. Then, applying $G$-equivariant Minimal Model Program to $S$, we obtain a $G$-equivariant morphism $\pi\colon S\to\overline{S}$ such that either $\overline{S}$ is a smooth del Pezzo surface, or there is a $G$-equivariant morphism $\eta\colon \overline{S}\to C$ such that $C$ is a smooth curve, and general fibers of $\eta$ are conics. Set $\overline{D}=\pi(D)$. Then $\overline{D}$ is a reduced non-zero $G$-invariant divisor in $|-K_{\overline{S}}|$. 

If $\overline{S}$ is a del Pezzo surface, it follows from Lemma~\ref{lemma:PSL2-del-Pezzos} that either $\overline{S}\simeq\PP^2$ or $\overline{S}$ is a del Pezzo surface of degree $2$ described in Lemma~\ref{lemma:PSL2-del-Pezzos}. In both cases, the linear system $|-K_{\overline{S}}|$ does not contain $G$-invariant divisors by Remark~\ref{remark:invariants}.

Thus, we see that there exists a $G$-equivariant morphism $\eta\colon \overline{S}\to C$ such that $C$ is a smooth curve, and the general fiber of $\eta$ is a smooth conic. Write $\overline{D}=\overline{D}_h+\overline{D}_v$, where $\overline{D}_h$ is a $G$-invariant reduced effective Weil divisor such that none of its components is contracted by $\eta$ to a point in $C$, and $\overline{D}_v$ is a $G$-invariant reduced effective Weil divisor such that all its components are contracted by $\eta$ to points in $C$. Then, intersecting $\overline{D}_h$ with a general fiber of $\eta$, we see that either $\overline{D}_h$ is irreducible, or it contains at most two irreducible components. On the other hand, it follows from the adjunction formula that $\mathrm{p}_a(\overline{D}_h)\leqslant 1$, so, in particular, there exists an irreducible component of $\overline{D}_h$ whose arithmetic genus is at most $1$. This implies that either $C\simeq\PP^1$ or $C$ is an elliptic curve. In both cases, the $G$-action on $C$ is trivial by Lemma~\ref{lemma:PSL27-curves}. The same lemma implies that $G$ also acts trivially on general fiber of $\eta$, so $G$ acts trivially on $\overline{S}$, which is absurd. 
\end{proof}

\begin{corollary}
\label{corollary:PSL27-real-rational-surface}
Let $S$ be a geometrically irreducible and geometrically rational real surface. Then $\mathrm{Aut}(S)$ does not contain subgroups isomorphic to $\PSL_2(\FF_7)$.
\end{corollary}

\begin{proof}
Suppose that $\mathrm{Aut}(S)$ contains a subgroup $G$ such that $G\simeq\PSL_2(\FF_7)$. Taking a minimal resolution of singularities, we may assume that $S$ is smooth. Applying equivariant Minimal Model Program to $S$ and Lemma~\ref{lemma:PSL27-curves}, we may further assume that $S$ is a $G\mathbb{Q}$-Fano variety. 

By Lemma~\ref{lemma:PSL2-del-Pezzos}, either $S_{\mathbb{C}}\simeq\PP^2$ or $S_{\CC}$ is a del Pezzo surface of degree $2$. In the former case, we also have $S\simeq\PP^2$, because $\PP^2$ has no non-trivial real forms. This contradicts Corollary~\ref{corollary:Pn-PSL27}. In the latter case, the anticanonical linear system $|-K_S|$ gives a $G$-equivariant double cover $S\to\PP^2$, which again contradicts Corollary~\ref{corollary:Pn-PSL27}. 
\end{proof}

Now, we move to $\PSL_2(\FF_7)$-actions on surfaces with trivial canonical divisors. We start with

\begin{lemma}
\label{lemma:abelian}
The group $\PSL_2(\FF_7)$ cannot  effectively act on a complex abelian surface.
\end{lemma}

\begin{proof}
Let $S$ be an abelian surface. Then its automorphism group is a semi-direct product:
$$
\Aut(A)= \operatorname{Transl} (A) \rtimes  \Aut(A,0),
$$
where $\operatorname{Transl}(A)$   is the group of translations, and $\Aut(A,O)$ is the subgroup of automorphisms that fix a point $O\in A$. Suppose that $\Aut(A)$ contains a subgroup $G\simeq\PSL_2(\FF_7)$. Since the group $\operatorname{Transl}(A)$ abelian, we may assume that $G\subset\Aut(A,O)$. Then $G$ faithfully acts on $T_{A,O}$, which contradicts Lemma~\ref{lemma:reps}.
\end{proof}

The group $\PSL_2(\FF_7)$ can faithfully act on smooth complex K3 surfaces. For instance, it follows from Remark~\ref{remark:invariants} that the quartic $\{x_1x_2^3+x_2x_3^3+x_3x_1^3+x_4^4=0\}\subset\PP^3$ admits a faithful action of $\PSL_2(\FF_7)$. Another smooth quartic surface acted by $\PSL_2(\FF_7)$ is described in  Remark~\ref{remark:invariants:P3}. 

\begin{lemma}[{\cite{Mukai,Xiao}}]
\label{lemma:K3}
Let $S$ be a smooth complex K3 surface such that $\mathrm{Aut}(S)$ contains a subgroup $G\simeq\PSL_2(\FF_7)$. Then $\mathrm{Pic}(S)^G\simeq\mathbb{Z}$.
\end{lemma}

\begin{corollary}
\label{corollary:K3}
Let $S$ be a complex K3 surface with at most Du Val singularities such that the group $\mathrm{Aut}(S)$ contains a subgroup $G\simeq\PSL_2(\FF_7)$. Then $S$ is smooth.
\end{corollary}

\begin{proof}
Apply Lemma~\ref{lemma:K3} to the minimal resolution of singularities of the surface $S$.    
\end{proof}

As we were informed by Alex Degtyarev, real K3 surfaces with effective $\PSL_2(\FF_7)$-action exists, but they are more tricky to construct \cite{Leningrad}.

\subsection{Actions on 3-folds (local results)}
\label{subsection:PSL27-3-folds}
Now, we present several local results about actions on complex and real 3-folds of the group $\PSL_2(\FF_7)$ and some its subgroups. 

Recall from \cite[Definition~3.1]{Janos-Nick} that a 3-dimensional singularity $(P\in X)$ is said to be 
\emph{pseudo-terminal}  if it is canonical and its canonical index-$1$ cover has only cDV  (not necessarily isolated) singularities. Equivalently, $(P\in X)$ is pseudo-terminal if it is $\QQ$-Gorenstein and for some (or equivalently, every) resolution $f\colon \widetilde{X}\to  X$ whose exceptional prime divisors are $E_1,\dots,E_n$, we have 
$$
K_{\widetilde{X}}\sim_{\mathbb{Q}}f^*(K_X)+\sum_{i=1}^n a_i E_i,
$$
where $a_i\geqslant 0$ for all $i$ and $a_j>0$ if  $E_i$ is a point \cite[Proposition~3.7]{Reid:MM}. For the classification of pseudo-terminal singularities, we refer to \cite{Hayakawa-Takeuchi-1987}.

\begin{lemma}
\label{lemma:fixed-point:21}
Let $X$ be a complex 3-fold  with pseudo-terminal singularities such that $\mathrm{Aut}(X)$ contains a subgroup $\Gamma\simeq\mumu_7\rtimes \mumu_3$.
Suppose that $\Gamma$ fixes a singular point $P\in X$. Then $X$ has a cyclic quotient singularity of type $\frac{1}{2}(1,1,1)$ at~$P$. 
\end{lemma}

\begin{proof}
We can regard $X$ as a small neighborhood of~$P$. First, we assume that $X$ has Gorenstein singularity at~$P$. Then $X$ has a $\mathrm{cDV}$ hypersurface singularity at $P$, so we may assume that there exists an equivariant embedding $X\subset \CC^4=T_{X,P}$, where the action of $\Gamma$ on $T_{X,P}$ is linear. Then, by Lemma~\ref{lemma:reps:mu7mu3}, we have a decomposition  $T_{P,X}=T_1\oplus T_3$, where  $T_1$ and $T_3$ are irreducible representations of the group $\Gamma$ of dimensions $1$ and $3$, respectively. Since the action of $\Gamma$ on $T_3$ has no invariants of degree $1$, $2$ and $3$, the local equation of $X$ has the form
$$
\lambda_2 x_1^2+\lambda_3 x_1^3 +(\text{terms of degree $\geqslant 4$})=0.
$$
This contradicts the classification of  $\mathrm{cDV}$-singularities \cite{Reid:YPG}.

Thus, we see that $(P\in X)$ is a non-Gorenstein singularity of index $r\geqslant 2$.  Let $\pi\colon X^\prime\to X$ be the canonical index-$1$ cover \cite[Section~3.5]{Reid:YPG}, and let $P^\prime=\pi^{-1}(P)$. Then $X=X^\prime/\mumu_r$, where  the action of $\mumu_r$ on $X^\prime\setminus P^\prime$ is free. Moreover, there exists a natural exact sequence
\begin{equation}
\label{eq:seq:Gmu}
1\longrightarrow \mumu_r \longrightarrow \Gamma^\prime \overset{\gamma}\longrightarrow \Gamma \longrightarrow 1,
\end{equation}
where $\Gamma^\prime$ is a subgroup in $\Aut(X^\prime)$ that fixes $P^\prime$ and acts faithfully on $T_{X^\prime,P^\prime}$. The group $\Gamma$ acts by conjugation on $\mumu_r$ in \eqref{eq:seq:Gmu}, which induces a homomorphism $\theta\colon\Gamma\to\mathrm{Aut}(\mumu_r)\simeq(\mathbb{Z}/r\mathbb{Z})^\ast$. Hence the commutator subgroup $[\Gamma,\Gamma]=\mumu_7$ lies in the kernel of $\theta$.
If $\Gamma\subset\mathrm{ker}(\theta)$, the extension \eqref{eq:seq:Gmu} is central. Obviously, this holds if $r\leqslant 3$ because the size of $\Gamma$ is odd.

First, we suppose that $X^\prime$ is smooth at~$P^\prime$, so $(P\in X)$ is a terminal cyclic quotient singularity. We have to show that $r=2$. Suppose that $r\geqslant 3$. By \cite{Reid:YPG}, we may assume that generator $z\in\mumu_r$ acts on $T_{X^\prime,P^\prime}\simeq\mathbb{C}^3$ diagonally as $\mathrm{diag}\big(\zeta_r,\zeta_r^b,\zeta_r^{-b}\big)$, where $\zeta_r$ is a primitive $r$-th root of unity, and $b$ is an integer coprime to $r$. If $r=3$, then the size of $\Gamma^\prime$ is odd, hence the extension \eqref{eq:seq:Gmu} is central, and $T_{X^\prime,P^\prime}$ is an irreducible representation of the group $\Gamma^\prime$ by Lemma~\ref{lemma:mu7mu3}, which implies that $\mumu_r$ acts on $T_{X^\prime,P^\prime}$ by scalar matrices, which is a contradiction. Thus, $r\geqslant 4$. If $b=1\mod r$, we have 
$$
T_{X^\prime,P^\prime}=T^\prime\oplus T^{\prime\prime},
$$
where $T^\prime$ is the two-dinensional eigenspace of $z$ with eigenvalue $\zeta_r$, and $T^{\prime\prime}$ is the one-dinensional eigenspace  with eigenvalue $\zeta_r^{-1}$. In this case, $T^\prime$ and $T^{\prime\prime}$ are $\Gamma^\prime$-invariant, and $\mumu_r$ acts on $T^\prime$ by scalar matrices, so the natural homomorphism $\Gamma^\prime\to\mathrm{PGL}(T^\prime)$ factors through $\Gamma$, and the unique subgroup $\mumu_7\subset\Gamma$ is contained in the kernel of the induced homomorphism $\Gamma\to\mathrm{PGL}(T^\prime)$, because $\mathrm{PGL}_2(\mathbb{C})$ does not contain subgroups isomorphic to $\Gamma$. The later implies that the action of $\Gamma^\prime$ on $T^{\prime}$ is diagonalizable, so the action of $\Gamma^\prime$ on $T_{X^\prime,P^\prime}$ is also diagonalizable, which is impossible, since the group $\Gamma^\prime$ is not abelian. Hence, $b\ne 1\mod r$. Similarly, we see that $b\ne -1\mod r$.  

Hence, $z$ acts on $T_{X^\prime,P^\prime}$ diagonally with three 
distinct eigenvalues $\zeta_r$, $\zeta_r^b$ and $\zeta_r^{-b}$. If the extension \eqref{eq:seq:Gmu} is central, each eigenspace of $z$ is $\Gamma^\prime$-invariant, so $\Gamma^\prime$ also acts on $T_{X^\prime,P^\prime}$ diagonally, which is impossible, since $\Gamma$ is not abelian. Hence, we see that $\mathrm{im}(\theta)\simeq\mumu_3$, and $\Gamma^\prime$ cyclically permutes the eigenspaces of $z$. Let $g$ be a lift in $\Gamma^\prime$ of an element of $\Gamma$ that has order $3$. Then $gzg^{-1}=z^k$ for some integer $k$ such that $k$ is coprime to $r$ and $k\ne 1\mod r$. Since $g$ cyclicly permutes eigenspaces of $z$, $gzg^{-1}=z^k$ acts on $T_{X^\prime,P^\prime}$ diagonally as 
$$
\mathrm{diag}\big(\zeta_r^{-b},\zeta_r,\zeta_r^b\big)=\mathrm{diag}\big(\zeta_r^k,\zeta_r^{kb},\zeta_r^{-kb}\big)
$$ 
or as 
$$
\mathrm{diag}\big(\zeta_r^b,\zeta_r^{-b},\zeta_r\big)=\mathrm{diag}\big(\zeta_r^k,\zeta_r^{kb},\zeta_r^{-kb}\big).
$$
In both cases, we compute $b=\pm 1\mod r$. This is a contradiction, since the eigenvalues of $z$ are distinct.

Thus, to complete the proof, we may assume that $X^\prime$ is singular at~$P^\prime$.

First, we suppose that $r\leqslant 3$. Since the size of $\Gamma$ is odd, the extension \eqref{eq:seq:Gmu} is central, and it follows from \cite{Hayakawa-Takeuchi-1987} that $T_{X^\prime,P^\prime}=T_1\oplus T_3$, where $T_1$ and $T_3$ are irreducible representations of the group $\Gamma^\prime$ of dimensions $1$ and $3$, respectively. Thus, $\mumu_r$ acts on $T_3$ by scalar matrices. Since this action is non-trivial, $\mumu_r$ acts trivially on $T_1$ by the classification of terminal singularities \cite{Reid:YPG}. This implies that $r=2$.  In this case, $(P^\prime\in X^\prime)$ is a hypersurface singularity whose equation is invariant \cite{Hayakawa-Takeuchi-1987}. On the other hand, the only $\Gamma^\prime$-invariant of degree $2$ on $T_{X^\prime,P^\prime}$ is a multiple of~$x_1^2$, where $x_1$ is a coordinate on $T_1$. Therefore, the local equation of $X^\prime$ has the form
$$
x_1^2+(\text{terms of degree $\geqslant 3$})=0,
$$ 
so $(P\in X)$ is a singularity of type $\mathrm{cD}/2$ or $\mathrm{cE}/2$. In these cases, the action of $\mumu_2$ on $x_1$ must be non-trivial \cite{Hayakawa-Takeuchi-1987}, which is a contradiction.

To complete the proof, we may assume that $r\geqslant 4$. As above, let $z$ be a generator of $\mumu_r$, and let $\zeta_r$ be a primitive $r$-th root of unity. Using the classification of three-dimensional pseudo-terminal singularities \cite{Hayakawa-Takeuchi-1987}, we may assume that either $r=4$ and $z$ acts on $T_{X^\prime,P^\prime}\simeq\mathbb{C}^4$ diagonally as $\mathrm{diag}\big(\zeta_4,\zeta_4,\zeta_4^2,\zeta_4^3\big)$, or $z$ acts on $T_{X^\prime,P^\prime}$ diagonally as $\mathrm{diag}\big(\zeta_4^a,\, \zeta_4^{-a},\, \zeta_4,\, 1\big)$ for some integer $a$ coprime to $r$. In the former case, we have
$$
T_{X^\prime,P^\prime}=E_1\oplus E_2\oplus E_3,
$$
where $E_1$, $E_2$, $E_3$ are eigenspaces of $z$ with eigenvalues $\zeta_4$, $\zeta_4^2$, $\zeta_4^3$, respectively. Then $E_1$ and $E_2\oplus E_3$ are $\Gamma^\prime$-invariant. Moreover, since the order of $\Gamma$ is odd, both eigenspaces $E_2$ and $E_3$ are also $\Gamma^\prime$-invariant. Arguing as above, we see that we can diagonalize the $\Gamma^\prime$-action on $E_1$, so $\Gamma^\prime$ acts on $T_{X^\prime,P^\prime}$ diagonally, which is impossible, since the action is faithful, and $\Gamma^\prime$ is not abelian. 

Hence, we see that $z$ acts on $T_{X^\prime,P^\prime}$ diagonally as $\mathrm{diag}\big(\zeta_4^a,\, \zeta_4^{-a},\, \zeta_4,\, 1\big)$. If $a=\pm 1\mod r$, we obtain a contradiction exactly as in the previous case. Thus, we have $a\ne\pm 1\mod r$. Then
$$
T_{X^\prime,P^\prime}=E_a\oplus E_{-a}\oplus E_{1}\oplus E_{0},
$$
where $E_a$, $E_{-a}$, $E_1$, $E_0$ are eigenspaces of $z$ with eigenvalues $\zeta_r^a$, $\zeta_r^{-a}$, $\zeta_r$, $1$, respectively. Since $\Gamma^\prime$ is not abelian, there exists $g\in\Gamma^\prime$ that does not preserve the eigenspaces of $z$. Then $g$ non-trivially permutes eigenspaces $E_a$, $E_{-a}$, $E_1$, $E_0$. In particular, $g$ does not commute with $z$, so  $gzg^{-1}=z^k$ for some $k\in\mathbb{Z}$ such that $k\ne 1\mod r$ and $k$ is co-prime to $r$. This shows that $g$ leaves invariant $E_0$, so $g$ cyclically permute eigenspaces $E_a$, $E_{-a}$, $E_1$. Now, arguing as above, we see that $a=\pm 1\mod r$, which is already excluded. The obtained contradiction completes the proof of Lemma~\ref{lemma:fixed-point:21}.
\end{proof}

\begin{corollary}
\label{cor:Klein-action-on-3-folds}
Suppose that $X$ is complex 3-fold with pseudo-terminal singularities such that $\mathrm{Aut}(X)$ contains a subgroup $G\simeq\PSL_2(\FF_7)$. Then $G$ cannot pointwise fix a curve in $X$.
\end{corollary}

\begin{proof}
Suppose that $G$ pointwise fixes a curve $C\subset X$. Let $P$ be a general point of this curve. Then $X$ and $C$ are smooth at $P$, and $G$ acts faithfully on $T_{X,P}$. This representation is reducible, since $G$ preserves $T_{C,P}$. But $G$ has no faithful reducible $3$-dimensional representations by Lemma~\ref{lemma:reps}.
\end{proof}

\begin{corollary}
\label{corollary:fixed-point}
Let $X$ be a real 3-fold with pseudo-terminal singularities such that $\mathrm{Aut}(X)$ contains a subgroup $\Gamma\simeq\mumu_{7}\rtimes\mumu_3$. Then $\Gamma$ does not fix points in $X(\mathbb{R})$.
\end{corollary}

\begin{proof}
Suppose that the group $\Gamma$ fixes a point $P\in X(\mathbb{R})$. Let us use notations introduced in the proof of Lemma~\ref{lemma:fixed-point:21}.
If $X$ is smooth at $P$, then  $T_{X,P}$ is a real three-dimensional faithful representation of $\Gamma$, which is impossible by Lemma~\ref{lemma:reps}. 
Thus, it follows from Lemma~\ref{lemma:fixed-point:21} that $(P\in X)$ is a cyclic quotient singularity of type $\frac{1}{2}(1,1,1)$, so $X=X^\prime/\mumu_2$, and we have 
$$
\Gamma^\prime\simeq \mumu_{14}\rtimes\mumu_3\simeq(\mumu_{7}\rtimes\mumu_3)\times \mumu_2
$$ 
by Lemma~\ref{lemma:mu7mu3}.
On the other hand, $\Gamma^\prime$ acts faithfully on $T_{X^\prime,P^\prime}$, and this representation is real and three-dimensional,  which is impossible by Lemma~\ref{lemma:mu7mu3}.
\end{proof}

\begin{corollary}
\label{corollary:fixed-point:21}
Let $X$ be a complex 3-fold  with pseudo-terminal singularities such that $\mathrm{Aut}(X)$ contains a subgroup $G=\PSL_2(\FF_7)$, and let $S$ be a $G$-invariant surface in $X$. Suppose that $S$ is a $\mathbb{Q}$-Cartier divisor on $X$, and $G$ fixes a point $P\in S$. Then $(X,S)$ is not log canonical at~$P$.
\end{corollary}

\begin{proof}
By Lemma~\ref{lemma:fixed-point:21}, we know that either $X$ is smooth at $P$, or $X$ has a cyclic quotient singularity of type $\frac{1}{2}(1,1,1)$ at~$P$. Let $f\colon\widetilde{X}\to X$ be the blow up of the point $P$, let $E$ be the exceptional divisor of this blow up, let $\widetilde{S}$ be the strict transform on $\widetilde{X}$ of the surface $S$, and let $m$ be a positive rational number such that $\widetilde{S}\sim_{\mathbb{Q}} f^*(S)-mE$.
Then $f$ is $G$-equivariant, $G$ acts faithfully on $E\simeq\PP^2$,
and $\widetilde{S}\vert_{E}$ is a curve of degree at least $4$, because $E$ contains no $G$-invariant lines, conics and cubics by Remark~\ref{remark:invariants}. If $X$ is smooth at $P$, then $m\geqslant 4$, so $(X,S)$ is not log canonical at~$P$. Similarly, if $X$ is singular at $P$, then $m\geqslant 2$ and
$$
K_{\widetilde{X}}+\widetilde{S}\sim_{\mathbb{Q}} f^*(K_X+S)+\Big(\frac{1}{2}-m\Big)E,
$$
which implies that the  pair $(X,S)$ is not log canonical at~$P$.
\end{proof}

\begin{lemma}
\label{lemma:fixed-point:S4}
Let $X$ be a complex 3-fold  with terminal singularities such that $\mathrm{Aut}(X)$ contains a subgroup $\Gamma\simeq\mathfrak{S}_4$. Suppose that $\Gamma$ fixes  a singular point $P\in X$ that has index $r\in\{2,3\}$. Then one of the following cases holds:
\begin{enumerate}
\item $r=2$ and $(P\in X)$ is a cyclic quotient singularity of type $\frac{1}{2}(1,1,1)$;
\item $r=2$ and $(P\in X)$ is a moderate singularity (see Definition \ref{def:simple:HQ});
\item $r=3$ and $(P\in X)$ is not a cyclic quotient singularity.
\end{enumerate}

\end{lemma}

\begin{proof}
We can regard $X$ as a small neighborhood of~$P$.  Let $\pi: X^\prime\to X$ be a canonical index-$1$ cover, and let $P^\prime:=\pi^{-1}(P)$. Then $X=X^\prime/\mumu_r$, and $\mumu_r$ acts freely on $X^\prime\setminus P^\prime$. Let $\Gamma^\prime$ be the natural lifting of $\Gamma$ to $\Aut(X^\prime)$. Then $\Gamma^\prime$ acts faithfully on $T_{X^\prime,P^\prime}$, and there is a natural exact sequence
$$
1\longrightarrow \mumu_r \longrightarrow \Gamma^\prime \overset{\gamma}\longrightarrow \Gamma \longrightarrow 1.
$$
Let $z\in \mumu_r$ be an element of order $r$, and let $\zeta_r$ be a primitive $r$-th root of unity.

Suppose that $X$ has a cyclic quotient singularity of index $3$ at~$P$. Then $T_{X^\prime,P^\prime}= T^\prime\oplus T''$, where $T^\prime$ and $T^{\prime\prime}$ are eigenspaces of $z$ with the eigenvalues $\zeta_3$ and $\zeta_3^{-1}$, respectively. We may further assume that $\dim(T^\prime)\geqslant \dim (T^{\prime\prime})$. Then $T^\prime$ and $T^{\prime\prime}$ are $\Gamma^\prime$-invariant, and it follows from the classification of terminal singularities \cite{Reid:YPG} that $\dim(T^\prime)=2$ and $\dim (T^{\prime\prime})=1$. Let $K$ be the kernel of the representation $\Gamma^\prime\to \GL(T^{\prime\prime})$. Then $K$ acts on $T^\prime$ faithfully and $\Gamma^\prime=\mumu_3\times K$, which implies that $K\simeq \mathfrak{S}_4$.
This is impossible, since $\mathfrak{S}_4$ has no faithful representations of dimension $2$.

Now we suppose that $r=2$, but $(P\in X)$ is not a cyclic quotient singularity. Then, as above, we have a decomposition
$T_{X^\prime,P^\prime}= T^\prime\oplus T^{\prime\prime}$, where $T^\prime$ and $T^{\prime\prime}$ are eigenspaces of $z$ corresponding to the eigenvalues $1$ and $-1$, respectively. Then it  follows from the classification of terminal singularities that $\dim(T^\prime)=1$ and $\dim(T^{\prime\prime})=3$, so both $T^\prime$ and $T^{\prime\prime}$ are $\Gamma^\prime$-invariant. Let $K$ be the kernel of the representation $\Gamma^\prime\to \GL(T^{\prime\prime})$. Then $K$ acts on $T^\prime$ faithfully, hence it is cyclic.  But the restriction homomorphism $\gamma|_{K}: K\to \gamma(K)$ is an isomorphism, since $K\cap\mumu_2=\{1\}$. Thus, $\gamma(K)$ is a normal cyclic subgroup of $\Gamma$, so $K$ is trivial. Thus, the representation $\Gamma^\prime \to \GL(T^{\prime\prime})$ is faithful. 

We claim that this representation is irreducible. Indeed, suppose that it reducible:  $T^{\prime\prime}=T_1\oplus T_2$, where $T_1$ and $T_2$ are $\Gamma^\prime$-invariant subspaces of dimension $2$ and $1$, respectively. As above, let $K_1$ be the kernel of the representation $\Gamma^\prime\to \GL(T_1)$. Then $K_1\cap \mumu_2=\{1\}$  and $\Gamma^\prime/K_1$ is a cyclic group.
Hence $K_1$ contains the derived subgroup $[\Gamma^\prime,\Gamma^\prime]$,
so $\gamma(K_1)$ contains $[\Gamma^\prime,\Gamma^\prime]\simeq \mathfrak{A}_4$.
On the other hand, $K_1\simeq \gamma(K_1)$ acts on $T_2$ faithfully.
Since $\mathfrak{A}_4$ has no faithful two-dimensional representations,
we get a contradiction. Thus, the representation $\Gamma^\prime \hookrightarrow \GL(T^{\prime\prime})$ is irreducible.

We may assume that there is an equivariant embedding $X\hookrightarrow T_{X^\prime,P^\prime}\simeq\mathbb{C}^4$, and $X$ is given by 
$$
\phi(x_1,x_2,x_3,x_4)=0,
$$ 
where $x_1,x_2,x_3$ are coordinates on $T^{\prime\prime}$ and $x_4$  is a coordinate on $T^\prime$.
Let $\phi_2$ be the quadratic term of the~function $\phi$. Then it follows from the classification of terminal singularities that $\phi_2(x_1,x_2,x_3,0)$ is not zero. Thus, since $T^{\prime\prime}$  is an irreducible representation of $\Gamma^\prime$ and $\phi_2$ is a semi-invariant of $\Gamma^\prime$, we see that the quadratic form $\phi_2(x_1,x_2,x_3,0)$ has rank $3$, so $X$ has a moderate singularity at~$P$.
\end{proof}

\subsection{Actions on Gorenstein Fano 3-folds}
\label{subsection:Fanos-Gorenstein}

Let $X$ be a complex Fano 3-fold  with Gorenstein canonical singularities. Then $\iota(X)\in\{1,2,3,4\}$. Moreover, if $\iota(X)=4$ then $X\simeq\PP^3$. Furthermore, if $\iota(X)=3$, then $X$ is a quadric hypersurface in $\PP^4$.

\begin{remark}
\label{remark:P3-quadric}
Recall from Corollary~\ref{corollary:Pn-PSL27-complex} that $\mathrm{Aut}(\PP^3)\simeq\mathrm{PGL}_4(\CC)$ contains two subgroups isomorphic to $\PSL_2(\FF_7)$ up to conjugation. Similarly, it follows from  Lemmas \ref{lemma:reps} and \ref{lemma:ext} that an irreducible quadric in $\PP^4$ does not admit a faithful action of the group $\PSL_2(\FF_7)$.
\end{remark}

\begin{lemma}
\label{lemma:SB}
Suppose that $X$ is defined over $\mathbb{R}$, and $\mathrm{Aut}(X)$ has a subgroup $G\simeq\PSL_2(\FF_7)$. Then $\iota(X)\leqslant 2$,
i.e., we have $X_{\CC}\not\simeq\PP^3$ and $X_{\CC}$ is not a quadric in $\PP^4$.
\end{lemma}

\begin{proof}
If $X_{\CC}$ is a quadric 3-fold in $\PP^4$, then $X$ is also a quadric in $\PP^4$ by \cite[Corollary 2.3]{pointless}, which is ruled out by Remark~\ref{remark:P3-quadric}. 

Suppose that $X_{\CC}\simeq\PP^3$. Then it follows from Corollary~\ref{corollary:Pn-PSL27} that $X\not\simeq\PP^3$. Moreover, it follows from \cite[Lemma 3.7]{CheltsovShramovKlein} or \cite{Edge1947} that $X_{\CC}$ contains a unique $G$-invariant smooth irreducible curve $C$ of degree $6$ and genus $3$, and the group $G$ act faithfully on this curve.  Since $C$ is unique, it must be defined over $\mathbb{R}$, which contradicts Corollary~\ref{corollary:PSL27-real-curves}. 

We can also obtain a contradiction as follows. The 3-fold $X_{\CC}$ has a unique $G$-orbit of length $8$, which must be defined over $\RR$. The net of quadrics in $X_{\CC}$ that pass through this orbit is defined over $\mathbb{R}$, and gives a $G$-equivariant rational map $X\dasharrow\PP^2$, whose general fiber is an elliptic curve. This contradicts Corollaries \ref{corollary:Pn-PSL27} and \ref{corollary:PSL27-real-curves}.
\end{proof}

If $\iota(X)=2$, $X$ is called a \emph{del Pezzo $3$-fold}, and  $\dd(X)=\frac{1}{8}(-K_X)^3$ is called the \emph{degree} of~$X$. 

\begin{lemma}
\label{lemma:DP}
Suppose that $\iota(X)=2$, the singularities of $X$ are terminal, and $\mathrm{Aut}(X)$ contains a subgroup $G\simeq\PSL_2(\FF_7)$.
Then one of the following holds:
\begin{enumerate}
\item $\dd(X)=1$, and $X$ is isomorphic to the hypersurface 
$$
\big\{z^2=y^3+\lambda y \phi_4(x_1,x_2,x_3)+ \phi_6(x_1,x_2,x_3)\}\subset\PP(1_{x_1},1_{x_2},1_{x_3},2_y,3_z),
$$
where $\phi_4$ and $\phi_6$ are polynomials described in Remark~\ref{remark:invariants}, and $\lambda\in\mathbb{C}$;

\item $\dd(X)=2$, and $X$ is isomorphic to the smooth hypersurface 
$$
\big\{z^2=x_0^4+\phi_4(x_1,x_2,x_3)\big\}\subset\PP(1_{x_1},1_{x_2},1_{x_3},1_{x_4},2_z),
$$
where $\phi_4$ is the polynomial described in Remark~\ref{remark:invariants};
\item $\dd(X)=2$, and $X$ is isomorphic to the smooth hypersurface 
$$
\big\{z^2=\phi_4(x_1,x_2,x_3,x_4)\big\}\subset\PP(1_{x_1},1_{x_2},1_{x_3},1_{x_4},2_z),
$$
where $\phi_4$ is the polynomial described in Remark~\ref{remark:invariants:P3};
\item
$\dd(X)=6$, and $X$  is isomorphic to the smooth hypersurface of degree $(1,1)$ in $\PP^2\times\PP^2$;

\item
$\dd(X)=7$, and $X$  is isomorphic to the blow up of $\mathbb{P}^3$ in a point.
\end{enumerate}
\end{lemma}

\begin{proof}
Let $A$ be a Cartier divisor on $X$ such that $-K_X\sim 2A$. Let us use classification of del Pezzo 3-folds with terminal singularities (see e.g. \cite{Shin1989}   or \cite{SashaYura-del-Pezzo}). In particular, we have $\dd(X)\leqslant 7$.

If $\dd(X)=1$, then $X$ is a hypersurface of degree $6$ in $\PP(1,1,1,2,3)$, and it can be written in the desired form by Remark~\ref{remark:invariants}. Similarly, if $\dd(X)=2$, then $X$ is a hypersurface of degree $4$ in $\PP(1,1,1,1,2)$, and its equation can be obtained from Remarks~\ref{remark:invariants} and
\ref{remark:invariants:P3}.

Suppose that $\dd(X)=3$. Then $X$ is a cubic hypersurface in $\PP^4$, the $G$-action lifts to $\mathbb{P}^4$, and it follows from Lemma~\ref{lemma:ext} that the action of $G$ on $\mathbb{P}^4$ is induced by a 5-dimensional representation of the group $G$. Moreover, it follows from Lemma~\ref{lemma:ext} that there exists a $G$-invariant hyperplane section $S=X\cap \PP^3$, and $G$ acts faithfully on $S$. This implies that $S$ is not ruled  (covered by lines), so $S$ is a cubic surface with at worst Du Val singularities by \cite[Theorem 8.1.11]{Dolgachev-ClassicalAlgGeom}, which is impossible by Lemma~\ref{lemma:PSL2-del-Pezzos}.

Now, we consider the case $\dd(X)=4$.  Then $X=Q_1\cap Q_2\subset \PP^5$, where $Q_1$ and $Q_2$ are quadric hypersurfaces. If $G$ leaves invariant a hyperplane section of $X$,  we can argue as above to obtain a contradiction. Hence, $\PP^5$ does not contain $G$-invariant points, and it does not contain $G$-invariant lines. On the other hand, the action on the pencil $\PPP$ of quadrics generated by $Q_1$ and $Q_2$ is trivial, and singular locus of a degenerate quadric $Q\in \PPP$ is either a point or a line. Hence, $\PPP$ contains no degenerate quadrics, a contradiction.

Consider the case  $\dd(X)=5$. If $X$ is smooth, then $\Aut(X)\simeq \PSL_2(\CC)$ and this group does not contain $G=\PSL_2(\FF_7)$ by Corollary~\ref{corollary:Pn-PSL27-complex}, a contradiction. Hence, $X$ is singular. Then $X$ has at most three singular points by \cite[Corollary 8.7]{P:GFano1}, which contradicts Lemma~\ref{lemma:fixed-point:21}.

If $\dd(X)=6$, then either $X\simeq \PP^1\times \PP^1\times \PP^1$ or $X$ is a divisor of degree $(1,1)$ in $\PP^2\times \PP^2$. Clearly, the former case is impossible. In the latter case, $X$ has at most one singular point, so it is smooth by Lemma~\ref{lemma:fixed-point:21}. Finally, if $\dd(X)=7$, then $X$ can be obtained by blowing up $\PP^3$ at a point.
\end{proof}

\begin{corollary}
\label{corollary:DP}
Suppose that $X$ is defined over $\mathbb{R}$, $\mathrm{Aut}(X)$ contains a subgroup $G\simeq\PSL_2(\FF_7)$, the singularities of $X_{\CC}$ are terminal, and $\iota(X)=2$. Then $X$ is the pointless real form of the divisor 
\begin{equation}
\label{equation:flag}
\{x_1y_1+x_2y_2+x_3y_3=0\}
\subset\PP^2_{x_1,x_2,x_3}\times\PP^2_{y_1,y_2,y_3}.
\end{equation}
\end{corollary}

\begin{proof}
By assumption, $-K_{X_\CC}\sim 2A$ for an ample Cartier divisor $A\in\mathrm{Pic}(X_{\CC})$, which may not be defined over $\mathbb{R}$. By Lemma~\ref{lemma:DP}, we know all possibilities for $X_{\CC}$. If $\dd(X_{\CC})=1$, the base locus of $|A|$ is a single point that is defined over $\mathbb{R}$ and fixed by $G$, which contradicts Lemma~\ref{lemma:fixed-point:21}. If $\dd(X_\CC)=2$ and $X(\mathbb{R})\ne\varnothing$, then $A$ is also defined over $\mathbb{R}$ (see \cite[Ch.~4,~Proposition~12]{C58}), so $|A|$ gives a $G$-equivariant double cover $X\to\PP^3$, which is impossible by Corollary~\ref{corollary:Pn-PSL27}. Moreover, if $\dd(X_\CC)=2$ and $A$ is not defined over $\mathbb{R}$, then $|A|$ is a twisted linear system in the sense of \cite{KollarSB}, so it gives a double cover $X\to Y$ defined over $\mathbb{R}$ such that $Y_{\CC}\simeq\PP^3$, which is impossible by Lemma~\ref{lemma:SB}. If $\dd(X_\CC)=7$, then $X$ is a blow up of $\mathbb{P}^3$ in a point, and this blowup must be $G$-equivariant, which contradicts Corollary~\ref{corollary:fixed-point}. Thus, $X$ is a form of the 3-fold \eqref{equation:flag}, which must be non-trivial by Corollary~\ref{corollary:Pn-PSL27}. 

Recall from \cite[Proposition 8.1]{RonanSusanna} that the 3-fold \eqref{equation:flag} has two non-trivial real forms, which can be described as follows. Let $Q^{\pm}$ be the real smooth quadric 3-fold 
$$
\big\{x^2+y^2+z^2+t^2\pm w^2=0\big\}\subset\mathbb{P}^4_{x,y,z,t,w},
$$
and let $S$ be its hyperplane section that is cut out by $w=0$. Then $S_{\mathbb{C}}$ contains complex conjugated lines $L_1=\{w=0,x=iy,z=it\}$ and
$L_2=\{w=0,x=-iy,z=-it\}$, and the curve $L_1+L_2$ is defined over $\mathbb{R}$.
Let $\alpha^{\pm}\colon\widetilde{Q}^{\pm}\to Q^{\pm}$ be the blowup of the curve $L_1+L_2$. Then we have the following real Sarkisov link:
$$
\xymatrix@R=1em{
&\widetilde{Q}^{\pm}\ar@{->}[ld]_{\alpha^{\pm}}\ar@{->}[rd]^{\beta^{\pm}}&&\\%
Q^{\pm}&& X^{\pm}}
$$
where $X^{\pm}$ is a non-trivial real form of \eqref{equation:flag}, $\beta^{\pm}$ is a birational morphism that contracts the strict transform of the surface $S$ to a smooth curve in $X^{\pm}$. By construction, the 3-fold $X^{+}$ is pointless, so it follows from \cite[Proposition 8.1]{RonanSusanna} that $\mathrm{Aut}(X^+)\simeq\mathrm{PSU}_3(\mathbb{C})\rtimes\mumu_2$. On the other hand, the real locus of the 3-fold $X^{-}$ is a 3-dimensional sphere, and $\mathrm{Aut}(X^-)\simeq\mathrm{PSU}(1,2)\rtimes\mumu_2$ by \cite[Proposition 8.1]{RonanSusanna}. Since $\mathrm{Aut}(X^-)$ has no subgroups isomorphic to $G$, we see that $X\simeq X^+$ as claimed.
\end{proof}

\begin{remark}
\label{remark:del-Pezzo-3-folds}
Let us use assumptions of Lemma~\ref{lemma:DP}. If $\dd(X)=6$, then $X$ is rational over $\mathbb{C}$. If $\dd(X)=2$, $X$ is irrational \cite{Voisin}. If $\dd(X)=1$ and $X$ is smooth, $X$ is irrational \cite{Grinenko1,Grinenko2}. 
\end{remark}

Now, we study properties of the linear system $|-K_X|$ in the case when $\mathrm{Aut}(X)$ has a subgroup isomorphic to $\PSL_2(\FF_7)$. We start with the following result, which can be derived from \cite{PriskaIvo}.

\begin{lemma}
\label{lemma:canFano:Bs}
Suppose that $\mathrm{Aut}(X)$ contains a subgroup $G=\PSL_2(\FF_7)$. Then $\mathrm{Bs}(|-K_X|)=\varnothing$.
\end{lemma}

\begin{proof}
Suppose that  $\Bs(|-K_X|)\neq \varnothing$. Then it follows from \cite{Shin1989} that one of the following holds:
\begin{itemize}
\item either $\Bs(|-K_X|)$ is a smooth rational curve lying in the smooth part of $X$;
\item or $\Bs(|-K_X|)$ is a single point, $X$ has a hypersurface $\mathrm{cDV}$ singularity at this point, and a general surface in $|-K_{X}|$ has an ordinary double singularity at this point.
\end{itemize}
In the first case, the curve is pointwise fixed by $G$ by Lemma~\ref{lemma:PSL27-curves}, which contradicts Corollary~\ref{cor:Klein-action-on-3-folds}.
In the second case, we obtain a contradiction arguing as in  the proof of Lemma~\ref{lemma:fixed-point:21}.
\end{proof}

Now, we describe hyperelliptic Fano 3-folds \cite{HT} that admit a faithful $\PSL_2(\FF_7)$-action.

\begin{lemma}[{cf. \cite{HT}}]
\label{lemma:hyperelliptic}
Suppose that  $-K_X$ is not very ample, and $\mathrm{Aut}(X)$ contains a subgroup $G=\PSL_2(\FF_7)$. Then $X$ is isomorphic to one of the following 3-folds:
\begin{enumerate}
\item \label{prop:canFano:hyp3}
a hypersurface of degree $6$ in $\PP(1,1,1,1,3)$ (a sextic double solid);
\item \label{prop:canFano:hyp2}
a hypersurface of degree $6$ in $\PP(1,1,1,2,3)$ (a double Veronese cone);
\item \label{prop:canFano:hyp1}
$S\times \PP^1$, where $S$ is the smooth del Pezzo surface of degree $2$ described in Lemma~\ref{lemma:PSL2-del-Pezzos}.
\end{enumerate}
Moreover, in the case \ref{prop:canFano:hyp1}, the invariant Picard number    $\rho(X)^G$ equals $2$ and the invariant Mori cone $\overline{\operatorname{NE}}(X)^G$
is generated by the curves in the fibers of projections $X\to \PP^1$ and $X\to S$.
\end{lemma}

\begin{proof}
Let $g=\frac12 (-K_{X})^3+1$. It follows from Lemma~\ref{lemma:canFano:Bs} and \cite{HT} that $\dim(|-K_X|)=g+1$, and the linear system $|-K_X|$ gives a $G$-equivariant double cover morphism $\phi\colon X \to Y$ such that $Y$  is a 3-fold of degree $g-1$ in $\PP^{g+1}$. Let $B\subset Y$ be the branch divisor.
Then $B$ is $G$-invariant, and it follows from the Hurwitz formula that
\begin{equation}
\label{eq:Hurwitz}
\textstyle
K_X=\phi^*\left(K_Y+\frac 12 B\right)=\phi^*\big(\OOO_Y(1)\big).
\end{equation}
If $g=2$, then  $Y=\PP^3$ and $B$ is a surface of degree $6$, so we get case \ref{prop:canFano:hyp3}. If $g=3$, then  $Y$ is a quadric in $\PP^4$, which is impossible by Remark~\ref{remark:P3-quadric}. Thus, we see that $g\geqslant 4$. Then, according to the classification of varieties of minimal degree \cite{EisenbudHarris:VMD}, we have one of the following cases:
\begin{enumerate}
\renewcommand\labelenumi{\rm (\alph{enumi})}
\renewcommand\theenumi{\rm (\alph{enumi})}
\item
\label{cases:hyp:1}
$Y$ is a cone in $\PP^6$ over the Veronese cone, and $Y\simeq \PP(1,1,1,2)$;
\item
\label{cases:hyp:2}
$Y$ is a rational scroll and  $Y\simeq \PP_{\PP^1}(\EEE)$, where $\EEE$ is a rank $3$ vector bundle over $\PP^1$;
\item
\label{cases:hyp:3}
$Y$ is a cone over a rational scroll $F\subset \PP^g$ such that $F\simeq \PP_{\PP^1}(\EEE)$, where $\EEE$ is a rank $2$ vector bundle over $\PP^1$;

\item
\label{cases:hyp:4}
$Y$ is a cone over a rational normal curve $C\subset \PP^{g-1}$ that has degree $g-1$.
\end{enumerate}
If $X$ is the Veronese cone, then
$X\simeq X_6\subset \PP(1,1,1,2,3)$, so we get case \ref{prop:canFano:hyp2}. 

Suppose that $Y\simeq \PP_{\PP^1}(\EEE)$. We may assume that $E=\OOO_{\PP^1}(a)\oplus\OOO_{\PP^1}(b)\oplus\OOO_{\PP^1}$ with $a\geqslant b\geqslant 0$. The projection $\pi\colon \PP_{\PP^1}(\EEE) \to \PP^1$ must be $G$-equivariant and
the action of $G$ on $\PP^1$ is trivial. Thus $G$ acts faithfully any fiber $F\simeq \PP^2$.
If $b>0$ ($b=0$ and $a>0$, respectively), then $\PP_{\PP^1}(\EEE)$ contains an exceptional subscroll $Z$ corresponding to the surjection $\EEE\to \OOO_{\PP^1}(a)\oplus\OOO_{\PP^1}(b)$ ($\EEE\to \OOO_{\PP^1}(a)$, respectively). This subscroll must be $G$-invariant, and the intersection $F\cap Z$ is a point
(line, respectively). But then the action of the group $G$ on $F$ must be trivial, a contradiction. Therefore, we have $a=b=0$, so $Y\simeq \PP^2\times \PP^1$, where $G$ acts on $ \PP^2\times \PP^1$ through the first factor. By \eqref{eq:Hurwitz}, the divisor $B\subset \PP^2\times \PP^1$
has bidegree $(4,d)$ for $d\in\{0,1\}$. On the other hand, $B$ is invariant, which implies that $d=0$, 
so $X$ is $G$-equivariantly isomorphic to $S\times \PP^1$, where $S$ is the smooth del Pezzo surface of degree $2$ described in Lemma~\ref{lemma:PSL2-del-Pezzos},
which gives case \ref{prop:canFano:hyp1}. Then $\rho(X)^G=\rho(S)^G+1=2$, and $\overline{\operatorname{NE}}(X)^G$ is generated by the curves in the fibers of projections $X\to \PP^1$ and $X\to S$.

Therefore, we may assume that $Y$ is a cone as in \ref{cases:hyp:3} or  \ref{cases:hyp:4}. Then there exists a $G$-equivariant resolution of singularities $\mu\colon \widetilde Y\to Y$ such that $\widetilde{Y}\simeq \PP_{\PP^1}(\EEE)$, where $\EEE$ is a rank $3$ vector bundle over $\PP^1$. Arguing as above, we get $\widetilde Y\simeq \PP^2\times \PP^1$, hence $\mu$ is an isomorphism, a contradiction.
\end{proof}

\begin{corollary}
\label{corollary:hyperelliptic}
Suppose that $X$ is defined over $\mathbb{R}$, and $\mathrm{Aut}(X)$ has a subgroup $G\simeq\PSL_2(\FF_7)$. Then the divisor $-K_X$ is very ample.
\end{corollary}

\begin{proof}
Suppose that  $-K_X$ is not very ample. Then it follows from Lemma~\ref{lemma:hyperelliptic} that $X_{\CC}$ is isomorphic to one of the  complex 3-folds \ref{prop:canFano:hyp3}--\ref{prop:canFano:hyp1} listed in Lemma~\ref{lemma:hyperelliptic}. In the first case, the anticanonical linear system $|-K_X|$ gives a $G$-equivariant double cover $X\to\PP^3$, which is impossible by Corollary~\ref{corollary:Pn-PSL27}. The second case is ruled out by Corollary~\ref{corollary:DP}. In the third case, $X$ is smooth, and 
it follows from Corollary~\ref{corollary:DP} that $\rho(X)^G=2$ and there are 
two equivariant  extremal Mori contractiions 
$\varphi_1: X\to C$ and $\varphi_2: X\to Y$, where $Y$ is a real form of the smooth del Pezzo surface of degree $2$ described in Lemma~\ref{lemma:PSL2-del-Pezzos},
and $C$ is a smooth real conic in $\PP^2$.
This is impossible by Lemma~\ref{lemma:PSL27-curves} and Corollary~\ref{corollary:PSL27-real-rational-surface}.
\end{proof}

Now, we suppose that the divisor $-K_X$ is very ample. Then $|-K_X|$ gives an $\mathrm{Aut}(X)$-equivariant embedding $X\hookrightarrow\PP^{g+1}$, where $g=\g(X)=\frac12(-K_{X})^3+1\geqslant 4$ by the Riemann--Roch theorem. Let us identify $X$ with its image in $\PP^{g+1}$. Now, we describe trigonal Fano 3-folds \cite{HT} that admits a faithful action of the group $\PSL_2(\FF_7)$.

\begin{lemma}[{cf. \cite{HT}}]
\label{lemma:trigonal}
Suppose that $X$ is not an intersection of quadrics in $\PP^{g+1}$, and $\mathrm{Aut}(X)$ contains a subgroup $G=\PSL_2(\FF_7)$.
Then one of the following cases holds:
\begin{enumerate}
\item $g=3$ and $X$ is a quartic 3-fold in $\PP^{4}$;
\item $g=4$ and $X$ is a complete intersection of a quadric and a cubic in
$\PP^{5}$.
\end{enumerate}
\end{lemma}
\begin{proof}
Clearly, $g\geqslant 3$ and $X\subset \PP^4$ is a quartic  in the case $g=3$.
Let $Y\subset \PP^{g+1}$ be the intersection of all quadrics passing through $X$.
Then $Y$ is irreducible, normal and it is a four-dimensional variety of minimal degree  (see e.g. \cite{HT}).  Hence $X\subset \PP^5$ is an intersection of
a quadric and a cubic in the case $g=4$. Thus from now on we assume that $g\geqslant 5$.

Consider the case when $Y$ is either a scroll or a cone over scroll. Then we have the following $G$-equivariant diagram
$$
\xymatrix{X\ar@{^{(}->}[d]& \widetilde X\ar@{^{(}->}[d]\ar[l]\ar[r]^{\varphi} & \PP^1\ar@{=}[d]\\
Y & \widetilde Y\ar[l]_{\mu}\ar[r]^{\pi} & \PP^1}
$$
where  $\mu$ is an isomorphism if $Y$ is a (smooth) scroll, and $\mu$ is the blowup of the vertex if $Y$ is a cone. Here $\pi$ is a $\PP^3$-bundle and $\varphi$ is a fibration into cubic surfaces.  Since the action of $G$ on $\PP^1$ is trivial, $G$ acts faithfully on the fibers of $\varphi$, i.e. on cubic surfaces.
This contradicts Lemma~\ref{lemma:PSL2-del-Pezzos} 

Thus, $g=6$ and $Y$ is a cone in $\PP^7$ over the Veronese surface with vertex a line, say $L$.
In this case the representation of $G$ in $H^0(X,\OOO_X(-K_X))$ has the following decomposition:
$$
 H^0(X,\OOO_X(-K_X)) = V_6\oplus V_1\oplus V_1',
$$
where $V_1$ and $V_1'$ are trivial representations, and $V_6$ is an irreducible $6$-dimensional representation.
The  subspace $\PP(V_6)\subset\mathbb{P}^7$ is $G$-invariant, and $\PP(V_6)\cap L=\varnothing$.
Let $S:=Y\cap  \PP(V_6)$. Then $S$ is a $G$-invariant Veronese surface, and $G$ acts faithfully on $S$. If $S\not\subset X$, then the intersection $X\cap S$ is a $G$-invariant curve of degree $10$, so it is a plane quintic curve in $S\simeq \PP^2$, which is impossible, by Remark~\ref{remark:invariants}.
Thus, we see that $S\subset X$.

Let $\PPP$ be the  pencil  of hyperplane sections passing through $\PP(V_6)$. Clearly, $S$ is a fixed component of $\PPP$,
so we can write $\PPP=F+\MMM$, where  $F\neq 0$ is the fixed part of $\PPP$ and  $\MMM$ is a pencil without fixed components.
A general member  $M\in \MMM$ is irreducible and invariant, hence its linear span is a 6-dimensional subspace in $\PP^7$. 
Then $\deg (M)\geqslant 5$. On the other hand, we have
$$
\deg (M)=10-\deg(F)\leqslant 10-\deg(S)=6.
$$
This gives $\deg (M)=6$ and $F=S$. In particular, a general hyperplane section of $M$ is a curve of arithmetic genus $\leqslant 1$.
Since $M$ is acted by $G$, it cannot be ruled  (covered by lines). Then the only possibility is that $M$ is a del Pezzo surface with at worst Du Val singularities (see e.g. \cite[Theorem 8.1.11]{Dolgachev-ClassicalAlgGeom}). But this contradicts Lemma~\ref{lemma:PSL2-del-Pezzos}.
\end{proof}

\begin{corollary}
\label{corollary:trigonal}
Suppose that $X$ is defined over $\mathbb{R}$, and $\mathrm{Aut}(X)$ has a subgroup $G\simeq\PSL_2(\FF_7)$. Then either $g\geqslant 5$ and $X$ is an intersection of quadrics in $\PP^{g+1}$, or $g=4$ and $X$ is a complete intersection of a quadric and a cubic in
$\PP^{5}$ such that the action of the group $G$ on $\PP^5$ is induced by its unique $6$-dimensional faithful representation that is irreducible over $\mathbb{C}$, and $X(\mathbb{R})=\varnothing$.
\end{corollary}

\begin{proof}
Suppose that $X$ is not an intersection of quadrics in $\PP^{g+1}$. Then it follows from Lemma~\ref{lemma:trigonal} that $g\in\{3,4\}$, 
and either $X$ is a quartic 3-fold in $\PP^{4}$, or $X$ is a complete intersection of a quadric and a cubic in $\PP^{5}$. In the former case, the action of the group $G$ on $\PP^{4}$ is faithful, and it is induced by a real faithful $5$-dimensional representation of the group $G$, which does not exist by Lemma~\ref{lemma:reps}. In the latter case, the action of the group $G$ on $\PP^{5}$ is induced by a real faithful $6$-dimensional representation of the group $G$. By Lemma~\ref{lemma:reps}, there are two such representations. If it is reducible over $\mathbb{C}$, then $\PP^5$ has no $G$-invariant cubic hypersurfaces. Thus, the representation must be irreducible over $\mathbb{C}$. Then  $\PP^5$ contains unique invariant quadric, and a pencil of invariant cubics. The unique $G$-invariant quadric is real, and it is pointless, so $X$ is pointless too.
\end{proof}

Now, we suppose further that $g=\g(X)\geqslant 5$ and $X$ is an intersection of quadrics in $\PP^{g+1}$. Then it follows from \cite[Corollary 5.5]{P:JAG:simple} that 
\begin{equation}
\label{equation:quadrics}
h^0\big(\PP^{g+1},\, \mathcal{I}_X(2)\big)=\frac12 (g-2)(g-3),
\end{equation}
where $\mathcal{I}_X$ is the ideal sheaf of the 3-fold $X$. We also know that $X$ is projectively normal \cite{IP99}.

\begin{lemma}
\label{lemma:g-not-5}
Suppose that $X$ is defined over $\mathbb{R}$, and $\mathrm{Aut}(X)$ has a subgroup $G\simeq\PSL_2(\FF_7)$. 
Then $\g(X)\ne 5$.
\end{lemma}

\begin{proof}
Suppose that $\g(X)=5$. Then $X$ is a complete intersection of three quadrics in $\PP^6$,
the action of the group $G$ on $\PP^{6}$ is faithful, and it is induced by a real faithful $7$-dimensional representation of the group $G$. In every possible case, $\PP^6$ has no $G$-invariant two-dimensional real linear system of quadrics, which is a contradiction.
\end{proof}

\begin{lemma}
\label{lemma:g-not-7}
Suppose that $X$ is defined over $\mathbb{R}$, and $|-K_{X}|$ has at most one $G$-invariant surface. Then $g\not\in\{7,8,9\}$.
\end{lemma}

\begin{proof}
Recall that $H^0(X,\mathcal{O}_X(-K_X))$ is a faithful real representation of  $G$ of dimension $g+2$. By assumption, it has at most one $1$-dimensional subrepresentation. Then $g\not\in\{8,9\}$ by Lemma~\ref{lemma:reps}.

Suppose that $g=7$. Then it follows from Lemma~\ref{lemma:reps} that $H^0(X,\mathcal{O}_X(-K_X))$ splits as 
follows
$$
H^0(X,\mathcal{O}_X(-K_X))= \mathbb{V}_1 \oplus  \mathbb{V}_8. 
$$
Here we used notations of Lemma~\ref{lemma:reps}. In particular, we see that $|-K_X|$ contains exactly one $G$-invariant surface, which we denote by $S$. 
Note that the space $W\subset \mathrm{Sym}^2\big(H^0(X,\mathcal{O}_X(-K_X))\big)$ of quadratic forms vanishing on $X$ is invariant and $10$-dimensional by  \eqref{equation:quadrics}.
On the other hand, we have the following splitting of $G$-representations
$$
\mathrm{Sym}^2\Big(H^0(X,\mathcal{O}_X(-K_X))\Big)= \mathbb{V}_1 \oplus \mathbb{V}_1 \oplus 
\mathbb{V}_6 \oplus \mathbb{V}_6 \oplus \mathbb{V}_7 \oplus 
\mathbb{V}_8 \oplus \mathbb{V}_8 \oplus \mathbb{V}_8.
$$
Then for the splitting of $W$ the only possibility is $W=\mathbb{V}_1 \oplus \mathbb{V}_1 \oplus 
\mathbb{V}_8$, hence $X$ is contained in every $G$-invariant quadric hypersurface in $\PP^8$, and the linear system $|-2K_X|$ does not contain $G$-invariant divisors, which is not true since $2S\in |-2K_X|$ and the divisor $2S$ is $G$-invariant. 
\end{proof}

\subsection{Actions on real $G\mathbb{Q}$-Fano 3-folds}
\label{subsection:PSL27-GQ-Fanos}

Let $X$ be a real geometrically irreducible 3-fold such that $X$ has terminal singularities, and $X$ is $G\mathbb{Q}$-factorial,  i.e. every Weil divisor on $X$ defined over $\mathbb{R}$ whose class in $\mathrm{Cl}(X)$ is $G$-invariant is $\mathbb{Q}$-Cartier. 
Suppose, in addition, that $X_\CC$ is rationally connected, and $\mathrm{Aut}(X)$ contains a subgroup $G\simeq\PSL_2(\mathbf{F}_7)$. 

\begin{lemma}
\label{lemma:fibration}
The 3-fold $X$ is $G$-birational over $\mathbb{R}$ neither to a $G$-conic bundle nor to a $G$-del Pezzo fibration.
\end{lemma}

\begin{proof}
Suppose that the assertion is not true. Then $X$ is $G$-birational over $\mathbb{R}$ to a normal 3-fold $Y$ with terminal $G\mathbb{Q}$-factorial singularities, and there there exists a $G$-equivariant surjective morphism $\pi\colon X\to Z$ with normal $Z$ such that
\begin{itemize}
\item either $\pi$ is a conic bundle, and $Z$ is a real geometrically rational surface,
\item or $\pi$ is a del Pezzo fibration and $Z_\CC\simeq\PP^1$.
\end{itemize}
In both cases, the $G$-action on $Z$ is trivial by Lemma~\ref{lemma:PSL27-curves} and Corollary~\ref{corollary:PSL27-real-rational-surface}. 
Moreover, applying Lemma~\ref{lemma:PSL27-curves} again, we see that $\pi$ is not a conic bundle.

Thus, we see that $\pi$ is a del Pezzo fibration. Then it follows from Lemma~\ref{lemma:PSL2-del-Pezzos} that its fiber over a general point in $Z_\CC$ is either $\PP^2$ or the smooth del Pezzo surface of degree $2$ described in Lemma~\ref{lemma:PSL2-del-Pezzos}. Moreover, if $Z\simeq\PP^1$, then we obtain a contradiction applying Corollary~\ref{corollary:PSL27-real-rational-surface} to the fiber of $\pi$ over a general point in $Z(\mathbb{R})$. Hence, we have $Z(\mathbb{R})=\varnothing$, which implies that $X(\mathbb{R})=\varnothing$.

Let $S$ be the generic fiber of $\pi$, and let $\Bbbk=\mathbb{R}(Z)$. Then $S$ is a smooth del Pezzo surface over $\Bbbk$, and $G$ acts faithfully on $S$, because it acts trivially on $Z$.  Moreover, either $S$ is a form of $\PP^2_\Bbbk$, or $S$ is a del Pezzo surface of degree $2$. In the latter case, the linear system $|-K_S|$ gives a $G$-equivariant double cover $S\to\PP^2_\Bbbk$ ramified in a quartic curve, which implies that $\mathrm{PGL}_3(\Bbbk)\simeq\mathrm{Aut}(\PP^2_\Bbbk)$ contains a subgroup isomorphic to $G$. On the other hand, it follows from \cite[Proposition~5.5]{Hu} that $\mathrm{PGL}_3(\Bbbk)$ contains a subgroup isomorphic to $G$ if and only if $\sqrt{-7}\in\Bbbk$ (and such subgroup is unique up to conjugation by \cite[Proposition 2.2]{BeauvillePGL}). Since $\sqrt{-7}\notin\Bbbk$, we conclude that $\mathrm{PGL}_3(\Bbbk)$ does not contain subgroups isomorphic to $G$, and $S$ is a non-trivial form of $\PP^2_\Bbbk$.

Over $\mathbb{C}$, the del Pezzo fibration $\pi$ has many sections. Let $C$ be one of them, and let $C^\prime$ be a complex conjugate curve. Then the curve $C+C^\prime$ is defined over $\mathbb{R}$, so it defines a $\mathbb{K}$-point on $S$ for some quadratic extension $\mathbb{K}$ of the field $\Bbbk$. This implies that $S\simeq\PP^2_\Bbbk$ \cite{GilleSzamuely,KollarSB}, which is a contradiction.
\end{proof}

Thus, applying $G$-equivariant Minimal Model Program over $\mathbb{R}$, we obtain a $G$-birational map from $X$ to a real $G\mathbb{Q}$-Fano 3-fold. 
Therefore, we assume that $X$ is already a real $G\mathbb{Q}$-Fano 3-fold. However, for applications in Section~\ref{section:non-Gorenstein}, we make a little more general assumptions in the following lemma.

\begin{lemma}
\label{lemma:PSL27-real-3-folds-log-pair}
Let $X$ be a Fano 3-fold with only $G\QQ$-factorial pseudo-terminal singularities and $\rho(X)^G=1$.
Let $D$ be a $G$-invariant effective $\mathbb{Q}$-divisor on  $X$ defined over $\mathbb{R}$ such that $D\sim_{\mathbb{Q}}-K_X$. Then $(X,D)$ is log canonical.
\end{lemma}

\begin{proof}
Suppose that the  pair $(X,D)$ is is not log canonical. Then, using \cite[Lemma 2.4.10]{Icosahedron} or arguing as in the proofs of  \cite[Theorem 1.10]{Kawamata-1} and \cite[Theorem 1]{Kawamata-2}, we see that there are rational number $\lambda<1$ and an effective real $G$-invariant $\mathbb{Q}$-divisor $D^\prime$ on the 3-fold $X$ such that $D^\prime\sim_{\mathbb{Q}} D$, the log pair $(X,\lambda D^\prime)$ is log canonical,  $(X,\lambda D^\prime)$ is not Kawamata log terminal, and the locus $\mathrm{Nklt}(X,\lambda D^\prime)$ consists of the $G$-orbit of a real minimal log canonical center of the log pair $(X,\lambda D^\prime)$. Here, we do not assume that minimal log canonical center is geometrically irreducible.

It follows from \cite{Kawamata-1} that complex irreducible components of minimal log canonical centers of the log pair $(X,\lambda D^\prime)$ are disjoint.
On the other hand, it follows from the Koll\'ar--Shokurov vanishing theorem that the locus $\mathrm{Nklt}(X,\lambda D^\prime)$ is connected. Thus,  $\mathrm{Nklt}(X,\lambda D^\prime)$ consists of a real geometrically irreducible $G$-invariant subvariety $Z$.

If $Z$ is a point, we obtain a contradiction with Corollary~\ref{corollary:fixed-point}. Thus, either $Z$ is a curve or a surface. Then it follows from Kawamata's subadjunction theorem \cite{Kawamata-2} that $Z$ is normal and geometrically rational. Thus, if $Z$ is a curve, then $G$ acts trivially on $Z$, since $\mathrm{PGL}_2(\CC)$ does not contain subgroups isomorphic to $G$. Similarly, if $Z$ is a surface, then $G$ acts trivially on $Z$ by Corollary~\ref{corollary:PSL27-real-rational-surface}. Now, using Corollary~\ref{cor:Klein-action-on-3-folds}, we obtain a contradiction.
\end{proof}

\begin{corollary}
\label{cor:PSL27-real-3-folds-log-pair-mobile}
Let $\mathcal{M}$ be a non-empty $G$-invariant (real) mobile linear system on $X$ such that 
$$
\mathcal{M}\sim_{\mathbb{Q}}a(-K_X)
$$ 
for some positive real number $a\leqslant 1$. Then either $a=1$ and $(X,\mathcal{M})$ is canonical, or $X$ is $G$-birational to the pointless form of the 3-fold \eqref{equation:flag}, and $X$ is not rational over $\mathbb{R}$.
\end{corollary}

\begin{proof}
Assume either $a<1$ or the singularities of  $(X,\mathcal{M})$ are worse than  canonical. Then
by Lemmas~\ref{lemma:MMP}   and~\ref{lemma:fibration} there is a $G$-equivariant real birational map $\widetilde{X}\dasharrow X^\prime$ such that $X^\prime$ is a real Gorenstein $G\mathbb{Q}$-Fano 3-fold. 
In this case $\iota(X)\geqslant 2$, so it follows from Corollary~\ref{corollary:DP}  that $X$ is the pointless form of the 3-fold \eqref{equation:flag}, so $X$ is not rational over $\mathbb{R}$. 
\end{proof}

\begin{corollary}
\label{corollary:K3-surfaces-1}
Let $\mathcal{M}$ be a non-empty real $G$-invariant linear subsystem in $|-K_X|$ that has positive dimension. Suppose $X$ is rational over $\mathbb{R}$. Then $\mathcal{M}$ is mobile, and $(X,\mathcal{M})$ is canonical.
\end{corollary}

\begin{proof}
If $\mathcal{M}$ has a fixed component, its mobile part is $\mathbb{Q}$-rationally equivalent to $a(-K_X)$ for some rational $a<1$, which contradicts Corollary~\ref{cor:PSL27-real-3-folds-log-pair-mobile}, so the assertion follows by  Corollary~\ref{cor:PSL27-real-3-folds-log-pair-mobile}.
\end{proof}

\begin{corollary}
\label{corollary:K3-surfaces-2}
Let $\mathcal{M}$ be a non-empty real $G$-invariant linear subsystem in $|-K_X|$ that has positive dimension, and let $S$ be a general surface in $\mathcal{M}$.  Suppose that $X$ is rational over $\mathbb{R}$. Then $(X,S)$ is purely log terminal, and $S$ is an irreducible K3 surface with at most Du Val singularities.
\end{corollary}

\begin{proof}
By Corollary~\ref{corollary:K3-surfaces-1}, the linear system $\mathcal{M}$ is mobile, and the log pair $(X,\mathcal{M})$ has canonical singularities. Moreover, general member of the linear system $\mathcal{M}$ is irreducible by Bertini theorem, since otherwise $\mathcal{M}$ would be composed from a mobile pencil $\mathcal{P}$, so $\mathcal{M}\sim n\mathcal{P}$ for some $n>1$, which would contradict Corollary~\ref{cor:PSL27-real-3-folds-log-pair-mobile}.

Now, applying \cite[Theorem 4.8]{KollarPairs} or \cite[Lemma 1.12]{Alexeev:ge}, we see that $(X,S)$ has purely log terminal singularities. Then $S$ has Kawamata log terminal singularities by \cite[Theorem 7.5]{KollarPairs}, and they are Du Val, since $K_S\sim 0$ by the adjunction formula. Finally,the equality $h^1(\mathcal{O}_{S})=1$ follows from the exact sequence
\[
0 \longrightarrow \OOO_X(K_X) \longrightarrow \OOO_X \longrightarrow \OOO_S \longrightarrow 0
\]
and vanishings $H^0(\mathcal{O}_{X}(K_X))=H^1(\mathcal{O}_{X}(K_X)=0$.
\end{proof}

Let us conclude this section with one useful technical result similar to \cite[Lemma 4.7]{P:JAG:simple}, whose proof is almost identical to the proof of \cite[Lemma 4.7]{P:JAG:simple}.

\begin{lemma}
\label{lemma:G-invariant-surfaces}
Suppose that $|-K_X|$ has a $G$-invariant real surface $S$. Then $S$ is reduced, and
\begin{itemize}
\item either $S_{\CC}$ is an irreducible smooth K3 surface,
\item or $S$ is reducible over $\mathbb{R}$, and $G$ acts transitively on the set of its real components.
\end{itemize}
\end{lemma}

\begin{proof}
By Lemma~\ref{lemma:PSL27-real-3-folds-log-pair}, \ $(X,S)$ is log canonical.
Hence, $S$ is reduced. Let $D$ be the $G$-orbit of an irreducible component of $S$. Then $D$ is $G$-invariant, so $D\sim_{\mathbb{Q}} a(-K_X)$ for some $a\in\mathbb{Q}_{>0}$ such that $a\leqslant 1$. Moreover, if $D\ne S$, then $a<1$, so  $(X,\frac{1}{a}D)$ is not log canonical, which contradicts Lemma~\ref{lemma:PSL27-real-3-folds-log-pair}. Hence, we see that $D=S$, so $G$ acts transitively on the set of real irreducible components of the surface $S$. If $S$ is reducible over $\mathbb{R}$, we are done. Thus, we may assume that $S$ is irreducible over $\mathbb{R}$. Let us show that $S$ is a smooth K3 surface. 

We claim that $S$ is geometrically irreducible. Indeeed, suppose it is not. Then $S_{\CC}=S^\prime+S^{\prime\prime}$, where $S^\prime$ and $S^{\prime\prime}$ are irreducible $G$-invariant complex surfaces that are swapped by the complex conjugation. Let $f\colon\widetilde{X}\to X_{\CC}$ be the $G\mathbb{Q}$-factorialization of the 3-fold $X_{\CC}$, let $\widetilde{S}^\prime$ and $\widetilde{S}^{\prime\prime}$ be the strict transforms on the 3-fold $\widetilde{X}$ of the surfaces $S^\prime$ and $S^{\prime\prime}$, respectively. Then $(\widetilde{X},\widetilde{S}^\prime+\widetilde{S}^{\prime\prime})$ has log canonical singularities, and
$$
\widetilde{S}^\prime+\widetilde{S}^{\prime\prime}\sim -K_{\widetilde{X}}\sim_{\mathbb{Q}} f^*(-K_X).
$$
Now, let $h\colon\widehat{S}^\prime\to \widetilde{S}^\prime$ be the normalization, and let $B_{\widehat{S}^\prime}$ be the divisor on $\widehat{S}^\prime$ known as the different of the log pair $(\widetilde{X},\widetilde{S}^\prime+\widetilde{S}^{\prime\prime})$, which is defined as in \cite{Kawakita}. Then $B_{\widehat{S}^\prime}$ is a non-zero effective Weil divisor on $\widehat{S}^\prime$ such that $B_{\widehat{S}^\prime}\sim -K_{\widehat{S}^\prime}$, and $(\widehat{S}^\prime,B_{\widehat{S}^\prime})$ has log canonical singularities  \cite{Kawakita}, so the divisor $B_{\widehat{S}^\prime}$ is reduced. Now, replacing $\widehat{S}^\prime$ by its minimal resolution, and replacing $B_{\widehat{S}^\prime}$ by its log pull back, we obtain a smooth surface faithfully acted by $G$ whose anticanonical linear system contains a non-zero reduced $G$-invariant divisor. This is impossible by Corollary~\ref{corollary:PSL27-ruled-surface}. 

Hence, we see that $S$ is geometrically irreducible. Moreover, the arguments we used to show this also imply that $S$ is normal. If $S$ is smooth, then, arguing as in the proof of Corollary~\ref{corollary:K3-surfaces-2}, we see that $S$ is a K3 surface, so we are done. Thus, we may assume that $S$ is singular. Then it follows from Corollary~\ref{corollary:K3} that $S_{\CC}$ has a non-Du Val singular point~$P$. If $P$ is not $G$-fixed, its $G$-orbit contains at least $7$ points by Corollary~\ref{cor:act}, but $S_{\CC}$ can have at most two non-Du Val singular points \cite{Umezu}, see also \cite[Theorem 6.9]{Shokurov92} and \cite{Fujino}. Thus, $P$ is fixed by $G$. Then it follows from Lemma~\ref{lemma:fixed-point:21} that either $X_{\CC}$ is smooth at $P$, or $X_{\CC}$ has cyclic quotient singularity of type $\frac{1}{2}(1,1,1)$ at~$P$. Then $(X,S)$ is not log canonical by Corollary~\ref{corollary:fixed-point:21}, which is not the case.
\end{proof}

\section{Gorenstein case}
\label{section:Gorenstein}

Let $X$ be a real Fano 3-fold with terminal Gorenstein singularities such that $\mathrm{Aut}(X)$ contains a subgroup $G\simeq\PSL_2(\FF_7)$. Suppose that $X$ is a $G\mathbb{Q}$-Fano 3-fold. In this section, we will prove 

\begin{theorem}
\label{theorem:PSL27-real-Gorenstein}
The 3-fold $X$ is not rational over $\mathbb{R}$.
\end{theorem}

Let us prove Theorem~\ref{theorem:PSL27-real-Gorenstein}. Suppose that $X$ is rational over $\mathbb{R}$. Let us seek for a contradiction. 

\subsection{Real Gorenstein $G\mathbb{Q}$-Fano 3-folds}
\label{subsection:Gorenstein}
By Lemma~\ref{lemma:SB} and Corollary~\ref{corollary:DP} we have
$\iota(X)=1$, i.e. the anticanonical class $-K_{X_\CC}$ is a primitive element of the lattice $\mathrm{Pic}(X_\CC)$.
By Lemma~\ref{lemma:canFano:Bs} and Corollary~\ref{corollary:hyperelliptic},  $|-K_X|$ gives a $G$-equivariant embedding $X\hookrightarrow\PP^{g+1}$, where $g=\g(X)=\frac12 (-K_{X})^3+1\geqslant 3$. We identify $X$ with its anticanonical image in $\PP^{g+1}$. Then it follows from Corollary~\ref{corollary:trigonal} that $g\geqslant 5$ and $X$ is an intersection of quadrics in $\PP^{g+1}$. 
It follows from \cite{Namikawa,SmoothingPic} that $X_{\CC}$ admits a $\mathbb{Q}$-Gorenstein smoothing to a smooth complex Fano 3-fold $V$ such that $(-K_{V})^3=(-K_{X_{\CC}})^3=2g-2\geqslant 10$ and $\rho(V)=\rho(X_{\CC})\geqslant\rho(X)$. Moreover, it follows from \cite{Namikawa} that
\begin{equation}
\label{equation:Sing-Gorenstein}
|\mathrm{Sing}(X_{\CC})|\leqslant 20+h^{1,2}\big(V\big)-\rho\big(V\big)\leqslant 19+h^{1,2}\big(V\big).
\end{equation}
Furthermore, $-K_{X_{\CC}}$ is divisible by $\iota(V)$ in $\mathrm{Pic}(X_{\CC})$, so $\iota(V)=\iota(X)=1$ by Corollary~\ref{corollary:DP}. 

\begin{lemma}
\label{lemma:g-7-8-9}
Suppose  $g\in\{6,7,8,9,10\}$. Then $|-K_{X_{\CC}}|$ has at most one $G$-invariant surface.
\end{lemma}

\begin{proof}
Suppose that $|-K_{X_{\CC}}|$ contains more than one $G$-invariant surface. Then $|-K_X|$ contains two distinct $G$-invariant real surfaces $S$ and $S^\prime$. These surfaces generate a pencil $\mathcal{M}\subset|-K_X|$, and every surface in this pencil is $G$-invariant. Moreover, by Corollary~\ref{corollary:K3-surfaces-2} and Lemma~\ref{lemma:G-invariant-surfaces}, general surface in $\mathcal{M}$ is a smooth geometrically irreducible K3 surface. In fact, the same holds for any real surface in the pencil $\mathcal{M}$. 

Indeed, suppose that $S$ is not a smooth geometrically irreducible K3 surface. Then it follows from Lemma~\ref{lemma:G-invariant-surfaces} that $S$ is reducible over $\mathbb{R}$, and $G$ transitively permutes its components. On the other hand, $X_\CC$ does not contain planes \cite{Prokhorov-G-Fanos}, so it follows from Corollary~\ref{cor:act} and $\mathrm{deg}(S)=2g-2\in\{10,12,14,16,18\}$ that either $\g(X)=8$ and $S=S_1+\cdots+S_7$, or $\g(X)=9$ and $S=S_1+\cdots +S_8$,
where each $S_i$ is a geometrically irreducible quadric surface. Let $H$ be a general hyperplane section of $X$, and let $C_i:= S_i|_H$. Then $H$ is a smooth K3 surface, and each $C_i$ is a smooth conic on $H$. In particular, we see that $C_i^2=-2$ on the surface $H$ by the adjunction formula. Since $G$ acts doubly transitive on the irreducible components of $S$, we have $C_i\cdot C_j=C_1\cdot C_2$ for every possible $i\ne j$. Thus, if $\g(X)=8$, then 
$$
14=2\g(X)-2=(C_1+\cdots+C_7)^2=-2\cdot 7+7(7-1)C_1\cdot C_2,
$$
which gives $C_1\cdot C_2=\frac{2}{3}$. Similarly, if  $\g(X)=9$, then
$$
16=2\g(X)-2=(C_1+\cdots+C_8)^2=-2\cdot 8+8(8-1)C_1\cdot C_2,
$$
which gives $C_1\cdot C_2=\frac{4}{7}$. However, $C_1\cdot C_2$ is an integer. The obtained contradiction shows that $S$ is a smooth geometrically irreducible K3 surface. In particular, $X$ is smooth along $S$.

Note that $G$ acts faithfully on $S$ by Corollary~\ref{cor:Klein-action-on-3-folds}. 
Thus, it follows from Lemma~\ref{lemma:K3} that $\mathrm{Pic}(S_{\CC})^G=\mathbb{Z}[L]$ for some ample divisor $L$ on the surface $S_{\CC}$ such that $L^2$ is even. Set $Z=S^\prime\vert_{S}$. Then $Z$ is a $G$-invariant real curve in $S$, so $Z_\CC\sim nL$ for some $n\geqslant 1$. Then 
$$
n^2L^2=Z^2=2g-2\in\{10,12,14,16,18\}.
$$
In particular, if $n\ne 1$, then $g=9$, $L^2=4$, and $n=2$. 

We claim that $Z$ is reduced and $G$-irreducible. Indeed, if this is not the case, then $Z_{\CC}=Z_1+Z_2$, where $Z_1$ and $Z_2$ are $G$-irreducible curves such that $Z_1\sim Z_2\sim L$. Thus, if $Z_1\ne Z_2$, then the intersection $Z_1\cap Z_2$ is $G$-invariant subset that consists of at most $4$ points, which is impossible by Lemma~\ref{lemma:Klein-action-on-surfaces}. Similarly, if $Z_1=Z_2$, then $Z_1^2=L\cdot Z_1=4$, so it follows from Corollary~\ref{cor:act} and adjunction formula that $Z_1$ is a real geometrically irreducible curve of arithmetic genus $3$, which is impossible by Corollary~\ref{corollary:PSL27-real-curves}.
Thus, $Z$ is reduced and $G$-irreducible. 

Assume that the curve $Z_{\CC}$ is not irreducible.  Write $Z_{\CC}=C_1+\cdots +C_m$, where $m>1$ and each $C_i$ is a complex irreducible curve. Then $m$ divides $\deg(Z)=2g-2$, so it follows from Corollary~\ref{cor:act} that one of the following three cases holds:
\begin{itemize}
\item $g=9$, $m=8$, each curve $C_i$ is a smooth conic;
\item $g=8$, $m=7$, each curve $C_i$ is a smooth conic;
\item $g=8$, $m=14$, each curve $C_i$ is a line.
\end{itemize}
Moreover, if $m=7$ or $m=8$, then $G$ acts doubly transitive on the set $\{C_1,\ldots,C_m\}$, which gives
$$
2g-2=Z^2=\big(C_1+\cdots +C_m\big)^2=-2m+m(m-1)C_1\cdot C_2,
$$
which is impossible, since $C_1\cdot C_2$ is a non-negative integer. Thus, we conclude that $g=8$, $m=14$, and each curve $C_i$ is a line. Then
$$
1=\deg (C_1)=-K_X\vert_{S}\cdot C_1=Z\cdot C_1=C_1^2+\sum_{i=2}^{14}C_i\cdot C_1=-2+\sum_{i=2}^{14}C_i\cdot C_1,
$$
which gives $(C_2+\cdots+C_{14})\cdot C_1=3$, so $|(C_2\cup\cdots\cup C_{14})\cap C_1|\leqslant 3$. On the other hand, the stabilizer of the curve $C_1$ in $G$ is isomorphic to $\mathfrak{A}_4$, and it does not fix points in $S_{\CC}$, because $\mathfrak{A}_4$ has no two-dimensional faithful representations. Thus, the stabilizer acts faithfully on $C_1\simeq\PP^1$, and it leaves invariant the subset $(C_2\cup\cdots\cup C_{14})\cap C_1$, which is impossible, because $\mathfrak{A}_4$ has no cyclic subgroups of index $\leqslant 3$.

Therefore, $Z_{\CC}$ is  irreducible. Then $\mathrm{p}_{\mathrm{a}}(Z)=\g(X)$ by the adjunction formula, and $G$ acts faithfully on $Z$ by Lemma~\ref{lemma:Klein-action-on-surfaces}. Hence, it follows from Lemma~\ref{lemma:Klein-action-on-surfaces} that $Z$ is a smooth curve of genus $\g(X)$. Thus, applying Corollary~\ref{corollary:PSL27-real-curves}, we see that $\g(X)=8$, because the curve $Z$ is real. In particular, we conclude that $X_{\CC}$ is singular, since otherwise it is known to be irrational \cite{IP99}.

Note that $H^0(X,\mathcal{O}_{X}(-K_{X}))$ is a real faithful $10$-dimensional $G$-representation. By our assumption, it has a $2$-dimensional trivial subrepresentation, which corresponds to the pencil $\mathcal{M}$. If it had a $3$-dimensional trivial subrepresentation, we would have a $G$-invariant surface $S^{\prime\prime}\in|-K_X|$ that is not contained in the pencil $\mathcal{M}$, so the intersection $S^{\prime\prime}_{\CC}\cap Z_{\CC}$ would be a finite $G$-invariant subset in $Z_{\CC}$ such that $|S^{\prime\prime}_{\CC}\cap Z_{\CC}|\leqslant S^{\prime\prime}\cdot Z_{\CC}=14$, which is impossible by Lemma~\ref{lemma:PSL27-curves}. Hence, it follows from Lemma~\ref{lemma:reps} that $H^0(X,\mathcal{O}_{X}(-K_{X_{\CC}}))$ splits as a sum of two trivial $1$-dimensional representations and one irreducible $8$-dimensional representation. In particular, $X_{\CC}$ has no $G$-orbits of length $7$.

On the other hand, we have $h^{1,2}(V)=5$, so \eqref{equation:Sing-Gorenstein} gives $|\mathrm{Sing}(X_{\CC})|\leqslant 24$. Hence, it follows from Corollary~\ref{cor:act} and Lemma~\ref{lemma:fixed-point:21} that the singular points of the 3-fold $X_{\CC}$ form one $G$-orbit of length $14$, $21$ or $24$. 
Thus, there is a  real surface $S^\sharp\in \mathcal{M}$ such that $S_{\CC}^\sharp$ contains all singular points of  $X_{\CC}$.
On the other hand, we have shown that any real surface in $\mathcal{M}$ is smooth and contained in the smooth locus of $X$, a contradiction.
\end{proof}

\begin{corollary}
\label{corollary:g-7-8-9}
One has $g\not\in\{7,8,9\}$. If $|-K_{X_{\CC}}|$ has a $G$-invariant surface, then $g\ne 10$.
\end{corollary}

\begin{proof}
By Lemma~\ref{lemma:g-not-7} and Lemma~\ref{lemma:g-7-8-9}, we have $g\not\in\{7,8,9\}$. If $g=10$ and $|-K_{X_{\CC}}|$ contains a $G$-invariant surface, then $|-K_{X_{\CC}}|$ contains at least two $G$-invariant surfaces by  Lemma~\ref{lemma:reps}, which contradicts Lemma~\ref{lemma:g-7-8-9}.
\end{proof}

\begin{lemma}
\label{lemma:Pic-Z}
One has $\rho(X)=1$ and $\rho(X_{\CC})\leqslant 2$.
\end{lemma}

\begin{proof}
By \cite{SmoothingPic} (see also \cite[Proposition 2.5]{KP:1nodal}), there is natural identification $\mathrm{Pic}(X_{\CC})\simeq\mathrm{Pic}(V)$ that preserves the intersection form. Thus, the group $G$ and the complex conjugation both act on $\mathrm{Pic}(V)$ such that every invariant element in $\mathrm{Pic}(V)$ is a multiple of $-K_V$. Moreover, if $\rho(X)>1$, the action of the group $G$ on $\mathrm{Pic}(X_\CC)$ is faithful, so it follows from Lemma~\ref{lemma:reps} that $\rho(V)=\rho(X_{\CC})\geqslant 7$, which is impossible by \cite[Theorem 1.2]{P:GFano2}. Thus, we see that $\rho(X)=1$, and $G$ acts trivially on $\mathrm{Pic}(X_\CC)$. Therefore, we see that the $\mathrm{Gal}(\mathbb{C}/\mathbb{R})$-invariant part of $\mathrm{Pic}(V)$ is generated by $-K_V$, so it follows from \cite[Proposition 5.2]{P:GFano2} and \cite[Lemma 4.4]{P:GFano2} that $\rho(X_{\CC})=\rho(V)\leqslant 2$ unless $V$ is an intersection of divisors of degree $(1,1,0)$, $(1,0,1)$, $(0,1,1)$ in $\mathbb{P}^2\times\mathbb{P}^2\times\mathbb{P}^2$. In the latter case, it follows from \cite[Proposition 6.4]{P:GFano2} and  \cite[Proposition 6.3]{P:GFano2} that the action of the group $\mathrm{Gal}(\mathbb{C}/\mathbb{R})$ on $\mathrm{Pic}(V)\otimes_{\mathbb{Z}}\mathbb{R}$ preserves the cone of nef divisors. This cone is generated by $\pi_1^*(\mathcal{O}_{\mathbb{P}^2}(1))$, $\pi_2^*(\mathcal{O}_{\mathbb{P}^2}(1))$ and $\pi_3^*(\mathcal{O}_{\mathbb{P}^2}(1))$, where each $\pi_i\colon V\to\mathbb{P}^2$ is the projection to the $i$-th factor of  $\mathbb{P}^2\times\mathbb{P}^2\times\mathbb{P}^2$. Hence, one of these generators must be $\mathrm{Gal}(\mathbb{C}/\mathbb{R})$-invariant, which is impossible, since the $\mathrm{Gal}(\mathbb{C}/\mathbb{R})$-invariant part of $\mathrm{Pic}(V)$ is generated by $-K_V$.
\end{proof}

If $\rho(X_{\CC})=2$, it follows from \cite[Theorem 1.2]{P:GFano2} that one of the following  cases holds:
\begin{enumerate}
\item\label{G-Fano-1} $V$ is a complete intersection of three divisors of degree $(1,1)$ in $\PP^3\times\PP^3$;
\item\label{G-Fano-2} $V$ is a blowup of a smooth quadric 3-fold in $\PP^4$ along a twisted quartic curve;
\item\label{G-Fano-3} $V$ is either a smooth divisor of degree $(2,2)$ in $\PP^2\times\PP^2$ or a double cover of a smooth divisor degree $(1,1)$ in $\PP^2\times\PP^2$ ramified in an anticanonical surface.
\end{enumerate}
In the first two cases,  $\mathrm{Aut}(V)$ has been studied in \cite{2-12} and \cite{Joe2024}. In particular, we know from \cite{2-12} that there exists a unique smooth complex complete intersection of three divisors of degree $(1,1)$ in $\PP^3\times\PP^3$  that admits a faithful action of the group $G$.

\begin{corollary}
\label{corollary:rho-1}
One has $\rho(X)=\rho(X_{\CC})=1$.
\end{corollary}

\begin{proof}
Suppose that $\rho(X_{\CC})=2$. If $V$ is a blow up of a quadric 3-fold in $\PP^4$ along quartic curve, it follows from \cite[Theorem 6.5]{P:GFano2} that there is a $G$-equivariant birational morphism $X_{\CC}\to Q$ such that $Q$ is a quadric in $\PP^4$, which is impossible by Remark~\ref{remark:P3-quadric}. This excludes case \ref{G-Fano-1}.

Suppose that we are in case \ref{G-Fano-2}. Then it follows from  \cite[Theorem 6.5]{P:GFano2} that $X_{\CC}$ is also an intersection of three divisors of degree $(1,1)$ in $\PP^3\times\PP^3$. Let $\mathrm{pr}_1\colon X_\CC\to\PP^3$ be the projection to the first factor. Then $\mathrm{pr}_1$ is $G$-equivariant and birational. Moreover, away from finitely many points in~$\PP^3$, $\mathrm{pr}_1$ is a blow up of $G$-invariant (possibly reducible) sextic curve $C$. If the $G$-action on $\PP^3$ leaves a hyperplane invariant, it follows from Lemma~\ref{lemma:Klein-action-on-surfaces} that $C$ is contained in this hyperplane, so $-K_{X_\CC}$ is not nef. Hence, the $G$-action on $\PP^3$ is induced by an irreducible $4$-dimensional representation of the central extension of the group $G$. Then it follows from by \cite[Lemma 3.7]{CheltsovShramovKlein} that $C$ is a smooth irreducible curve of genus $3$, so $X_{\CC}$ is smooth, and $\mathrm{pr}_1$ is a blow up of the curve $C$. Observe that $\PP^3$ contains a unique $G$-orbit of length $8$ \cite[Lemma~3.2]{CheltsovShramovKlein}. By Lemma~\ref{lemma:PSL27-curves}, this orbit is not contained in $C$. Thus, $X_{\CC}$ has a unique $G$-orbit of length $8$. By Corollary~\ref{cor:act} the action of $G$ on the points of this orbit is doubly transitive, hence all  these points are real, which contradicts Corollary~\ref{corollary:fixed-point}. 

Thus, we are in case \ref{G-Fano-3}. Then it follows from the proof of \cite[Theorem 6.5]{P:GFano2} that either $X_{\CC}$ is a divisor of degree $(2,2)$ in $\PP^2\times\PP^2$, or there exists a $G$-equivariant double cover $X_{\CC}\to W$ such that $W$ is a smooth divisor of degree $(1,1)$ in $\PP^2\times\PP^2$, and the ramification divisor is contained in the linear system $|-K_W|$. In both cases, the $G$-action lifts to $\PP^2\times\PP^2$. Moreover, if $\PP^2\times\PP^2$ contains a $G$-invariant divisor of degree $(1,1)$, then the action of the group $G$ on $\PP^2\times\PP^2$ is twisted diagonal, which implies that the only $G$-invariant divisor of degree $(2,2)$ is a multiple of the invariant divisor of degree $(1,1)$. Hence, we conclude that $X_{\CC}$ is a divisor of degree $(2,2)$ in $\PP^2\times\PP^2$, and $G$ acts on $\PP^2\times\PP^2$ diagonally. In this case, $X_{\CC}$ is the only $G$-invariant divisor of degree $(2,2)$ in $\PP^2\times\PP^2$. In suitable coordinates on $\PP^2_{x_1,x_2,x_3}\times\PP^2_{y_1,y_2,y_3}$, $X_{\CC}$ is given by the following equation (cf. \cite{GrossPopescu}):
$$
x_1x_2y_1^2+x_1^2y_1y_2+x_2x_3y_2^2+x_3^2y_1y_3+x_2^2y_2y_3+x_1x_3y_3^2=0. 
$$
Then $X_{\CC}$ is smooth, so it is irrational (see e.g. \cite{Alzati-Bertolini-1992a}), which contradicts to our assumption. 
\end{proof}

Thus, $\rho(V)=\iota(V)=1$. Using the classification of smooth Fano 3-folds, we get $g\in\{6,10,12\}$.

\begin{lemma}
\label{lemma:g-12}
One has $g\ne 12$.
\end{lemma}

\begin{proof}
Suppose that $g=12$. Let $E$ be a real surface in $X$, let $S$ be its $G$-orbit, and let $N$ be the number of iurreducible components of $S$. Then $S$ is a $\mathbb{Q}$-Cartier divisor, so $S$ is a Cartier divisor, since $X$ has terminal Gorenstein singularities. Then $S\sim a(-K_X)$ for $a\in\mathbb{Z}_{>0}$, and 
$$
N(-K_X)^2\cdot E=(-K_X)^2\cdot S=a(-K_X)^3=22a.
$$
Thus, since $N$ divides $|G|=168$, we conclude that $(-K_X)^2\cdot E$ is divisible by $11$. Therefore, for every complex surface $F\subset X_{\CC}$, its degree $\deg(F)=(-K_X)^2\cdot F$ is also divisible by $11$.

Now, using \cite[Lemma 8.2]{Prokhorov-g-12} or results presented in Section~\ref{subsection:almost-Fanos}, we see that $\mathrm{r}(X_{\CC})=1$. Hence, it follows from \cite{Prokhorov-G-Fanos} that $X$ is smooth. Then it follows from \cite{SashaYuraCostya} that the Hilbert scheme of conics on $X_{\CC}$ is isomorphic to $\PP^2$, and $G$ acts faithfully on it. This contradicts Lemma~\ref{corollary:PSL27-real-rational-surface}.
\end{proof}

Hence, $g=10$ or $g=6$. We will deal with these two cases in Sections \ref{subsection:g-10} and \ref{subsection:g-6}, respectively. 

\subsection{Fano 3-folds of genus $10$}
\label{subsection:g-10}

Let us use all assumptions, results and notations from Section~\ref{subsection:Gorenstein}. Suppose that $g=10$. Then it follows from Lemmas~\ref{lemma:reps} and \ref{lemma:g-7-8-9} that $|-K_X|$ has no $G$-invariant divisors. Moreover, if $X$ is smooth, then it follows from \cite[Corollary 4.3.5]{SashaYuraCostya} that there exists a group monomorphism $\mathrm{Aut}(X_{\CC})\hookrightarrow\mathrm{Aut}(A)$, where $A$ is an abelian surface. This contradicts Lemma~\ref{lemma:abelian}. Thus, we see that $X$ is singular.

Recall that $\rho(X_{\CC})=1$. If $X_{\CC}$ is $\mathbb{Q}$-factorial, it follows from \cite[Theorem 1.3]{Prokhorov-G-Fanos} that $X_{\CC}$ has at most two singular points, so each singular point of $X_{\CC}$ is fixed by the group $G$, which contradicts Lemma~\ref{lemma:fixed-point:21}. Hence, we see that $X_{\CC}$ is not $\mathbb{Q}$-factorial. Let $f\colon \widetilde{X}\to X_{\CC}$ be a $\mathbb{Q}$-factorialization. Then $-K_{\widetilde{X}}\sim f^*(-K_{X_{\CC}})$, and $\widetilde{X}$ is an almost Fano 3-fold such that $\rho(\widetilde{X})>1$. 

\begin{lemma}
\label{lemma:g-10-planes}
Let $\widetilde{E}$ be an irreducible surface in $\widetilde{X}_{\CC}$. Set $E=f(\widetilde{E})$. Then $\deg(E)\geqslant 3$, and $\deg(E)$ is divisible by $3$.
\end{lemma}

\begin{proof}
Let $S$ be the real $G$-irreducible surface such that $E$ is an irreducible component of $S_\CC$, and $N$ be the number of irreducible components of the surface $S_{\CC}$. Then $S\sim a(-K_X)$ for some positive integer $a$ such that $N\deg(E)=\deg(S)=a(-K_X)^3=18a$. Thus, since $N$ divides $2|G|=336$, we conclude that $\deg(E)$ is divisible by $3$.
\end{proof}

In particular, in the notations of Section~\ref{subsection:almost-Fanos}, the almost Fano 3-fold $\widetilde{X}$ does not contain \textit{planes}. Therefore, it follows from the results presented in Section~\ref{subsection:almost-Fanos} that there exists an extremal contraction $h\colon\widetilde{X}\to Z$ such that one of the following $3$ cases holds:
\begin{enumerate}
\item $h$ is birational, it contracts a surface to a curve, and $Z$ is an almost Fano 3-fold,
\item $h$ is a fibration into del Pezzo surfaces, $Z\simeq\PP^1$, and $\rho(\widetilde{X})=2$;
\item $h$ is a conic bundle, and $Z$ is a smooth del Pezzo surface.
\end{enumerate}

Let $D$ be the  effective divisor in $X_{\CC}$ spanned by all lines in $X_{\CC}$. Then $D$ is defined over $\mathbb{R}$, and it is $G$-invariant. 

\begin{lemma}
\label{lemma:Hilb-lines}
One has 
\begin{equation}
\label{eq:famly-lines}
D\sim a(-K_X),\qquad a\in\{2,3\},
\end{equation}
and $D$ is reduced and $G$-irreducible. 
\end{lemma}

\begin{proof}
If \eqref{eq:famly-lines} holds, then $D$ is reduced and $G$-irreducible, because $|-K_X|$ does not contain $G$-invariant divisors by Corollary~\ref{corollary:g-7-8-9}. 
To show \eqref{eq:famly-lines}, 
consider a smoothing $\mathfrak{X}\to \mathfrak{B}$ of $X_{\CC}$  (see \cite{Namikawa}).
Thus $ \mathfrak{B}$ is a smooth curve and the fiber over some point $o\in  \mathfrak{B}$ is isomorphic to $X_{\CC}$,
and the fiber over each point $b\in \mathfrak{B}\setminus \{o\}$ is a smooth Fano $3$-fold $\mathfrak{X}_b$ 
with $\rho(\mathfrak{X}_b)=\rho(X_{\CC})=1$, $\iota(\mathfrak{X}_b)=\iota(X)=1$, and $\g(\mathfrak{X}_b)=\g(X)=10$.
Let $\mathfrak{H}$ be the relative Hilbert scheme of lines in the fibers of $\mathfrak{X}/ \mathfrak{B}$ and let $\mathfrak{U}$ be the corresponding universal family. Thus, we have the following diagram
$$
\xymatrix@R=0.6em{
&\mathfrak{U}\ar[dl]_{p}\ar[dr]^{q}&\\
\mathfrak{X}\ar[dr]&&\mathfrak{H}\ar[dl]\\
&\mathfrak{B}&}
$$
The schemes  $\mathfrak{H}$ and $\mathfrak{U}$
are flat over $\mathfrak{B}$ of relative dimension $1$ and $2$, respectively.
Then $p$ is birational on the fibers  over $b\in \mathfrak{B}\setminus \{o\}$ and 
$p(\mathfrak{U}_b)$ is the divisor spanned by all lines in $\mathfrak{X}_b$.
Similarly, $p(\mathfrak{U}_o)=D$ is the divisor spanned by all lines in $X_\CC$.
Hence we have
$$
\deg(D)= D\cdot (-K_{X})^2\leqslant  p(\mathfrak{U})\cdot  \mathfrak{X}_o\cdot\big(-K_{\mathfrak{X}}\big)^2 =
p(\mathfrak{U})\cdot  \mathfrak{X}_b\cdot \big(-K_{\mathfrak{X}}\big)^2 =\deg\big(p(\mathfrak{U}_b)\big).
$$
Observe that $D$ is a Cartier divisor and $D\sim a(-K_X)$ for some positive integer $a$, because $\iota(X)=1$.
Thus, the anticanonical degree of $D$ is $(-K_{X})^2\cdot D=18a$, so it is enough to show that  
the anticanonical degree of the divisor $N$ spanned by all lines in general (smooth) member $Y= \mathfrak{X}_b$ equals
$54$. This is asserted by \cite[Theorem~4.2.7]{IP99}, so $a\leqslant 3$ as claimed. Since \cite[Theorem 4.2.7]{IP99} is presented without a proof, let us explain how to compute $a$. By Lemma 5 in Lecture 4 in \cite{Tyurin}, we know that $a=k-1$, where $k$ is the number of lines in $Y$ intersecting a general line in $Y$. To compute $k$, fix a general line $\ell\subset Y$. Recall from \cite[Theorem~4.3.3]{IP99} that we have the following commutative diagram:
$$
\xymatrix{
&&\widetilde{Q}\ar[dll]_{\alpha}\ar[rrd]_{\beta}\ar@{-->}[rrrr]^{\omega}&&&&\widetilde{Y}^\prime\ar[drr]^{\delta}\ar[lld]^{\gamma}&&\\
Q\ar@{-->}[rrrr]&&&&Y_0&&&&Y\ar@{-->}[llll]}
$$
where $Q$ is a smooth quadric 3-fold in $\PP^4$, $\alpha$ is a blow up of a smooth curve $Z\subset Q$ of degree $7$ and genus $2$, $\omega$ is a composition of flops, $Y_0$ is a Fano 3-fold with terminal Gorenstein singularities of genus $8$ anticanonically embedded in $\PP^9$, both $\beta$ and $\gamma$ are small birational morphisms, $\delta$ is a blow up of the line $\ell$, the map $Q\dasharrow Y_0$ is given by the linear subsystem in $|-K_Q|$ consisting of all surfaces passing through $Z$, and $Y\dasharrow Y_0$ is the linear projection from $\ell$. Then the curves contracted by $\gamma$ are strict transforms of the lines in $Y$ intersecting $\ell$. Let $S$ be the strict transform on $Q$ of the $\delta$-exceptional surface. Then $S$ is a smooth del Pezzo surface of degree $4$ that is cut out on $Q$ by another quadric hypersurface \cite[Theorem~4.3.3]{IP99}. Then $Z\subset S$, and $\alpha$ induces an isomorphism of $S$ with its proper transform $\widetilde{S}\subset\widetilde{Q}$. Furthermore, $\overline{S}=\beta(\widetilde{S})$ is a normal surface of degree $3$ in $Y_0$. Since $\widetilde{S}$ is a smooth del Pezzo surface, $\overline{S}$ is a smooth and $\overline{S}\simeq \mathbb{F}_1$. This implies that the morphism $\beta$ contracts $K_{\overline{S}}^2-K_{\widetilde{S}}^2=4$ disjoint smooth rational curves, and $\omega$ is a composition of Atiyah flops along them. Thus, the small morphism $\gamma$ also contracts $4$ disjoint curves, so $k=4$, and $a=3$ as claimed in \cite[Theorem 4.2.7]{IP99}.  
\end{proof}

Lemma~\ref{lemma:Hilb-lines} gives

\begin{lemma}
\label{lemma:g-10-birational}
The morphism $h$ is not birational.
\end{lemma}

\begin{proof}
Suppose that $h$ is birational. Let $\widetilde{E}$ be the $h$-exceptional surface, let $E=f(\widetilde{E})$, let $S$ be the real $G$-irreducible surface such that $E$ is an irreducible component of the surface $S_{\CC}$, and let $N$ be the number of irreducible components of the surface $S_{\CC}$. Then, as in the proof of Lemma~\ref{lemma:g-10-planes}, we get $S\sim a(-K_X)$ for some positive integer $a$. On the other hand, we know that $h$ contracts $\widetilde{E}$ to a curve, so $E$ is spanned by lines. Hence, $S\subset\mathrm{Supp}(D)$, so $S=D$ and $a=3$ by Lemma~\ref{lemma:Hilb-lines}. Thus, we have $N\deg(E)=54$. On the other hand, it follows from Corollary~\ref{cor:act} that one of the following two cases holds:
\begin{itemize}
\item either $N=1$, $E$ is $G$-invariant, and $E$ is real;
\item or $N=2$, $E$ is $G$-invariant, and $E$ is not real.
\end{itemize}
If $N=1$, then $\widetilde{E}\sim 3(-K_{\widetilde{X}})$, which is impossible, since $-K_{\widetilde{X}}$ is nef and big. Hence, we get $N=2$, the surface $E$ is $G$-invariant, but $E$ is not real. We also have $(-K_{\widetilde{X}})^2\cdot \widetilde{E}=\deg(E)=27$. Hence, it follows \eqref{equation:degrees} that
$(-K_Z)^3\geqslant (-K_{\widetilde{X}})^3+2(-K_{\widetilde{X}}^2)\cdot\widetilde{E}-2=18+2(-K_{\widetilde{X}})^2\cdot\widetilde{E}-2\geqslant 70$, which is a contradiction, since $(-K_Z)^3\leqslant 64$.
\end{proof}

Thus, we see that either $h$ is a conic bundle or $h$ is a del Pezzo fibration.

\begin{lemma}
\label{lemma:g-10-del-Pezzo-fibration}
The morphism $h$ is a conic bundle.
\end{lemma}

\begin{proof}
Suppose $h$ is a fibration into del Pezzo surfaces. Let $\widetilde{F}$ be its general fiber. Set \mbox{$F=f(\widetilde{F})$}. Then $F$ is a smooth del Pezzo surfaces of degree
$$
\deg(F)=\big(-K_{X_{\CC}}\big)^2F=\big(-K_{\widetilde{X}}\big)^2\cdot\widetilde{F}=(-K_F)^2\in\{1,2,3,4,5,6,8,9\}.
$$
But $G$ acts trivially on  $\mathrm{Cl}(X_{\CC})$ by Lemma~\ref{lemma:reps}, since $\rho(\widetilde{X})=\mathrm{r}(X_{\CC})=2$. Thus, both $f$ and $h$ are $G$-equivariant. Then the $G$-action on $Z\simeq\PP^1$ is trivial by Lemma~\ref{lemma:PSL27-curves}, so $F$ is $G$-invariant, and $G$ acts faithfully on $\widetilde{F}$, which gives $\deg(F)\in\{2,9\}$ by Lemma~\ref{lemma:PSL2-del-Pezzos}. Let $F^\prime$ be the surface obtained from $F$ by the complex conjugation. Then $F+F^\prime$ is real and $G$-invariant, so $F+F^\prime\sim b(-K_X)$ for some positive integer $b$. Then $2\deg(F)=2(-K_{X_\CC})^2\cdot F=(-K_X)^2\cdot(F+F^\prime)=b(-K_X)^3=18b$, which gives $\deg(F)=9$ and $b=1$. This is impossible, since $|-K_X|$ has no $G$-invariant surfaces.
\end{proof}

Thus, we see that $h\colon \widetilde{X}\to Z$ is a conic bundle, and $Z$ is a smooth del Pezzo surface.

\begin{lemma}
\label{lemma:g-10-Z-minimal}
Either $Z\simeq\PP^2$ or $Z\simeq\PP^1\times\PP^1$.
\end{lemma}

\begin{proof}
Suppose that $Z\not\simeq\PP^2$ and $Z\not\simeq\PP^1\times\PP^1$. Then $Z$ contains a $(-1)$-curve $C$.
Thus, it follows from Lemma~\ref{MMp-to-surface} that there is the following commutative diagram:
$$
\xymatrix{
&\widehat{X}\ar[dl]_{\varphi}\ar[dr]_{f^\prime}&&\widetilde{X}\ar@{-->}[ll]_{\chi}\ar[dl]^{f}\ar[dr]^{h}&\\
X^\prime\ar[drr]_{h^\prime}&&X&&Z\ar[dll]\\
&&Z^\prime&&}
$$
where $\widehat{X}$ is also an almost Fano 3-fold, $\chi$ is either an isomorphism or a composition of flops, $f^\prime$ is a pluri-anticanonical morphism, $\phi$ is a divisorial extremal contraction, $h^\prime$ is a conic bundle, and $Z\to Z^\prime$ is the contraction of the curve $C$. Here, $f^\prime$ is also a $\mathbb{Q}$-factorialization of $X$, so existence of such commutative diagram contradicts Lemma~\ref{lemma:g-10-birational}.
\end{proof}

We see that $\rho(Z)\leqslant 2$. Then $\rho(\widetilde{X})\leqslant 3$. Hence, as in the proof of Lemma~\ref{lemma:g-10-del-Pezzo-fibration}, we see that the group $G$ acts trivially on  $\mathrm{Cl}(X_{\CC})$. This implies that $f$ and $h$ are both $G$-equivariant. Moreover, the group $G$ acts faithfully on $Z$ by Lemma~\ref{lemma:PSL27-curves}. Then $Z\simeq\PP^2$ by Lemma~\ref{lemma:PSL2-del-Pezzos}, so $\rho(\widetilde{X})=2$. Hence, applying the involution of the complex conjugation, we obtain the following Sarkisov link:
$$
\xymatrix{
\widetilde{X}\ar[d]_{h}\ar[rrd]_{f}\ar@{-->}[rrrr]^{\chi}&&&&\widetilde{X}^\prime\ar[d]^{h^\prime}\ar[lld]^{f^\prime}\\
\PP^2&&X&&\PP^2}
$$
where $\chi$ is a non-biregular composition of flops, $f^\prime$ is a $\mathbb{Q}$-factorialization of the Fano 3-fold $X$,
and $h^\prime$ is a conic bundle, which is an extremal contraction.

Let $d$ be the degree of the discriminant curves of the conic bundles $h$ and $h^\prime$, let $\widetilde{H}$ be a general surface in $|h^*(\mathcal{O}_{\PP^2}(1))|$. Then $\widetilde{H}\cdot (-K_{\widetilde{X}})^2=12-d$. Set $H=f(\widetilde{H})$, and let $H^\prime$ be the surface obtained from $H$ by the complex conjugation. Then $H+H^\prime\sim \lambda(-K_{X})$ for some integer $\lambda\geqslant 1$. On the other hand, we have
$$
12-d=\widetilde{H}\cdot (-K_{\widetilde{X}})^2=H\cdot (-K_{X_\CC})^2=H^\prime\cdot (-K_{X_\CC})^2,
$$
which gives $2(12-d)=18\lambda$. Then $d=12-9\lambda$, so that either $d=3$ and $\lambda=1$. But $Z\simeq\PP^2$ has no $G$-invariant curves of degree $3$. The obtained contradiction shows that $g\ne 10$.

\subsection{Gushel--Mukai 3-folds}
\label{subsection:g-6}

Let us use assumptions, results and notations from Section~\ref{subsection:Gorenstein}. Now, we suppose that $g=6$, so $X$ is a rational real $G\mathbb{Q}$-Fano 3-fold of degree $(-K_X)^3=2g-2=10$, which is anticanonically embedded in $\PP^7$, and $X$ is an intersection of quadrics in $\PP^7$. Here, as always, we assume that $G=\PSL_2(\FF_7)$. In the following, we will use notations of Lemma~\ref{lemma:reps}.

\begin{lemma}
\label{lemma:g-6-G-invariant-K3}
The linear system $|-K_{X_{\CC}}|$ does not contain $G$-invariant divisors.
\end{lemma}

\begin{proof}
Suppose that there is a $G$-invariant surface $S\in|-K_{X_{\CC}}|$. Then $S$ is the only $G$-invariant surface in $|-K_{X_{\CC}}|$ by Lemma~\ref{lemma:g-7-8-9}, so $S$ is defined over $\mathbb{R}$. Moreover, it follows from Lemma~\ref{lemma:G-invariant-surfaces} that either $S_{\CC}$ is a smooth K3 surface, or $S$ is reducible over $\mathbb{R}$, and $G$ acts transitively on the set of its irreducible components. The latter option is impossible by Corollary~\ref{cor:act}, since $\deg(S)=(-K_X)^2\cdot S=10$. Hence, we see that $S_{\CC}$ is a smooth K3 surface. Furthermore, we know that the group $\mathrm{Pic}(S)$ is torsion free, and it follows from Lemma~\ref{lemma:K3} that $\mathrm{Pic}(S_{\CC})^G\simeq \mathbb{Z}$.
Moreover, $\mathrm{Pic}(S_{\CC})^G$ is generated by $-K_{X_{\CC}}\vert_{S_{\CC}}$ since $(-K_{X_{\CC}})^2\cdot S=10$.

On the other hand, we have $H^0(X_{\CC},\mathcal{O}_{X_{\CC}}(-K_{X_{\CC}}))=\mathbb{V}_1\oplus\mathbb{V}_7$ by Lemma~\ref{lemma:reps}. Then
$$
H^0(\PP^7,\mathcal{O}_{\PP^7}(2))=\mathbb{V}_1\oplus\mathbb{V}_1\oplus\mathbb{V}_6\oplus\mathbb{V}_6\oplus\mathbb{V}_7\oplus\mathbb{V}_7\oplus\mathbb{V}_8.
$$ 
Thus, it follows from \eqref{equation:quadrics} that $H^0(\PP^7, \mathcal{I}_{X}(2))=\mathbb{V}_6$, where $\mathcal{I}_X$ is the ideal sheaf of the 3-fold $X$. Since $X$ is projectively normal, we see that $|-2K_X|$ contains a pencil of $G$-invariant surfaces, which is generated by $2S$ and another real $G$-invariant surface, which we denote by $Q$. Then $S\not\subset\mathrm{Supp}(Q)$, because $S$ is the only $G$-invariant surface in $|-K_X|$.

Set $Z=Q\vert_{S}$. Then $Z\sim 2(-K_X)\vert_{S}$, so $Z^2=40$ on the surface $S$. Note that $Z_{\CC}$ is \mbox{$G$-irreducible}, because $\mathrm{Pic}^G(S_{\CC})$ is generated by $-K_{X_{\CC}}\vert_{S_{\CC}}$, and $|-K_{X_{\CC}}\vert_{S_{\CC}}|$ does not contain $G$-invariant curves. Moreover, since $Z_{\CC}$ is a curve of degree $20$, the curve $Z_{\CC}$ must be irreducible by Corollary~\ref{cor:act}. But its arithmetic genus is $21$, which contradicts Corollary~\ref{corolarry:Klein-action-on-singular-curves}.
\end{proof}

By Lemmas \ref{lemma:reps} and \ref{lemma:g-6-G-invariant-K3}, we have $H^0(X_{\CC},\mathcal{O}_{X}(-K_{X_{\CC}}))=\mathbb{V}_8$.
Now, as above, using GAP \cite{GAP4} and \eqref{equation:quadrics}, we see that 
\begin{equation}
\label{eq:GM:I}
H^0(\PP^7,\, \mathcal{I}_X(2))=\mathbb{V}_6, 
\end{equation}
where $\mathcal{I}_X$ is the ideal sheaf of the 3-fold $X$. 

\begin{lemma}
\label{lemma:non-factorial}
The 3-fold $X_{\CC}$ is not $\mathbb{Q}$-factorial.
\end{lemma}

\begin{proof}
Suppose that $X_{\CC}$ is $\mathbb{Q}$-factorial. Then $\mathrm{Cl}(X_{\CC})=\mathrm{Pic}(X_{\CC})=\mathbb{Z}[-K_X]$. Thus, if $S$ is a very general surface in $|-K_{X_{\CC}}|$, then $S$ is a smooth K3 surface, and $\mathrm{Pic}(S)=\mathbb{Z}[-K_X\vert_{S}]$ by \cite{RavindraSrinivas}. Hence, it follows from \cite{ArendSashaEmanuele,DebarreKuznetsov} that $X$ is a Gushel-Mukai 3-folds, which means that
$$
X=\widetilde{\Gr(2,5)} \cap\PP^7\cap Q\subset\PP^{10},
$$ 
where $\widetilde{\Gr(2,5)}$ is a cone over the Grassmannian $\Gr(2,5)\subset \PP^9$ and $Q$ is a quadric hypersurface. Moreover, it follows from \cite[Corollary 2.11]{DebarreKuznetsov} that $Q$ can be chosen to be $G$-invariant, so $H^0(\PP^7,\, \mathcal{I}_X(2))$ contains a one-dimensional subrepresentation of the group $G$, which contradicts \eqref{eq:GM:I}.
\end{proof}

\begin{remark}
\label{remark:GM}
In the proof of Lemma~\ref{lemma:non-factorial}, we can also obtain a contradiction as follows. If $X$ is a Gushel-Mukai 3-fold, it follows from  \cite{DebarreKuznetsov,ArendSashaEmanuele} that either there exists a $G$-equivariant double cover $\pi\colon X \to Y_5$ such that $Y_5$ is a del Pezzo 3-fold of degree $5$, or
$$
X=\Gr(2,5)\cap \PP^7\cap Q\subset\PP^9,
$$
where $\Gr(2,5)$ is the Grassmannian $\Gr(2,5)$, and $Q$ is a quadric hypersurface. The former case is impossible by  Lemma~\ref{lemma:DP}. In the latter case, it follows from \cite{DebarreKuznetsov} that the $G$-action lifts to the del Pezzo 4-fold $Y:=\Gr(2,5)\cap \PP^7$. If $Y$ is smooth, it follows from \cite[Theorem 6.6]{PiontkowskiVandeVen} that we have the following exact sequence of groups:
$$
1\longrightarrow
(\GG^{\mathrm{a}})^4\rtimes \GG^{\mathrm{m}} \longrightarrow   \Aut(Y) \longrightarrow  \PSL_2(\CC) \longrightarrow  1,
$$
which leads to a contradiction. Hence, $Y$ is singular. Then $\dim(\Sing(Y))\leqslant 1$, since the singularities of $X$ are isolated. Let $\mathcal{M}$ be the pencil  of skew-symmetric matrices corresponding to the pencil of hyperplane sections that cut out $Y$ in $\Gr(2,5)\subset \PP^9$. Then $\mathcal{M}$ contains a matrix $M$ that corresponds to a singular  hyperplane sections of  $\Gr(2,5)\subset \PP^9$, and $\rk(M)\leqslant 2$. Thus, there exists a hyperplane section $H$ of the Grassmannian $\Gr(2,5)\subset\PP^9$ that contains $Y$ such that $H$ is a Schubert divisor. Then $\Sing(H)$ is a plane, and the intersection $\Sing(H)\cap Y$ is a $G$-invariant line. But $\Sing(H)\cap Y\cap Q\subset\mathrm{Sing}(X)$, so $\Sing(H)\cap Y\cap Q$ is a $G$-invariant set that consists of at most two singular points of $X_{\CC}$, which contradicts Lemma~\ref{lemma:fixed-point:21}.
\end{remark}

Recall that $\rho(X_{\CC})=1$. On the other hand, the 3-fold $X_{\CC}$ is not $\mathbb{Q}$-factorial by Lemma~\ref{lemma:non-factorial}. In particular, we see that $X_{\CC}$ is singular.

\begin{lemma}
\label{lemma:g-6-planes}
Let $E$ be an irreducible (complex) surface in $X_{\CC}$. Then $\deg(E)\geqslant 5$, and $\deg(E)$ is divisible by $5$. Moreover, if $E$ is $G$-invariant, then $\deg(E)\geqslant 10$.
\end{lemma}

\begin{proof}
Let $S$ be the $G$-irreducible real surface such that $E$ is an irreducible component of the surface $S_\CC$, and $N$ be the number of irreducible components of the surface $S_\CC$. Then $S\sim a(-K_X)$ for some positive integer $a$. Hence, we have $N\deg(E)=\deg(S)=a(-K_X)^3=10a$.
Thus, since $N$ divides $2|G|=336$, we conclude that $\deg(E)$ is divisible by $5$.

Suppose that $E$ is $G$-invariant. If $E$ is real, then $S=E$ and $N=1$, so $\deg(E)$ is divisible by $10$.
If $E$ is not real, then $N=2$, so $\deg(E)=5a$, but $a\ne 1$ by Lemma~\ref{lemma:g-6-G-invariant-K3}, because $S$ is $G$-invariant. Hence, we see that $\deg(E)\geqslant 10$.
\end{proof}

Let $f\colon \widetilde{X}\to X_{\CC}$ be a $\mathbb{Q}$-factorialization of $X_{\CC}$. Then $\widetilde{X}$ is a complex $\mathbb{Q}$-factorial almost Fano 3-fold such that $\rho(\widetilde{X})>1$. Since $-K_{\widetilde{X}}\sim f^*(-K_{X_{\CC}})$, Lemma~\ref{lemma:g-6-planes} gives

\begin{corollary}
\label{corollary:g-6-planes}
If $\widetilde{E}$ is a surface in $\widetilde{X}_{\CC}$, then $\big(-K_{\widetilde{X}_{\CC}}\big)^2\cdot \widetilde{E}$ is positive and divisible by $5$.
\end{corollary}

Thus, using Corollary~\ref{corollary:g-6-planes} and applying results of Section~\ref{subsection:almost-Fanos}, either we get the diagram
\begin{equation}
\label{equation:MMP-0}
\xymatrix{
X_{\CC}&&\widetilde{X}\ar[ll]_{f}\ar[rr]^{\eta}&& Z}
\end{equation}
such that $\eta$ is an extremal non-divisorial contraction with $\dim(Z)\in\{1,2\}$, or we get the diagram
\begin{equation}
\label{equation:MMP}
\xymatrix{
&\widetilde{X}\ar[ld]_{f}\ar[rr]^{h_1}&&\widetilde{X}_1\ar[rr]^{h_2}&&\cdots\ar[rr]^{h_{n}}&&\widetilde{X}_n\ar[rd]^{\eta}&\\
X_{\CC}&&&&&&&& Z}
\end{equation}
such that each $\widetilde{X}_i$ is a $\mathbb{Q}$-factorial almost Fano 3-fold, each $h_i$ is an extremal birational contraction of a surface to a curve, and $\eta$ is an extremal non-divisorial contraction such that $\dim(Z)\in\{0,1,2\}$. Setting $\widetilde{X}_0=\widetilde{X}$, we may further assume that \eqref{equation:MMP} includes \eqref{equation:MMP-0} as a special case when $n=0$. Furthermore, if $\dim(Z)\in\{1,2\}$, then it follows from \cite{Prokhorov-g-12} that
\begin{enumerate}
\item either $\eta$ is a fibration into del Pezzo surfaces, $Z\simeq\PP^1$, and $\rho(\widetilde{X}_n)=2$;
\item or $\eta$ is a conic bundle, and $Z$ is a smooth del Pezzo surface.
\end{enumerate}
If $n\geqslant 1$, we let $h=h_n\circ \cdots \circ h_1$, and we let $\widetilde{E}_i$ be the strict transform on $\widetilde{X}$ of the $h_i$-exceptional surface for every $i\in\{1,\ldots,n\}$. If $n=0$, we let $h=\mathrm{Id}_{\widetilde{X}}$.

\begin{lemma}
\label{lemma:g-6-G-class-group}
The group $G$ acts non-trivially on $\mathrm{Cl}(X_{\CC})$.
\end{lemma}

\begin{proof}
Suppose that $G$ acts trivially on $\mathrm{Cl}(X_{\CC})$. Then \eqref{equation:MMP} is $G$-equivariant. Let us show that this leads to a contradiction.

Suppose that $\eta$ is a del Pezzo fibration. Then $G$ acts trivially on $Z$ by Lemma~\ref{lemma:PSL27-curves}. Let $F$ be a general fiber of the composition morphism $\phi\circ h\colon \widetilde{X}\to\PP^1$. Then it follows from the adjunction formula that $F$ is smooth del Pezzo surface, which implies that $(-K_F)^2\leqslant 9$. But Corollary~\ref{corollary:g-6-planes} gives $(-K_F)^2=(-K_{\widetilde{X}_{\CC}})^2\cdot F\geqslant 10$,
which is a contradicts. Hence, $\eta$ is not a del Pezzo fibration.

Now, we suppose that $\eta$ is a conic bundle. Then $Z$ is a smooth del Pezzo surface, general fiber of $\eta$ is isomorphic to $\PP^1$, so $G$ acts faithfully on $Z$, but $G$ acts trivially on $\mathrm{Pic}(Z)$, because $G$ acts trivially on $\mathrm{Cl}(X_{\CC})$. Hence, we have $Z\simeq\PP^2$ by Lemma~\ref{lemma:PSL2-del-Pezzos}.

Let $L$ be a general line in $Z$, and let $H=(\eta\circ h)^*(L)$. Then $H$ is a smooth surface, and $\eta\circ h$ induces a conic bundle $\phi\colon\vert_{H}\colon H\to L$. Moreover, we have $-K_H\sim-K_{\widetilde{X}}\vert_{H}-C$, where $C$ is a general fiber of the conic bundle $\phi$. Hence, we see that $-K_H$ is $\phi$-ample. In particular, each singular fiber of the conic bundle $\phi$ is a union of two smooth rational curves that intersect transversally at one point. Note that $\phi$ has $8-K_H^2$ singular fibers. On the other hand, we have
$$
(-K_{\widetilde{X}})^2\cdot H=(-K_H+C)^2=4+K_H^2\leqslant 12.
$$
Hence, by Corollary~\ref{corollary:g-6-planes}, either $(-K_{\widetilde{X}_{\CC}})^2\cdot H=10$ and $K_H^2=6$, or $(-K_{\widetilde{X}_{\CC}})^2\cdot H=5$ and $K_H^2=1$. Thus, we see that $8-K_H^2\in\{2,7\}$.

If $n=0$, let $\Delta$ be the discriminant curve of the conic bundle $\eta$. If $n\geqslant 1$, let $\Delta$ be the union of the discriminant curve of the conic bundle $\eta$ and all curves in $Z$ that are images of the $h$-exceptional surfaces. Then $\Delta$ is $G$-invariant, and $L\cap \Delta$ is the discriminant divisor of the conic bundle $\phi$.  Thus, we see that $\deg(\Delta)=8-K_H^2\in\{2,7\}$, which is impossible, since $Z$ does not contains $G$-invariant curves of degree $2$ or $7$ by Remark~\ref{remark:invariants}.

Thus, $\widetilde{X}_n$ is a Fano 3-fold with $\mathbb{Q}$-factorial terminal Gorenstein singularities, and $\rho(\widetilde{X}_n)=1$. In particular, we have $n\geqslant 1$ It follows from Corollary~\ref{corollary:g-6-planes} that $(-K_{\widetilde{X}})^2\cdot\widetilde{E}_i\geqslant 10$ for each $h$-exceptional surface $\widetilde{E}_i$. Recall from Section~\ref{subsection:almost-Fanos}, that each $h_i$ is a blow up of the ideal sheaf of a (possibly singular) irreducible curve $C_i\subset\widetilde{X}_i$, which is contained in the smooth locus of   $\widetilde{X}_i$. Since $h_i$ is $G$-equivariant, we see that the curve $C_i$ is $G$-invariant. Moreover, the group $G$ acts faithfully on $C_i$ by Corollary~\ref{cor:Klein-action-on-3-folds}. Thus, if $C_i$ is smooth, we have its genus is at least $3$ by Lemma~\ref{lemma:PSL27-curves}. Moreover, if $C_i$ is singular, it has at most planar singularities, so $\mathrm{p}_a(C_i)\geqslant 24$ by Corollary~\ref{corolarry:Klein-action-on-singular-curves}. This implies that each $C_i$ is smooth. Indeed, if $C_1$ is singular, then it follows from Section~\ref{subsection:almost-Fanos} that
$$
64\geqslant10+2(-K_{\widetilde{X}})^2\cdot\widetilde{E}_1+2\mathrm{p}_a(C_1)-2\geqslant 76,
$$
which is impossible, since $(-K_{\widetilde{X}_1})^3\leqslant 64$ by Remark~\ref{remark:smoothing}. Similarly, using Lemma~\ref{lemma:MMP-no-planes}, we see that every curve $C_i$ is smooth. In particular, if $\widetilde{X}_n$ is smooth, then $\widetilde{X}$ is also smooth. Similarly, using results presented in Section~\ref{subsection:almost-Fanos}, we see that
$$
64\geqslant\big(-K_{\widetilde{X}_n}\big)^3\geqslant \big(-K_{\widetilde{X}}\big)^3+\sum_{i=1}^{n}\big(2(-K_{\widetilde{X}})^2\cdot\widetilde{E}_i+2\g(C_i)-2\big)\geqslant \big(-K_{\widetilde{X}}\big)^3+24n=10+24n,
$$
where $\g(C_i)$ is the genus of the curve $C_i$. This gives $n=1$ or $n=2$.

Suppose that $n=2$. Then $(-K_{\widetilde{X}_2})^3\geqslant 10+24n=58$, so $\widetilde{X}_2\simeq\PP^3$ by Remark~\ref{remark:smoothing}. Moreover, we have
$$
64\geqslant \big(-K_{\widetilde{X}}\big)^3+2(-K_{\widetilde{X}})^2\cdot\widetilde{E}_1+2\g(C_1)-2+2(-K_{\widetilde{X}})^2\cdot\widetilde{E}_2+2\g(C_2)-2\geqslant 46+2\big(\g(C_1)+\g(C_2)\big).
$$
so $\g(C_1)+\g(C_2)\leqslant 9$. Hence, we see that $\g(C_1)=\g(C_2)=3$ by Lemma~\ref{lemma:PSL27-curves}. On the other hand, it follows from  Section~\ref{subsection:almost-Fanos} that $68-8\deg(C_2)=(-K_{\widetilde{X}_1})^3=14+2(-K_{\widetilde{X}})^2\cdot\widetilde{E}_1\geqslant 24$,
which implies that $\deg(C_2)\leqslant 4$. But $C_2$ is a $G$-invariant curve in $\PP^3$. Now, it follows from \cite{BlancLamy} or \cite[Lemma 3.7]{CheltsovShramovKlein} that $C_2$ is a plane quartic curve, which implies that $X_1$ is not an almost Fano 3-fold, which is a contradiction.

Hence, we see that $n=1$. Then, since $(-K_{\widetilde{X}})^2\cdot\widetilde{E}_1\geqslant 10$, we have
$$
(-K_{\widetilde{X}_1})^3=10+2(-K_{\widetilde{X}})^2\cdot\widetilde{E}_1+2\g(C_1)-2\geqslant 28+2\g(C_1),
$$
which gives $\g(C_1)\leqslant 18$, and $(-K_{\widetilde{X}})^2\cdot\widetilde{E}_1\in\{10,15,20,25\}$, because $(-K_{\widetilde{X}})^2\cdot\widetilde{E}_1$ is divisible by $5$ by Corollary~\ref{corollary:g-6-planes}. Then $\g(C_1)\in\{3,8,10,15,17\}$ by Lemma~\ref{lemma:PSL27-curves}, and it follows from Remark~\ref{remark:smoothing} that
$$
(-K_{\widetilde{X}_1})^3\in\big\{2,4,6,8,10,12,14,16,18,22,24,32,40,54,64\big\}.
$$
Moreover, $(-K_{\widetilde{X}_1})^3\ne 54$, since otherwise $\widetilde{X}_1$ would be a quadric in $\PP^4$ by Remark~\ref{remark:smoothing}, which contradicts Remark~\ref{remark:P3-quadric}. Now, going through all remaining possibilities, we get $(-K_{\widetilde{X}_1})^3=64$, and either $(-K_{\widetilde{X}})^2\cdot\widetilde{E}_1=20$ and $\g(C_1)=8$, or $(-K_{\widetilde{X}})^2\cdot\widetilde{E}_1=25$ and $\g(C_1)=3$. Thus, $\widetilde{X}_1\simeq\PP^3$ by Remark~\ref{remark:smoothing}. Moreover, it follows from  Section~\ref{subsection:almost-Fanos} that
$$
10=(-K_X)^3=(-K_{\widetilde{X}})^3=(-K_{\widetilde{X}_1})^3 + 2(-K_{\widetilde{X}_1})\cdot C_1+2\g(C)-2=62-8\deg(C_1)g+2\g(C),
$$
which implies that $4\deg(C_1)=26+\g(C_1)\in\{29,34\}$, which is a contradiction.
\end{proof}

Since the group $G\simeq\PSL_2(\FF_7)$ acts non-trivially on $\mathrm{Cl}(X_{\CC})$, it follows from Lemma~\ref{lemma:reps} that the rank of the group $\mathrm{Cl}(X_{\CC})$ is at least $7$. Thus, we see that $\rho(\widetilde{X})\geqslant 7$.

\begin{lemma}
\label{lemma:g-6-fibration}
The morphism $\eta$ is not a del Pezzo fibration. If $\eta$ is a conic bundle, then $Z\simeq\PP^2$.
\end{lemma}

\begin{proof}
Suppose that one of the following cases holds:
\begin{enumerate}
\item either $\eta$ is a del Pezzo fibration,
\item of $\eta$ is a conic bundle and $Z\not\simeq\PP^2$.
\end{enumerate}
In the first case, we let $\pi=\eta\circ h$. In the second case, there exists a conic bundle $\theta\colon Z\to\PP^1$, and we let $\pi=\theta\circ\eta\circ h$. In both cases, let $F$ be a general fiber of the morphism $\pi\colon \widetilde{X}\to\PP^1$. Then $F$ is smooth, and, by adjunction formula, we have
$-K_{F}\sim -K_{\widetilde{X}}\vert_{F}$, which implies that $F$ is a smooth del Pezzo surface of degree $(-K_F)^2=(-K_{\widetilde{X}_{\CC}})^2\cdot F$. Thus, it follows from that Corollary~\ref{corollary:g-6-planes} and classification of smooth del Pezzo surfaces that $(-K_F)^2=5$.

Now, let $F^\prime$ be any scheme fiber of the morphism $\pi$. Then $5=(-K_{\widetilde{X}})^2\cdot F=(-K_{\widetilde{X}})^2\cdot F^\prime$, so, applying Corollary~\ref{corollary:g-6-planes} to irreducible components of the fiber $F^\prime$, we conclude that $F^\prime$ is irreducible and reduced. Then  $\rho(\widetilde{X})\leqslant\rho(F)+1=6$, which is a contradiction, since  $\rho(\widetilde{X})\geqslant 7$.
\end{proof}

Either $Z$ is a point and $\rho(\widetilde{X}_n)=1$, or $\eta$ is a conic bundle, $Z\simeq\PP^2$ and $\rho(\widetilde{X}_n)=2$. Then $n\geqslant 5$.

\begin{lemma}
\label{lemma:g-6-conic-bundle}
The morphism $\eta$ is not a conic bundle.
\end{lemma}

\begin{proof}
Suppose that $\eta$ is a conic bundle. Then $Z=\PP^2$. Let $L$ is a general line in $Z$, let $L^\prime$ be some line in $Z$, let $H=(\eta\circ h)^*(L)$ and $H^\prime=(\eta\circ h)^*(L^\prime)$. Then $H$ is irreducible, and it follows from Corollary~\ref{corollary:g-6-planes} that $(-K_{\widetilde{X}})^2\cdot H^\prime=(-K_{\widetilde{X}})^2\cdot H>0$, and $(-K_{\widetilde{X}})^2\cdot H$ is divisible by $5$. Moreover, if $(-K_{\widetilde{X}})^2\cdot H=5$, then applying  Corollary~\ref{corollary:g-6-planes} to irreducible components of $H^\prime$, we see that $H^\prime$ is reduced and irreducible. Similarly, if $(-K_{\widetilde{X}})^2\cdot H=10$, then $H^\prime$ can have at most two irreducible components.

Now, we observe that $H$ is a smooth surface, and $\eta\circ h$ induces a conic bundle $\phi\vert_{H}\colon H\to L$. By adjunction formula, we have $-K_H\sim-K_{\widetilde{X}}\vert_{H}-C$, where $C$ is a general fiber of the conic bundle $\phi$. Hence, we see that $-K_H$ is $\phi$-ample. In particular, each singular fiber of the conic bundle $\phi$ is a union of two smooth rational curves that intersect transversally at one point. Let $k=8-K_H^2$. Then $\phi$ has $k$ singular fibers, and $(-K_{\widetilde{X}})^2\cdot H=(-K_H+C)^2=4+K_H^2\leqslant 12$. Hence, it follows from Corollary~\ref{corollary:g-6-planes} that one of the following two cases holds:
\begin{enumerate}
\item either $(-K_{\widetilde{X}})^2\cdot H=10$, $K_H^2=6$, $k=2$;
\item or $(-K_{\widetilde{X}})^2\cdot H=5$, $K_H^2=1$, $k=7$.
\end{enumerate}

Suppose that $(-K_{\widetilde{X}})^2\cdot H=10$, $K_H^2=6$, $k=2$. Then $H^\prime$ can have at most two irreducible components for any choice of the line $L^\prime\subset Z$. On the other hand, if $\eta\circ h(\widetilde{E}_i)$ is a curve in $Z$, then the fiber of the conic bundle $\phi$ over every point of the intersection $\eta\circ h(\widetilde{E}_i)$ is a singular fiber of $\phi$. Moreover, if $\eta\circ h(\widetilde{E}_i)$ and $\eta\circ h(\widetilde{E}_j)$ are both curves for $i\ne j$, then $\eta\circ h(\widetilde{E}_i)\ne\eta\circ h(\widetilde{E}_j)$. Hence, since $k=2$ and $n\geqslant 5$, we conclude that at least $3$ surfaces among $\widetilde{E}_1,\ldots,\widetilde{E}_n$ are mapped by $\eta\circ h$ to points in $Z$. Without loss of generality, we may assume that $\eta\circ h(\widetilde{E}_1)$ and $\eta\circ h(\widetilde{E}_2)$ are points. Now, choosing $L^\prime$ to be a line in $Z$ that contains these points, we see that $\mathrm{Supp}(H^\prime)$ contains both surfaces $\widetilde{E}_1$ and $\widetilde{E}_2$, and $H^\prime\ne \widetilde{E}_1+\widetilde{E}_2$, which is a contradiction.

Hence, we conclude that $(-K_{\widetilde{X}})^2\cdot H=5$, $K_H^2=1$, $k=7$. Then $H^\prime$ is irreducible and reduced for any choice of the line $L^\prime\subset Z$. Thus, arguing as in the previous case, we see that $\eta\circ h(\widetilde{E}_i)$ is a curve for every $i\in\{1,\ldots,n\}$, and at least one of these curves is must be line in $Z$. Now, choosing $L^\prime$ to be this line, we see that $H^\prime$ is reducible, which is a contradiction.
\end{proof}

Thus, $Z$ is a point, so $\rho(\widetilde{X}_{n})=1$ and $n=\rho(\widetilde{X})-1\geqslant 7$. Moreover, we have $(-K_{\widetilde{X}_n})^3\leqslant 64$, since $\widetilde{X}_n$ is an almost Fano 3-fold. Hence, it follows from Section~\ref{subsection:almost-Fanos} and Corollary~\ref{corollary:g-6-planes} that
$$
64\geqslant\big(-K_{\widetilde{X}_n}\big)^3\geqslant \big(-K_{\widetilde{X}}\big)^3+\sum_{i=1}^{n}\big(2(-K_{\widetilde{X}})^2\cdot\widetilde{E}_i-2\big)\geqslant \big(-K_{\widetilde{X}}\big)^3+8n=10+8n,
$$
since each $(-K_{\widetilde{X}})^2\cdot\widetilde{E}_i$ is a positive integer divisible by $5$. Thus, we conclude that $n=6$, $\rho(\widetilde{X})=7$ and
$(-K_{\widetilde{X}})^2\cdot\widetilde{E}_i=5$ for $i\in\{1,\ldots,6\}$. Let $\Gamma$ be a subgroup in $G$ such that $\Gamma\simeq\mumu_7\rtimes\mumu_3$.

\begin{lemma}
\label{lemma:g-6-final}
The group $\mathrm{Cl}(X_{\CC})^\Gamma$ is of rank $1$
\end{lemma}

\begin{proof}
Observe that $\mathrm{Cl}(X_{\CC})\otimes_{\mathbb{Z}}\mathbb{R}$ is a faithful $7$-dimensional real representation of the group $G$, which has trivial one-dimensional subrepresentation corresponding to the divisor class $-K_{X_{\CC}}$. Therefore, it follows from Lemma~\ref{lemma:reps} that either $\mathrm{Cl}(X_{\CC})\otimes_{\mathbb{Z}}\mathbb{C}$ splits as a sum of a one-dimensional representation and an irreducible $6$-dimensional representation of the group $G$, or $\mathrm{Cl}(X_{\CC})\otimes_{\mathbb{Z}}\mathbb{C}$ splits as a sum of a one-dimensional representation and two irreducible complex-conjugate $3$-dimensional representations of the group $G$. In both cases, we compute that $\mathrm{Cl}(\widetilde{X})\otimes_{\mathbb{Z}}\mathbb{C}$ contains a unique one-dimensional subrepresentation of the group $\Gamma$, which implies the required assertion.
\end{proof}

Set $E_1=f(\widetilde{E}_1)$, let $S$ be the $\Gamma$-orbit of the surface $E_1$, and let $N$ be the number of irreducible components of the surface $S$. Then it follows from Lemma~\ref{lemma:g-6-final} that $S\sim a(-K_X)$ for some positive integer $a$. Therefore, we have $N\deg(E_1)=\deg(S)=a(-K_X)^3=10a$. Thus, since $N$ divides $|\Gamma|=21$, we conclude that $\deg(E_1)$ is divisible by $10$. But $\deg(E_1)=(-K_{\widetilde{X}})^2\cdot\widetilde{E}_1=5$, which is absurd. The obtained contradiction completes the proof of Theorem~\ref{theorem:PSL27-real-Gorenstein}.

\section{Non-Gorenstein case}
\label{section:non-Gorenstein}
This section is devoted to the proof of the following result.
\begin{theorem}
\label{theorem:PSL27-real-non-Gorenstein}
Let $X$ be a real non-Gorenstein $G\QQ$-Fano $3$-fold
such that $X$ is rational over $\RR$.
Then $G\not\simeq \PSL_2(\mathbf{F}_7)$.
\end{theorem}

Thus, throughout this section, we will always assume that $X$ is a real non-Gorenstein $G\QQ$-Fano $3$-fold with $G=\PSL_2(\mathbf{F}_7)$.
We will further assume that $X$ is  rational over $\RR$, and the canonical class $K_X$ is not Cartier.
Let us seek for a contradiction.

\subsection{Non-Gorenstein singular points}
\label{subsection:non-Gorenstein-singularities}

\begin{proposition}
\label{proposition:1-2-points}
Any non-Gorenstein singularity of $X_\CC$ is either a cyclic quotient singularity of type $\frac{1}{2}(1,1,1)$, or
a moderate singularity of index $2$ and axial weight $2$ (see Definition \ref{def:simple:HQ}).
\end{proposition}

\begin{proof}
Let $\Sigma$ be the set of all non-Gorenstein points,
let $\Sigma=\Sigma_1\cup \cdots\cup \Sigma_M$ be the orbit decomposition,
and let $N_i := |\Sigma_i|$.
For a point $P_i \in  \Omega_{i}$, let $r_i$ be its index and  let $Q_{i,j}\in \B(X_\CC)$, $j = 1, \dots , k_i$ be all ``virtual'' points in the basket over $P_i$, and let $r_{i,j}$ be the index of $Q_{i,j}$. We may assume that $r_{i}=r_{i,1}$ (see Construction \ref{construction:basket}).
By Corollary~\ref{corollary:BM-24}, we have
\begin{equation}
\label{eq:KB}
\sum_{i=1}^M N_i \sum_{j=1}^{k_i} \left(r_{i,j}-\frac 1 {r_{i,j}}\right) <24.
\end{equation}
Assume that the singularity $P_i\in X$ is worse than cyclic quotient of index $2$.
By Lemma~\ref{lemma:fixed-point:21}\  $P_i$ is not a fixed point, hence by
Corollary~\ref{cor:act} we have $N_i\geqslant 7$.
If  $r_i\geqslant 3$, then the only possibility is $r_i=3$ and $k_i=1$, i.e.
$P_i$ is a cyclic quotient singularity of index $3$. Moreover, $N_i=7$ or $8$,
hence the stabilizer of $P_i$ is isomorphic to either $\mathfrak{S}_4$ or $\mumu_7\rtimes\mumu_3$.
This contradicts Lemmas \ref{lemma:fixed-point:21}
and \ref{lemma:fixed-point:S4}.
Hence $r_i=2$. Then $M=1$,  $k_i=2$ and $N_i=7$. By Lemma~\ref{lemma:Klein-subgroups}
the stabilizer $G_{P_i}$ of $P_i$ is isomorphic to $\mathfrak{S}_4$.
Then $P_i\in X$ is a moderate singularity by Lemma~\ref{lemma:fixed-point:S4}.
Its  axial weight equals $k_i=2$.
This completes the proof of Proposition \ref{proposition:1-2-points}.
\end{proof}

\begin{corollary}
\label{corollary:1-2-points-2K}
The divisor $2K_{X}$ is Cartier.
\end{corollary}

\begin{corollary}
\label{corollary:1-2-points}
Let $\Sigma$ be the set of non-Gorenstein singular points of $X_\CC$.
Then all the points in $\Sigma$ are of the same analytic type. Moreover, they of index $2$
and one of the following holds:
\begin{table}[H]\renewcommand{\arraystretch}{1.5}
\begin{tabular}{|p{0.17\textwidth}|p{0.17\textwidth}|p{0.55\textwidth}|}
\hline
$|\Sigma|$ &\rm Type&$\Sigma$
\\\hline 
$7$&moderate\newline  of axial weight $2$& $G$-orbit of real  points
\\\hline
$2k$,\ $1\leqslant k\leqslant 7$&$\frac12(1,1,1)$&$k$ pairs of complex conjugate $G$-fixed points
\\ \hline
$7+2k$, $0\leqslant k\leqslant 4$&$\frac12(1,1,1)$&union of $k$ pairs of complex conjugate $G$-fixed points and  $G$-orbit of length $7$ of a real point
\\ \hline
$14$&$\frac12(1,1,1)$&union of two $G$-orbits of length $7$ of real points
\\ \hline
$14$&$\frac12(1,1,1)$&union of two complex conjugate $G$-orbits of length $7$
\\ \hline
$14$&$\frac12(1,1,1)$& $G$-orbit of a (possibly complex) point
\\\hline
\end{tabular}
\caption{}
\label{tab0}
\end{table}
\end{corollary}

\begin{proof}
It follows from \eqref{eq:KB} that $|\Sigma|<16$, so the assertion follows from Proposition \ref{proposition:1-2-points} and Corollaries \ref{cor:act} and \ref{corollary:fixed-point}.
\end{proof}

\subsection{Non-empty anticanonical linear system}
\label{subsection:anticanonical-system-is-empty}

The goal of this subsection is to prove

\begin{proposition}
\label{proposition:non-empty}
The linear system $|-K_X|$ is not empty.
\end{proposition}

Suppose that $|-K_X|$ is empty. Let us seek for a contradiction. 
From Corollaries~\ref{corollary:BM-RR}  and \ref{corollary:1-2-points} we obtain the following possibilities:
\begin{table}[H]\renewcommand{\arraystretch}{1.5}
\begin{tabular}{|c|c|l|c|}
\hline
&$(-K_X)^3$ & \multicolumn{1}{c|}{ non-Gorenstein points of $X_{\CC}$} & $\dim(|-2K_X|)$
\\\hline
\nr\label{tab:1} &$\frac{1}{2}$ & $13$ cyclic quotient singularities of type $\frac{1}{2}(1,1,1)$
& $2$
\\
\hline\nr\label{tab:2} &$1$ & $7$ moderate singularities of index $2$ and axial weight $2$& $3$
\\
\hline\nr\label{tab:3} &$1$& $14$ cyclic quotient singularities of type $\frac{1}{2}(1,1,1)$ & $3$
\\
\hline\nr\label{tab:4} &$\frac{3}{2}$ & $15$ cyclic quotient singularities of type $\frac{1}{2}(1,1,1)$
& $4$
\\\hline
\end{tabular}
\caption{}
\label{tab}
\end{table}
Since $G$ has no non-trivial real representations of dimension less than $6$, we see that $G$ acts trivially on $|-2K_X|$. This means that every surface in $|-2K_X|$ is $G$-invariant.

Let $S$ be a general surface in $|-2K_X|$. Then the action of the group $G$ on $S$ is faithful. Moreover,  $(X,S)$ is purely log terminal  by \cite[Main Theorem]{Florin}. By \cite[\S~3]{Shokurov92} or \cite[Theorem~17.6]{Utah} the surface
$S$ is geometrically irreducible, normal, and $S$ has Kawamata log terminal singularities.

\begin{lemma}
\label{lemma:empty-bad-points}
Suppose that $S_\CC$ contains a non-Gorenstein singular point $P$ of the 3-fold $X_{\CC}$. Then $P$ is not $G$-fixed, and $S_{\CC}$ has cyclic quotient singularity of type $\frac{1}{4}(1,1)$ at the point~$P$. 
\end{lemma}

\begin{proof}
Let $G_P$ be the stabilizer in $G$ of the point~$P$.  If the $G$-orbit of $P$ has length $7$, then $G\simeq\mathfrak{S}_4$ by Corollary~\ref{cor:act}. Similarly, if the $G$-orbit of $P$ has length $14$, then $G\simeq\mathfrak{A}_4$. 

First, we suppose that $X$ has cyclic quotient singularity of type $\frac{1}{2}(1,1,1)$ at~$P$. Let $g\colon\widetilde{X}\to X_{\CC}$ be the blow up of the point $P$, let $E$ be the $g$-exceptional divisor. Then $g$ is $G_P$-equivariant, $G_P$ acts faithfully on $E\simeq\PP^2$, $\widetilde{X}$ is smooth near $E$, and $E\vert_{E}\sim_{\mathbb{Q}}\mathcal{O}_{\PP^2}(-2)$. Let $\widetilde{S}$ be the strict transform on $\widetilde{X}$ of the surface $S$. Then and $\widetilde{S}\sim_{\QQ} f^*(S)-mE$ for some integer $m\geqslant 1$, and 
$$
K_{\widetilde{X}}+\widetilde{S}\sim_{\mathbb{Q}} g^*(K_X+S)+\Big(\frac{1}{2}-m\Big)E,
$$
so $m=1$, because $(X,S)$ has purely log terminal singularities. Thus, $\widetilde{S}\vert_{E}$ is a $G_P$-invariant conic, and $G_P\ne G$ by  Remark~\ref{remark:invariants}. Hence, $P$ is not fixed by $G$. Thus, either $G_P\simeq\mathfrak{S}_4$ or $G_P\simeq\mathfrak{A}_4$. In both case, the group $G_P$ does not fix points in $E$, so the conic $\widetilde{S}\vert_{E}$ is smooth. This implies that $\widetilde{S}$ is smooth along $\widetilde{S}\vert_{E}$. Since $(E\vert_{\widetilde{S}})^2=-4$, we conclude that $S_{\CC}$ has cyclic quotient singularity of type $\frac{1}{4}(1,1)$ at the point~$P$. 

To complete the proof, we may assume that $S_{\CC}$ has moderate singularity of index $2$ and axial weight $2$ at the point~$P$. Then $G_P\simeq\mathfrak{S}_4$.
In this case, there exist $G_P$-equivariant birational morphism (often called the Kawamata blowup \cite{Kawamata-1992-e-app})  $f\colon\widetilde{X}\to X$ whose  
exceptional locus is an irreducible surface  $E\simeq\PP(1,1,4)$ with discrepancy $a(X,E)=\frac{1}{2}$. In fact, the morphism $f$ is a weighted blowup of $(P\in X)$ with weights $\frac{1}{2}(1,1,1,2)$ in suitable coordinates. Set $\widetilde{P}=\mathrm{Sing}(E)$. Then $\widetilde{P}=\mathrm{Sing}(\widetilde{X})\cap E$, and $\widetilde{X}$ has cyclic quotient singularity of type $\frac{1}{2}(1,1,1)$ at $\widetilde{P}$. We have $\widetilde{S}\sim f^*(S)-mE$ for some positive integer $m$, so  
$$
K_{\widetilde{X}}+\widetilde{S}\sim_{\mathbb{Q}} f^*(K_X+S)+\Big(\frac{1}{2}-m\Big)E,
$$
which gives $m=1$, since the pair $(X,S)$ is purely log terminal. 
Let $\ell$ be the ample generator of $\mathrm{Cl}(E)\simeq \mathbb{Z}$, and let $L=\widetilde{S}|_E$.
Since the singularities of $\tilde X$ are isolated,  we have  $L\sim E|_E\sim 4\ell$ and 
$$
6\ell \sim K_E\sim (K_{\widetilde{X}}+E)|_E\sim_{\mathbb{Q}} \frac32 E|_E.
$$
In particular, $L$ is an (effective) Cartier divisor on $E\simeq \PP(1,1,4)$. Note that the group $G_P\simeq\mathfrak{S}_4$ faithfully acts on $E\simeq \PP(1,1,4)$, and the curve $L$ is $G_P$-invariant. Let $\psi\colon E\dashrightarrow \PP^1$ be the natural projection. Then $\psi$ is equivariant, and $G_P$ acts faithfully on $\mathbb{P}^1$. 

We claim that $L$ is reduced and irreducible. Indeed, otherwise $L$ has a component $L_1$ such that $L_1\sim \ell$. But then $\psi(L_1)$ is a point on $\PP^1$ whose orbit has length $\leqslant 4$, which is a contradiction.  Hence $L$ reduced and irreducible.  In this case $L$ must be smooth. In particular, it does not pass through $\widetilde{P}$. Then $\widetilde{S}$ must be smooth near $E$. This means that the induced morphism $\widetilde{S}\to S$ is the minimal resolution of $P$, and  $L$ its irreducible exceptional divisor. Finally, we compute
$$
(L\cdot L)_{\widetilde{S}}=(L\cdot E\vert_{\widetilde{S}})_{\widetilde{S}}=L\cdot E=(L\cdot E\vert_{E})_{E}=-(4\ell \cdot 4\ell)_{E}=-4,
$$
so $L$ is a  $(-4)$-curve in $\widetilde{S}$. Hence, $(P\in S)$ is a cyclic quotient singularity of type $\frac{1}{4}(1,1)$.
\end{proof}

In the proof of the following lemma we use Lemma~\ref{lemma:empty-bad-points}, but, in fact, we do not need this, since the proof only uses the fact that $S$ has Kawamata log terminal $T$-singularities \cite{Janos-Nick}.

\begin{lemma}
\label{lemma:empty-Du-Val}
The surface $S_{\CC}$ does not contain non-Gorenstein singularities of the 3-fold $X_{\CC}$.
\end{lemma}

\begin{proof}
Suppose that $S_{\CC}$ contains a non-Gorenstein singular point of $X_{\CC}$. By Lemma~\ref{lemma:empty-bad-points}, $S_\CC$ has cyclic quotient singularities of type $\frac{1}{4}(1,1)$ at every non-Gorenstein singular point of $X_{\CC}$, and $S_{\CC}$ does not contain $G$-fixed non-Gorenstein singular points of $X_{\CC}$. Moreover, all other singular points of the surface $S_{\CC}$ are Du Val, since they are Gorenstein. In particular, $2K_S$ is an ample Cartier divisor, since  $K_S\sim -K_X\vert_{S}$ by the adjunction formula.

Suppose that $(-K_X)^3\ne\frac{1}{2}$. Then we can apply \cite[Theorem 1]{KawachiMasec} with $M=K_S$. Since the Kawachi's invariant of every non-Du Val point of the surface $S$ equals $1$ (see \cite[Theorem 2]{KawachiMasec}), it follows from \cite[Theorem 1]{KawachiMasec} that $S$ contains an effective Weil divisor $C$ such that $K_S\cdot C<\frac{1}{2}$. But $K_S\cdot C\geqslant\frac{1}{2}$, since $2K_S$ is an ample Cartier divisor. This is a contradiction.

Hence, we see that $(-K_X)^3=\frac{1}{2}$. Then $|-2K_X|$ is a nef, so the restriction $|-2K_X|\big\vert_{S}$ is a pencil in $|2K_S|$. In fact, one can show that $|-2K_X|\big\vert_{S}=|2K_S|$, but we do not need this. So far, we considered $S$ as a complex surface. However, we can choose $S\in |-2K_X|$ to be real, since the net $|-2K_X|$ is defined over $\mathbb{R}$. Thus, from now on, we assume that $S$ is real. 

Write $|-2K_X|\big\vert_{S}=\mathcal{M}+F$, where $\mathcal{M}$ is a mobile part of $|-2K_X|\big\vert_{S}$, and $F$ is its fixed part. By construction, both $\mathcal{M}$ and $F$ are real and $G$-invariant. Thus, it follows from Lemmas~\ref{lemma:reps} and \ref{lemma:ext}, that every curve in  $\mathcal{M}$ is $G$-invariant. Let $M$ be a general real curve in $\mathcal{M}$. Then it follows from the Hodge index theorem that
$$
M^2=M^2\cdot K_S^2\leqslant (M\cdot K_S)^2.
$$
But $M\cdot K_S\leqslant (M+F)\cdot K_S=2K_S^2=2$, since $M+F\sim 2K_S$. So, $M\cdot K_S\leqslant 2$ and $M^2\leqslant 4$. 

Suppose that $M$ is geometrically irreducible. Then $G$ acts faithfully on $M$ by Corollary~\ref{cor:Klein-action-on-3-folds}. Let $f\colon \widetilde{M}\to M$ be the normalization. Then $f$ is $G$-equivariant, and $f$ is defined over $\mathbb{R}$. On the other hand, it follows from the subadjunction lemma \cite[Lemma~5.1.9]{KMM} that 
$$
2\g(\widetilde{M})-2\leqslant (K_S+M)\cdot M\leqslant 6,
$$
so $\g(\widetilde{M})\leqslant 4$, which contradicts Lemma~\ref{lemma:PSL27-curves}. Hence,  $M$ is not geometrically irreducible.

Write $M_\CC=M_1+\cdots+M_n$, were each $M_i$ is an irreducible complex curve, and $n$ is the number of irreducible components of $M_{\CC}$. Since $2K_S\cdot M\leqslant 4$, we see that $n\leqslant 4$, because $2K_S\sim-2K_X\vert_{S}$ is an ample Cartier divisor. Thus, it follows from Corollary~\ref{cor:act} that each curve $M_i$ is $G$-invariant, and $G$ acts faithfully on $M_i$ by Corollary~\ref{cor:Klein-action-on-3-folds}. Since $M^2\leqslant 4$, there is $M_k$ such that $M_k^2\leqslant\frac{4}{n}\leqslant 2$. Then $M_k\cdot K_S<2$, since $M\cdot K_S\leqslant 2$. Let $f_k\colon \widetilde{M}_k\to M_k$ be the normalization. Then it follows from the subadjunction lemma that $2\g(\widetilde{M}_k)-2\leqslant (K_S+M_k)\cdot M_k<4$, which gives $\g(\widetilde{M})<3$. On the other hand, $f_k$ is $G$-equivariant, and $G$ acts faithfully on $\widetilde{M}_k$, which contradicts Lemma~\ref{lemma:PSL27-curves}. The obtained contradiction completes the proof of the lemma.
\end{proof}

\begin{corollary}
\label{corollary:empty-Du-Val}
The surface $S$ has Du Val singularities, and $H^1(\mathcal{O}_S)=H^0(\mathcal{O}_S(K_S))=0$.
\end{corollary}

\begin{proof}
Recall that $S$ has at most Kawamata log terminal singularities by \cite[Main Theorem]{Florin},
which are Gorenstein by Lemma~\ref{lemma:empty-Du-Val}. Hence, they are Du Val.

The equalities $H^1(\mathcal{O}_S)=0$ and $H^0(\mathcal{O}_S(K_S))=0$ follows from exact sequences 
\begin{eqnarray*}
&&0 \longrightarrow \OOO_X(2K_X) \longrightarrow\OOO_X \longrightarrow \OOO_S \longrightarrow 0,
\\
&&0 \longrightarrow \OOO_X(K_X) \longrightarrow \OOO_X(-K_X) \longrightarrow \OOO_S(K_S) \longrightarrow 0,
\end{eqnarray*}
and vanishings 
$H^q(\OOO_X(nK_X))=0$ for any $n$ and $q=1,\, 2$.
\end{proof}

Let $h\colon\widetilde{S}\to S$ be the minimal resolution of singularities of the surface $S$. Then $K_{\widetilde{S}}\sim h^*(K_S)$, the morphism $h$ is $G$-equivariant, $h^1(\mathcal{O}_{\widetilde{S}})=h^0(\mathcal{O}_{\widetilde{S}}(K_{\widetilde{S}}))=0$, and $K_{\widetilde{S}}^2=2(-K_X)^3\in\{1,2,3\}$.
Thus, using the Riemann--Roch theorem and the Kawamata--Viehweg vanishing theorem, we get
$$
\dim(|2K_{\widetilde{S}}|)= K_{\widetilde{S}}^2=2(-K_X)^3.
$$
As above, we see that every divisor in $|2K_{\widetilde{S}}|$ is $G$-invariant. Moreover, if $|2K_{\widetilde{S}}|$ has no fixed components, and $D_1$ and $D_2$ are general curves in $|2K_{\widetilde{S}}|$, the intersection $D_1\cap D_2$ is finite, not empty and $G$-invariant, which contradicts Lemma~\ref{lemma:Klein-action-on-surfaces}, since $|D_1\cap D_2|\leqslant D_1\cdot D_2=4K_{\widetilde{S}}^2\leqslant 12$. Thus, we conclude that  the linear system $|2K_{\widetilde{S}}|$ has a fixed component. 

Write $|2K_{\widetilde{S}}|=|M|+F$, where $|M|$ is the mobile part of the linear system $|2K_{\widetilde{S}}|$, and $F$ is its fixed part. This means that $2K_{\widetilde{S}}\sim M+F$, where $M$ and $F$ are non-zero effective divisors on $\widetilde{S}$ the linear system that $|M|$ is free from base curves, $h^0(\mathcal{O}_{\widetilde{S}}(2K_{\widetilde{S}}))=h^0(\mathcal{O}_{\widetilde{S}}(M))$ and $h^0(\mathcal{O}_{\widetilde{S}}(F))=1$. Then every curve in $|M|$ is $G$-invariant, and the divisor $F$ is $G$-invariant. But 
$M^2\cdot K_{\widetilde{S}}^2\leqslant \big(M\cdot K_{\widetilde{S}}\big)^2$ 
by the Hodge index theorem. Moreover, we have $2K_{\widetilde{S}}^2=M\cdot K_S+F\cdot K_S\geqslant M\cdot K_S$, because $2K_{\widetilde{S}}\sim M+F$. Thus, we have 
$M^2\cdot K_{\widetilde{S}}^2\leqslant(M\cdot K_{\widetilde{S}})^2\leqslant 4(K_{\widetilde{S}}^2)^2$, so $M^2\leqslant 4K_{\widetilde{S}}^2\leqslant 12$. Therefore, if $M_1$ and $M_2$ are general curves in $|M|$, then $|M_1\cap M_2|\leqslant M_1\cdot M_2\leqslant 4K_{\widetilde{S}}^2\leqslant 12$, so that $M_1\cap M_2=\varnothing$ by Lemma~\ref{lemma:Klein-action-on-surfaces}, which gives $M^2=M_1\cdot M_2=0$. Then
$$
12\geqslant 4K_{\widetilde{S}}^2\geqslant 2M\cdot K_{\widetilde{S}}=M^2+M\cdot F=M\cdot F,
$$
because $2K_{\widetilde{S}}\sim M+F$ and $2K_{\widetilde{S}}^2\geqslant M\cdot K_S$. On the other hand, we may assume that $M$ is a general curve in $|M|$, so  $M\cap\mathrm{Supp}(F)$ is finite and $G$-invariant. But $|M\cap\mathrm{Supp}(F)|\leqslant M\cdot F\leqslant 4K_{\widetilde{S}}^2\leqslant 12$. Hence, we have $M\cap\mathrm{Supp}(F)=\varnothing$ and $M\cdot F=0$ by Lemma~\ref{lemma:Klein-action-on-surfaces}.

Set $D=M+F$. Then $D\sim 2 K_{\widetilde{S}}$, and $D$ is not connected. On the other hand, it follows from the Kawamata--Viehweg vanishing theorem and Serre duality that $H^1(\widetilde{S},\OOO_{\widetilde{S}}(-K_{\widetilde{S}}))=0$, so the restriction map $H^0(\widetilde{S}, \OOO_{\widetilde{S}})\to   H^0(D, \OOO_{D})$ is surjective and $h^0(D, \OOO_{D})=1$, a contradiction.

The obtained contradiction completes the proof of Proposition \ref{proposition:non-empty}.

\subsection{Anticanonical map}
\label{subsection:anticanonical-map}

From Proposition \ref{proposition:non-empty}, we know that $|-K_X|\ne\varnothing$. In this section, we study the map given by $|-K_X|$. First, we show that $\dim(|-K_X|)\ne 0$. To do this, we need 

\begin{lemma}
\label{lemma:G-invariant-surfaces-G-fixed-points-no}
Suppose that $|-K_X|$ contains a $G$-invariant surface. Then the 3-fold $X_{\CC}$ has no non-Gorenstein singular points that are fixed by $G$.
\end{lemma}

\begin{proof}
Let $S$ be a $G$-invariant surface in the linear system $|-K_X|$, and let $P$ be a non-Gorenstein singular point of the 3-fold $X_{\CC}$ that is fixed by $G$. Then $P\in S_{\CC}$, and it follows from Corollary~\ref{corollary:1-2-points} that $X_{\CC}$ has a cyclic quotient singularity of type $\frac{1}{2}(1,1,1)$ at the point~$P$. By Corollary~\ref{corollary:fixed-point:21}, the singularities of the pair $(X_{\CC},S_{\CC})$ are not log canonical at $P$, which contradicts Lemma~\ref{lemma:PSL27-real-3-folds-log-pair}.
\end{proof}

Now, we are ready to prove 

\begin{proposition}
\label{proposition:K3-surfaces}
One has $\dim(|-K_X|)\ne 0$.
\end{proposition}

\begin{proof}
Suppose $\dim(|-K_X|)=0$. Let $S$ be the unique surface in $|-K_X|$. Then $S$ is $G$-invariant, so it follows from Lemma~\ref{lemma:G-invariant-surfaces-G-fixed-points-no} that $X_{\CC}$ has no non-Gorenstein singular points that are fixed by $G$. Hence, applying Corollaries~\ref{corollary:1-2-points} and \ref{corollary:BM-RR}, we get $(-K_X)^3=3$. But $S$ is singular at every non-Gorenstein point of $X_{\CC}$, so it follows from Lemma~\ref{lemma:G-invariant-surfaces} that $S$ is reduced, it is reducible, and $G$ acts transitively on the set of its irreducible components. Let $t$ be the number of these components, and let $S^\prime$ be an irreducible component of $S$. Then $t\geqslant 7$ by Corollary~\ref{cor:act}, and $2(-K_X)^2\cdot S^\prime$ is a positive integer, because $2K_X$ is a Cartier divisor by  Corollary~\ref{corollary:1-2-points-2K}. This gives $6=2(-K_X)^2\cdot S=\big(2(-K_X)^2S^\prime\big)t\geqslant t\geqslant 7$, which is absurd. 
\end{proof}

Now, using Corollaries \ref{corollary:K3-surfaces-1}  and \ref{corollary:K3-surfaces-2}, we obtain the following 4 corollaries.

\begin{corollary}
\label{corollary:K3-surfaces-1-cheap}
The linear system $|-K_X|$ has no fixed components.
\end{corollary}

\begin{corollary}
\label{corollary:K3-surfaces-2-easy}
The log pair $(X,|-K_X|)$ is canonical.
\end{corollary}

\begin{corollary}
\label{corollary:K3-surfaces-3-easy}
Let $S$ be a general surface in $|-K_X|$. Then $(X,S)$ is  purely log terminal, and $S$ is a K3 surface with at most Du Val singularities.
\end{corollary}

\begin{corollary}
\label{corollary:pencil}
The linear system $|-K_{X_{\CC}}|$ contains at most one $G$-invariant divisor.
\end{corollary}

\begin{proof}
Suppose that $|-K_{X_{\CC}}|$ contains two $G$-invariant surfaces $S_1$ and $S_2$, let $\mathcal{M}$ be the pencil generated by them, and let $S$ be a general surface in this pencil. Then, since $G$ has no non-trivial one-dimensional subrepresentations, every surface in $\mathcal{M}$ is $G$-invariant, so, in particular, the surface $S$ is also $G$-invariant. Then, by Corollary~\ref{corollary:K3-surfaces-2}, the surface $S$ is a K3 surface with at most Du Val singularities. On the other hand, the surface $S$ is singular at every non-Gorenstein singular point of the 3-fold $X$, which contradicts Corollary~\ref{corollary:K3}.
\end{proof}

\begin{corollary}
\label{corollary:dim-5}
One has $\dim(|-K_{X}|)\geqslant 5$.
\end{corollary}

\begin{proof}
Note that $H^0(X,\mathcal{O}_X(-K_X))$ is a real representation of the group $G$, which has at most one trivial $1$-dimensional subrepresentation by Corollary~\ref{corollary:pencil}. Thus, since $h^0(X,\mathcal{O}_X(-K_X))\geqslant 2$ by Propositions \ref{proposition:non-empty} and \ref{proposition:K3-surfaces}, it follows from Lemma~\ref{lemma:reps} that $h^0(X,\mathcal{O}_X(-K_X))\geqslant 6$.
\end{proof}

\begin{corollary}
\label{corollary:K3-surfaces-4}
All non-Gorenstein singular points of $X_{\CC}$ are cyclic quotient singularities.
\end{corollary}

\begin{proof}
Let $P$ be a non-Gorenstein singular point of $X_{\CC}$, and let $S$ be a general surface in the linear system $|-K_X|$. Then it follows from by Corollary~\ref{corollary:K3-surfaces-2} that  $(X,S)$ is purely log terminal, and $S$ has only Du Val singularities. Moreover, the surface $S$ contains $P$, and it is singular at this point.

Suppose $X_{\CC}$ has a non-Gorenstein singular point $P$ that is not a cyclic quotient singularity. Then, by Corollary~\ref{corollary:1-2-points},
$X_{\CC}$ has a moderate singularity of index $2$ and axial weight $2$ at $P$, the point $P$ is real, and its $G$-orbit has length $7$. 
Now, by \cite[\S~6.4B]{Reid:YPG} the singularity $P\in S$ is not of type $\mathrm{A}_1$ nor $\mathrm{A}_2$.
Let $\pi\colon\widetilde{S}\to S$ be the  minimal resolutions of singularities. Then the rank of the Picard group $\mathrm{Pic}(\widetilde{S})$ is at least $1+3\cdot 7=22$, which is a contradiction, since the rank of the Picard group of a smooth K3 surface is at most  $20$.
\end{proof}

\begin{remark}
\label{remark:LiuLiu}
If the~inequality \eqref{equation:BM} holds for $X_{\CC}$, then Corollary~\ref{corollary:K3-surfaces-4} immediately follows from \eqref{eq2.4} and Corollaries~\ref{corollary:BM-RR},  \ref{corollary:1-2-points}, \ref{corollary:dim-5}, which also imply that $X_{\CC}$ cannot have $14$ non-Gorenstein singular points. Moreover, if $\rho(X_{\CC})=1$, then it follows from \cite{Bayle,SanoEnriques,Sano,Takagi1,Takagi2} that  $X_{\CC}$ has at most $7$ non-Gorenstein singular points. 
\end{remark}

We see that $|-K_X|$ is mobile, and all non-Gorenstein singular points of $X$ are cyclic quotient singularities of type $\frac{1}{2}(1,1,1)$. Set $n=\frac{1}{2}(-K_X)^3+2-\frac{N}{4}$, where $N$ is the number of non-Gorenstein singular points of $X$. Then by Corollary~\ref{corollary:BM-RR} we have $\dim (|-K_X|)=n$, hence  $|-K_X|$ defines a rational map $\psi\colon X\dasharrow\PP^n$. Let $\overline{X}=\overline{\mathrm{im}(\psi)}\subset\PP^n$. Note that $n\geqslant 5$ by Corollary~\ref{corollary:dim-5}, and the map $\psi$ is undefined at every non-Gorenstein singular point of the 3-fold $X$, because these points are contained in the base locus of $|-K_X|$.  We will see soon that these are the only base points of $|-K_X|$, the indeterminacy of $\psi$ can be resolved by blowing them up, $\overline{X}$ is a Fano 3-fold with Gorenstein canonical singularities, and $-K_{\overline{X}}\sim\mathcal{O}_{\PP^n}(1)\vert_{\overline{X}}$.

Namely, let $\pi\colon Y\to X$ be the blow up of all non-Gorenstein singular points of the 3-fold $X$, and let $E_{1},\ldots,E_{N}$ be the (irreducible) exceptional divisors of $Y_\CC/X_\CC$. Then $\pi$ is $G$-equivariant, each $E_i\simeq\PP^2$, and $(-K_{Y})^3=(-K_X)^3-\frac{N}{2}=2n-4\geqslant 6$.

\begin{lemma}
\label{lemma:base-points}
The base locus of the linear system $|-K_{X_{\CC}}|$ consists of non-Gorenstein singular points of the 3-fold $X_{\CC}$,
the divisor $-K_Y$ is big and nef, and $\psi\circ\pi$ is a morphism given by $|-K_{Y}|$.
Thus there exists $G$-equivariant commutative diagram
$$
\xymatrix{
&Y\ar[ld]_{\pi}\ar[rd]^{\phi}\ar@{->}[rd]&
\\
X\ar@{-->}[rr]_{\psi}&&\overline{X},}
$$
where $\phi$ is a crepant birational contraction and  $\overline{X}$ is a Fano 3-fold with Gorenstein canonical singularities such that $-K_{\overline{X}}\sim\mathcal{O}_{\PP^n}(1)\vert_{\overline{X}}$.
\end{lemma}

\begin{proof}
Let $S$ be a general surface in $|-K_{X}|$. 
Then it follows from Corollary~\ref{corollary:K3-surfaces-2} that $S$ is a geometrically irreducible surface, which is a K3 surface with at most Du Val singularities. Let $\widetilde{S}$ be the strict transform on $Y$ of the surface $S$. Then
$$
K_{Y}+\widetilde{S}\sim_{\mathbb{Q}} \pi^*(K_X+S)+\sum_{i=1}^{N}\Big(\frac{1}{2}-m_i\Big)E_i,
$$
where each $m_i\in\mathbb{Q}_{>0}$ such that $2m_i\in\mathbb{Z}$. By Corollary~\ref{corollary:K3-surfaces-1}, each $m_i=\frac{1}{2}$, so $|-K_{Y}|$ is the strict transform of the linear system $|-K_X|$, and the surface $\widetilde{S}$ is smooth near each $\pi$-exceptional surface $E_i$, since $\widetilde{S}\vert_{E_i}$ is a line in $E_i\simeq\PP^2$, because $\widetilde{S}\vert_{E_i}\sim_{\mathbb{Q}}\mathcal{O}_{\PP^2}(1)$.

Consider the map  $\phi=\psi\circ\pi: Y \dasharrow \PP^n$;  it is given by the linear system  $|-K_{Y}|$.
Let $G_i$ be the stabilizer of the surface $E_i$ in $G$.
By Lemma~\ref{lemma:Klein-subgroups} we have $G_i\simeq \mathfrak{S}_4$, $G_i\simeq \mathfrak{A}_4$   or $G_i=G$
Then it follows from Corollary~\ref{corollary:1-2-points} that $G_i$ acts faithfully on $E_i\simeq\PP^2$, and $E_i$ contains no $G_i$-invariant  points otherwise $G_i$ 
would act faithfully on the tangent space to $E_i$ at a fixed point). This implies that 
$E_i$ also contains no $G_i$-invariant lines. Since $-K_{Y}\vert_{E_i}\sim\mathcal{O}_{\PP^2}(1)$, we see that the linear system $|-K_{Y}|$ has no base points in $E_1\cup\ldots\cup E_r$. This implies that the base curves of the linear system $|-K_{Y}|$ are disjoint from $E_1\cup\ldots\cup E_N$. Then $-K_{Y}$ is nef and big, because $(-K_{Y})^3=2n-4\geqslant 6$. 
Moreover, the map $\phi$ is regular near each $E_i$.

Thus, $Y$ is a weak Fano 3-fold with Gorenstein terminal singularities. For $m\gg 0$, the linear system $|m(-K_Y)|$ gives a morphism $\varphi\colon Y\to \overline{Y}$ such that $\overline{Y}$ is a Fano 3-fold with at worst canonical Gorenstein singularities. Then $-K_Y\sim \varphi^*(-K_{\overline{Y}})$. On the other hand, $|-K_{\overline{Y}}|$ is base point free by Lemma~\ref{lemma:canFano:Bs}, so $|-K_Y|$ is  base point free as well. Moreover, by Corollary~\ref{corollary:hyperelliptic}, that the divisor $-K_{\overline{Y}}$ is very ample. Thus, since $\dim(|-K_Y|)=\dim(|-K_{\overline{Y}}|)=n$ by the Riemann--Roch theorem, we see that $|-K_X|$ and $|-K_Y|$ are strict transforms of the linear system $|-K_{\overline{Y}}|$, which implies that $\overline{Y}\simeq\overline{X}$, so the assertion follows.
\end{proof}

Set $g=\g(X)=\frac{1}{2}(-K_{\overline{X}}^3)+1$. Then $n=g+1$. Applying Corollary~\ref{corollary:trigonal}, we see that $\overline{X}$ is an intersection of quadrics in $\PP^n$, and $g\geqslant 5$.

\begin{lemma}
\label{lemma:g-6-8-9-etc}
One has $6\leqslant g\leqslant 37$ and $g\not\in\{7,\, 8,\, 9\}$.
\end{lemma}

\begin{proof}
By Lemma~\ref{lemma:g-not-5}, we have $g\ne 5$. Moreover, it follows from Corollary~\ref{corollary:pencil} that the linear system $|-K_{\overline X}|$ contains at most one $G$-invariant divisor. Hence, by Lemma~\ref{lemma:g-not-7}, we have $g\not\in\{7,\, 8,\, 9\}$. Finally, we recall from \cite{Prokhorov-72} that $g\leqslant 37$.
\end{proof}

\subsection{Terminal Gorenstein model}
\label{subsection:anticanonical-map-not-contracting-divisors}

Let us use assumptions and notations from Section~\ref{subsection:anticanonical-map}. The goal of this section is to prove the following result.

\begin{proposition}
\label{proposition:terminal-singularities}
The birational morphism $\phi$ is small.
\end{proposition}

\begin{corollary}
\label{corollary:terminal-singularities}
The Fano 3-fold $\overline{X}$ has terminal Gorenstein singularities.
\end{corollary}

Suppose that $\phi$ is not small. Let $F_1,\ldots,F_k$ be all (irreducible) exceptional divisors of $Y_\CC/\overline X_\CC$
and set $\overline{E}_i=\phi(E_i)$ for every $i\in\{1,\ldots,N\}$. Then each $\overline{E}_i$ is a plane in $\PP^n$ that is contained in $\overline{X}$, and the induced morphism $E_i\to\overline{E}_i$ is an isomorphism. For $i\ne j$, the intersection $\overline{E}_i\cap\overline{E}_j$ is either empty, or consists of a single point, or consists of a single line.

\begin{lemma}
\label{lemma:divisors-to-points}
For every $j\in\{1,\ldots,k\}$, the image $\phi(F_j)$ is a curve. Therefore, the singularities of $\overline{X}$ are pseudo-terminal.
\end{lemma}

\begin{proof}
Since $X$ is a $G\mathbb{Q}$-Fano 3-fold, $F_j$ must intersect one of the surfaces $E_1,\ldots,E_N$, and the intersection must be a curve. Since this curve is contracted by the anticanonical morphism $\phi$, it has trivial intersection with $-K_{Y}$. On the other hand, since it is contracted by $\pi$, it must have positive intersection with $-K_{Y}$, which is absurd.
\end{proof}

For every $i\in\{1,\ldots,k\}$, we set $C_i=\phi(F_i)$ and $d_i=\deg(C_i)$. Note that we do not claim that the curves $C_1,\ldots,C_{k}$ are all distinct, since it may happen that $\phi(F_i)=\phi(F_j)$ for $i\ne j$. We have
\begin{equation}
\label{equation:19}
\sum_{i=1}^{k}d_i\leqslant 19.
\end{equation}
Indeed, let $\overline{H}$ be a general hyperplane section of the 3-fold $\overline{X}\subset\PP^n$, and let $H$ be its strict transform on the 3-fold $Y$. Then $\overline{H}$ has Du Val singularities, $H$ is a smooth K3 surface, and $E_1\vert_{H},\ldots,E_N\vert_{H}$ are disjoint $(-2)$-curves in $H$. Moreover, each $F_i\vert_{H}$ is a disjoint union of $d_i\geqslant 1$ $(-2)$-curves. Furthermore, all these $d_1+\cdots+d_k$ $(-2)$-curves $F_1\vert_{H},\ldots,F_k\vert_{H}$ are contracted by the birational morphism $\phi\vert_{H}\colon H\to\overline{H}$. This gives \eqref{equation:19}, since $\mathrm{rk}\,\mathrm{Pic}(H)\leqslant 20$. In particular, $k\leqslant 19$.

\begin{lemma}
\label{lemma:G-fixed-points}
If $E_i$ is $G$-invariant, then $E_i$ is disjoint from $F_1,\ldots,F_k$.
\end{lemma}

\begin{proof}
We may assume that $E_1$ is $G$-invariant, and $E_1\cap F_1\ne\varnothing$. Let us seek for a contradiction. Note that $F_1\vert_{E_1}$ is a curve of degree $d_1$ in $E_1\simeq\PP^2$, since $\phi$ induces an isomorphism $E_1\simeq\overline{E}_1$.

It follows from Lemma~\ref{lemma:Klein-action-on-surfaces} that $\overline{E}_1$ has no $G$-orbits of length less than $21$, which implies that the plane $\overline{E}_1$ has no $G$-invariant curves that are union of less than $21$ lines. Moreover, it follows from Corollary~\ref{corollary:1-2-points} that the $G$-orbit of each $\overline{E}_i$ consists of $1$ or $7$ planes. This gives $\overline{E}_1\cap\overline{E}_i=\varnothing$ for every $i\ne 1$.

By Corollary~\ref{cor:act}, the $G$-orbit of $F_1$ consists of $1$, $7$, $8$ or $14$ surfaces, since $k\leqslant 19$ by \eqref{equation:19}. Thus, the  stabilizer of the surface $F_1$ in $G$ does not leave invariant any line in $E_1\simeq\PP^2$. This gives $d_1=\deg\big(F_1\vert_{E_1}\big)\geqslant 2$. If the $G$-orbit of $F_1$ consists of $8$ surfaces, then the stabilizer of the surface $F_1$ in $G$ does not leave invariant any conic in $E_1$, so $d_1\geqslant 3$ in this case. Therefore, by \eqref{equation:19},  either $F_1$ is $G$-invariant, or the $G$-orbit of the surface $F_1$ consists of $7$ surfaces.

By Corollary~\ref{corollary:1-2-points}, $E_1$ is not real. We may assume that its complex conjugate is $E_2$. Then $E_2$ is also $G$-invariant. Moreover, we have $F_1\cap E_i=\varnothing$ for every $i\ne 1$, since we already proved that the planes $\overline{E}_1$ and $\overline{E}_i$ are disjoint for $i\ne 1$. In particular, we see that $F_1\cap E_2=\varnothing$, which implies that $F_1$ is not real. We may assume that its complex conjugate is $F_2$. Then $d_2=\deg\big(F_2\vert_{E_2}\big)=d_1\geqslant 2$, and $F_2\cap E_1=\varnothing$, since $F_1\cap E_2=\varnothing$. On the other hand, the $G$-orbits of the surfaces $F_1$ and $F_2$ are different and have the same length. So, it follows from \eqref{equation:19} that $F_1$ and $F_2$ are $G$-invariant.  Furthermore, the surfaces $F_1$ and $F_2$ are disjoint, since otherwise
$$
\varnothing\ne\phi(F_1\cap F_2)=C_1\cap C_2\subset\overline{E}_1\cap\overline{E}_2=\varnothing.
$$
Then the surfaces $\pi(F_1)$ and $\pi(F_2)$ are disjoint as well, because $F_1$ is disjoint from all $\pi$-exceptional surfaces except $E_1$, and $F_2$ is disjoint from $E_1$. But $\pi(F_1)+\pi(F_2)$ is real and $G$-invariant, which implies that $\pi(F_1)+\pi(F_2)\sim_{\mathbb{Q}} a(-K_X)$  for some $a\in\mathbb{Q}_{>0}$. In particular,  $\pi(F_1)+\pi(F_2)$ is ample, so it is connected, which is a contradiction.
\end{proof}

\begin{lemma}
\label{lemma:G-orbit-length-7}
If the $G$-orbit of $E_i$ in $X_{\CC}$ consists of $7$ surfaces, then the surface $E_i$ is disjoint from $F_1,\ldots,F_k$.
\end{lemma}

\begin{proof}
We may assume that the $G$-orbit of $E_1$ consists of the $\pi$-exceptional surfaces $E_1,\ldots, E_7$. In the following, we will work with the geometric model of the 3-fold $X$. Suppose that $E_1\cap F_1\ne\varnothing$. Let us seek for a contradiction.

We claim that the planes $\overline{E}_1,\ldots,\overline{E}_7$ are disjoint in codimension $1$. Indeed, suppose they are not. Since $G$ acts doubly transitively on the set of these $7$ planes, we see that $L_{ij}:=\overline{E}_i\cap \overline{E}_j$ is a line for every $i\ne j$ in $\{1,\ldots,7\}$. These lines form one $G$-orbit, i.e. $G$ transitively permutes them. Moreover, each plane among $\overline{E}_1,\ldots,\overline{E}_7$ contains $3$ or $6$ such lines, since the stabilizer of the plane does not leave invariant any line in the plane.  Lets count these lines.

Let $m$ be the number of planes among $\overline{E}_1,\ldots,\overline{E}_7$ that pass through the line $L_{12}$. If $\overline{E}_1$ contains three lines in the $G$-orbit of the line $L_{12}$, then $m=3$ and we have $(7\cdot 3)/m=7$ such lines in total. If $\overline{E}_1$ contains $6$ lines in the $G$-orbit of the line $L_{12}$, then $m=2$ and we have $(7\cdot 6)/m=21$ such lines. But each line $L_{ij}$ is one of the curves $C_1,\ldots,C_k$, so $k\geqslant 21$, which contradicts  \eqref{equation:19}. So, there are $7$ lines $L_{ij}$ in total, and there are $3$ planes among $\overline{E}_1,\ldots,\overline{E}_7$ that contain each line.

Now, if two such lines intersect (by a point), lets us call their unique intersection point ``vertex''. Lets count the number of vertices. The stabilizer of the plane $\overline{E}_1$ in $G$ is isomorphic to $\mathfrak{S}_4$, and it does not fix points in $\overline{E}_1$. Thus, each plane among $\overline{E}_1,\ldots,\overline{E}_7$ contains exactly $3$ vertices, and there are $5$ planes among $\overline{E}_1,\ldots,\overline{E}_7$ contains each vertex. Counting, we see that we have $(7\cdot 3)/5$ vertices in total, which is absurd. This contradiction shows that the planes $\overline{E}_1,\ldots,\overline{E}_7$ are disjoint in codimension one, i.e., if two of them intersect, they intersect by a point.

Recall that $F_1\vert_{E_1}$ is a curve of degree $d_1$ in $E_1\simeq\PP^2$, since $\phi$ induces an isomorphism $E_1\simeq\overline{E}_1$. Then $d_1\geqslant 2$, because the stabilizer of the surface $E_1$ in $G$ does not leave any line in $E_1$ invariant. Furthermore, the surface $F_1$ does not intersect any surface among $E_2,\ldots,E_7$, since otherwise the planes $\overline{E}_1,\ldots,\overline{E}_7$ would not be disjoint in codimension one. Therefore, using \eqref{equation:19}, we see that $d_1=2$, and the $G$-orbit of the surface $F_1$ consists of exactly $7$ surfaces.

Now, we restrict everything to $H$. The restrictions $E_1\vert_{H},\ldots,E_7\vert_{H}$ are disjoint $(-2)$-curves in $H$.
Similarly, each restrictions $F_i\vert_{H}$ splits as a union of two disjoint $(-2)$-curves which both intersects the $(-2)$-curve $E_i\vert_{H}$ transversally in one point. Note that $F_i\vert_{H}\cap F_{j}\vert_{H}=\varnothing$ for $i\ne j$, because otherwise we would have $C_i=C_j$, so the intersection $\overline{E}_i\cap\overline{E}_{j}$ would contain the curve $C_i=C_j$, which is impossible, since the planes $\overline{E}_1,\ldots,\overline{E}_7$ are disjoint in codimension one. Thus, the smooth K3 surface $H$ contains the following $21$ $(-2)$-curves:
$E_1\vert_{H},\ldots,E_7\vert_{H},F_1\vert_{H},\ldots,F_7\vert_{H}$. They generate a sublatice in $\mathrm{Pic}(H)$ of rank $21$. This is impossible, since $\mathrm{rk}\,\mathrm{Pic}(H)\leqslant 20$.
\end{proof}

\begin{lemma}
$\mathrm{r}^G(\overline{X})=1$.
\end{lemma}

\begin{proof}
Since $X$ is a real $G\mathbb{Q}$-Fano 3-fold, some surfaces among $E_1,\ldots,E_N$ must intersect some surfaces among $F_1,\ldots,F_k$. Thus, it follows from Corollary~\ref{corollary:1-2-points} and Lemmas \ref{lemma:G-fixed-points} and \ref{lemma:G-orbit-length-7} that $N=14$, and the 3-fold $X_{\CC}$ has $14$ singular quotient singularities of type $\frac{1}{2}(1,1,1)$, which form one $G$-orbit. Since $X$ is a $G\mathbb{Q}$-Fano 3-fold, we have $\mathrm{r}^G(\overline{X})=1$.
\end{proof}
 
 Set $\overline{D}=\overline{E}_1+\cdots+\overline{E}_{14}$. Then $\overline{D}\sim_{\mathbb{Q}} \frac{7}{g-1}(-K_{\overline{X}})$. In particular, $\overline{D}$ is ample, so it is connected in codimension one. Hence, there are $i,j\in\{1,\ldots,14\}$ such that $i\ne j$ and  $\overline{E}_i\cap\overline{E}_j$ is a line.

\begin{lemma}
\label{lemma:i-j-line}
If $E_i\cap F_j\ne \varnothing$, then $E_i\cap F_j$ is a line in $E_i\simeq\PP^2$.
\end{lemma}

\begin{proof}
Suppose that $E_1\cap F_1\ne\varnothing$. Set $Z=E_1\cap F_1$. Then $Z$ is a curve of degree $d$ in $E_1\simeq\PP^2$, so $C_1=\phi(Z)$ is a curve of degree $d$ in $\overline{E}_1$. If $d\ne 1$, then $C_1$ is not contained in any other plane among $\overline{E}_2,\ldots,\overline{E}_{14}$, so the $G$-orbit of the curve $C_1$ consists of at least $14$ curves of degree $d\geqslant 2$, which are curves among $C_1,\ldots,C_k$, which contradicts \eqref{equation:19}.
\end{proof}

Therefore, since every surface among $F_1,\ldots,F_k$ must intersect some surface among $E_1,\ldots,E_{14}$, it follows from Lemma~\ref{lemma:i-j-line} that all curves $C_1,\ldots,C_k$ are lines. Moreover, if $\overline{E}_i\cap\overline{E}_j$ is a line for some $i\ne j$, then this line is one of the lines $C_1,\ldots, C_k$. In fact, we can say more:

\begin{lemma}
\label{lemma:i-j}
Every line among $C_1,\ldots,C_{k}$ is contained in at least $3$ planes among $\overline{E}_1,\ldots,\overline{E}_{14}$,
and every plane among $\overline{E}_1,\ldots,\overline{E}_{14}$ contains at least $3$ lines among $C_1,\ldots,C_{k}$.
\end{lemma}

\begin{proof}
We may assume that $C_1$ is contained in $\overline{E}_1$. Let $G_1$ be the stabilizer of the plane $\overline{E}_1$ in $G$. Then $G_1\simeq\mathfrak{A}_4$, and the plane $\overline{E}_1$ contains neither $G_1$-invariant lines nor $G_1$-fixed points, which implies that the $G_1$-orbit of the line $C_{1}$ contains at least $3$ lines in $\overline{E}_1$. Similarly, if $C_1$ is contained in at most two planes among $\overline{E}_1,\ldots,\overline{E}_{14}$, then the $G$-orbit of $C_{1}$ contains at least $(14\cdot 3)/2=21$ lines. This is impossible by \eqref{equation:19}.
\end{proof}

\begin{corollary}
\label{corollary:i-j}
The singularities of the log pair $(\overline{X},\overline{D})$ are worse than log canonical at general point of every line among $C_1,\ldots,C_k$.
\end{corollary}

\begin{proof}
See e.g. \cite[Proposition 16.6]{Utah}.
\end{proof}

\begin{corollary}
\label{corollary:gg}
One has $g=6$.
\end{corollary}

\begin{proof}
Suppose that $g\ne 6$. Then $g \geqslant 10$ by Lemma~\ref{lemma:g-6-8-9-etc}. 
Note that $\overline{D}\sim_{\mathbb{Q}} \frac{7}{g-1}(-K_{\overline{X}})$. Thus, applying Lemma \ref{lemma:PSL27-real-3-folds-log-pair}, we get a contradiction, since $(\overline{X},\overline{D})$ is not log canonical.
 \end{proof}

\begin{lemma}
\label{lemma:no-G-invariant-surfaces}
The linear system $|-K_{\overline{X}}|$ does not contain $G$-invariant surfaces.
\end{lemma}

\begin{proof}
Suppose that the linear system $|-K_{\overline{X}}|$ contains a $G$-invariant surface $\overline{S}$. Then $\overline{S}$ cannot have more than $(-K_{\overline{X}})^{3}=2g-2=10$ geometrically irreducible components, which implies, in particular, than planes $\overline{E}_1,\ldots,\overline{E}_{14}$ are not contained in $\overline{S}$.

We claim that $\overline{S}$ does not contain any line among $C_1,\ldots,C_k$. Indeed, if it does, then it follows from Lemma~\ref{lemma:i-j} that $\overline{S}$ contains all planes $\overline{E}_1,\ldots,\overline{E}_{14}$, since $\overline{S}$ is cut out by a hyperplane in $\PP^n$. But we just showed that this is not the case.

Let $S$ be the strict transform  on $X$ of the surface $\overline{S}$.
Then $S$ and $\overline{S}$ are isomorphic in codimension one, they are reduced by Lemma~\ref{lemma:PSL27-real-3-folds-log-pair}.
Moreover, both of them are reducible (over $\mathbb{R}$), since otherwise $S$ would be a smooth K3 surface by Lemma~\ref{lemma:G-invariant-surfaces}, which is impossible, since $S_\CC$ is singular at every non-Gorenstein singular point of $X$.

By Lemma~\ref{lemma:G-invariant-surfaces}, the group $G$ acts transitively on the set of irreducible real components of $\overline{S}$. Let $t$ be the number of these irreducible components, and let $\overline{S}^\prime$ be one such component. Then 
$$
t\deg\big(\overline{S}^\prime\big)=\deg\big(\overline{S}\big)=(-K_{\overline{X}})^{2}\cdot\overline{S}=(-K_{\overline{X}})^{3}=2g-2=10,
$$
which contradicts Corollary~\ref{cor:act}. 
\end{proof}

Thus, applying Lemma~\ref{lemma:reps}, we get

\begin{corollary}
\label{corollary:g-6}
The action of the group $G$ on $\PP^7$ is induced by its unique real irreducible representation of dimension $8$.
\end{corollary}

Choose $\lambda<1$ such that $(\overline{X},\lambda\overline{D})$ is log canonical at general point of every line among $C_1,\ldots,C_k$, but it is not Kawamata log terminal at general point of one of these lines. Then
$$
\mathrm{Nklt}(\overline{X},\lambda\overline{D})\subset\bigcup_{i=1}^{k} C_i,
$$
and one of the lines $C_1,\ldots,C_k$ is a center of log canonical singularities of the log pair $(\overline{X},\lambda\overline{D})$.
Without loss of generality, we may assume that this line is $C_1$. Then $C_1\subset \mathrm{Nklt}(\overline{X},\lambda\overline{D})$.

Arguing as in the proof of Lemma~\ref{lemma:i-j}, we see that the $G$-orbit of $C_1$ contains another line among $C_2,\ldots,C_k$ that intersects $C_1$.
Without loss of generality, we may assume that $C_1\cap C_2\ne\varnothing$. Then it follows from \cite{Kawamata-1} that the point $C_1\cap C_2$ is a minimal center of log canonical singularities of the pair $(\overline{X},\lambda\overline{D})$. Arguing as in the proofs of \cite[Theorem 1.10]{Kawamata-1} and \cite[Theorem 1]{Kawamata-2}, we can find a rational number $\lambda^\prime<1$ and an effective $G$-invariant $\mathbb{Q}$-divisor $\overline{D}^\prime$ on $\overline{X}$ such that $\overline{D}^\prime\sim_{\mathbb{Q}} \overline{D}$, the log pair $(\overline{X},\lambda^\prime\overline{D}^\prime)$ has log canonical singularities at the $G$-orbit of the point $C_1\cap C_2$, the locus $\mathrm{Nklt}(\overline{X},\lambda^\prime\overline{D}^\prime)$ contains the $G$-orbit of the point $C_1\cap C_2$, and it does not contain curves. This also follows directly from \cite[Lemma 2.4.10]{Icosahedron}.

Let $\mathcal{L}$ be the subscheme in $\overline{X}$ given by the multiplier ideal sheaf of the log pair $(\overline{X},\lambda^\prime\overline{D}^\prime)$.
Then $\mathcal{L}$ is zero-dimensional, its support contains the $G$-orbit of the intersection point $C_1\cap C_2$, and it follows from Nadel vanishing theorem \cite[Theorem 9.4.8]{lazarsfeld2003positivity} applied to the log pair $(\overline{X},\lambda^\prime\overline{D}^\prime)$ that the restriction homomorphism of $G$-representations $H^0(\mathcal{O}_{\overline{X}}(-K_{\overline{X}}))\to H^0(\mathcal{O}_{\mathcal{L}})$ is surjective. Thus, the $G$-orbit of the point $C_1\cap C_2$ has length $\leqslant 8$. So, by Corollary~\ref{corollary:g-6},  the $G$-orbit of the point $C_1\cap C_2$ consists of $8$ points that are in general linear position in $\PP^7$.

By construction, we have
$$
C_1\cap C_2\subset\bigcup_{i=1}^{14}\overline{E}_{i}.
$$
We may assume that $C_1\cap C_2$ is contained in $\overline{E}_1$. Let $G_1$ be the stabilizer of the plane $\overline{E}_1$ in~$G$. Then $G_1\simeq\mathfrak{A}_4$, and $\overline{E}_1$ does not contain $G_1$-orbits of length $1$ and $2$, which implies that $\overline{E}_1$ contains at least $3$ points of the $G$-orbit of the point $C_1\cap C_2$. Let $G_1^\prime$ be the unique subgroup of $G$ such that $G_1^\prime$ contains $G_1$ and $G_1^\prime\simeq\mathfrak{S}_4$. Then the $G_1^\prime$-orbit of the plane  $\overline{E}_1$ consists of the plane  $\overline{E}_1$ and some other plane among $\overline{E}_1,\ldots,\overline{E}_{14}$. Without loss of generality, we may assume that this plane is~$\overline{E}_2$. Then $\overline{E}_1\cap\overline{E}_{2}=\varnothing$, because $\overline{E}_1$ does not contain $G_1$-invariant lines and $G_1$-fixed points. On the other hand, $\overline{E}_1\cup\overline{E}_2$ cannot contain the $G$-orbit of the point $C_1\cap C_2$, because points of this $G$-orbit are in general linear position. Thus, each plane $\overline{E}_1$ and $\overline{E}_2$ contains exactly $3$ points of the $G$-orbit of the point $C_1\cap C_2$, and two points of this orbit are not in $\overline{E}_1\cup\overline{E}_2$. Then the set of these two points is $G_1^\prime$-invariant, so the stabilizer of one of them must contain $G_1$, but its stabilizer in $G$ is isomorphic to $\mumu_7\rtimes\mumu_3$. This is a contradiction, since $\mumu_7\rtimes\mumu_3$ has no subgroups isomorphic to $\mathfrak{A}_4$. The proof of Proposition \ref{proposition:terminal-singularities} is complete.

\subsection{Non-Gorenstein reduction}
\label{subsection:reduction}

Let us use all assumptions and notations from Section~\ref{subsection:anticanonical-map}. By Proposition \ref{proposition:terminal-singularities}, the birational morphism $\phi\colon Y\to\overline{X}$ is small (or an isomorphism), so $\overline{X}$ has terminal Gorenstein singularities. Thus, it follows from \cite{Namikawa,SmoothingPic} that $\overline{X}$ admits a $\mathbb{Q}$-Gorenstein smoothing to a smooth Fano 3-fold $V$ such that
$(-K_{V})^3=(-K_{\overline{X}})^3=2g-2\geqslant 10$ (see Lemma~\ref{lemma:g-6-8-9-etc}) and $\rho(V)=\rho(\overline{X}_{\CC})$. In particular, since $(-K_{V})^3\leqslant 64$, we have $g=\g(\overline{X})\leqslant 33$. Moreover, it follows from \cite{Namikawa} that
\begin{equation}
\label{equation:Sing}
|\mathrm{Sing}(\overline{X}_{\CC})|\leqslant 20+h^{1,2}\big(V\big)-\rho\big(V\big)\leqslant 19+h^{1,2}\big(V\big).
\end{equation}
If $V$ is a smooth Fano 3-fold of Fano index $\iota(V)$, then $-K_{\overline{X}}$ is also divisible by $\iota(V)$ in $\mathrm{Pic}(\overline{X})$. Thus, it follows from Corollary~\ref{corollary:DP} that $\iota(V)=1$.

\begin{corollary}
\label{corollary:Sing}
One has $|\mathrm{Sing}(\overline{X}_{\CC})|\leqslant 29$.
\end{corollary}

\begin{proof}
Follows from \eqref{equation:Sing} the classification of smooth Fano 3-folds of Fano index $1$.
\end{proof}

By Corollaries \ref{corollary:1-2-points} and \ref{corollary:K3-surfaces-4}, all non-Gorenstein singular points of the 3-fold $X$ are cyclic quotient singularities of type  $\frac{1}{2}(1,1,1)$. As in Corollary~\ref{corollary:1-2-points}, let $\Sigma$ be the set of these singular points, and let $N=|\Sigma|$. Then
$N\in\{2,4,6,7,8,9,10,11,12,13,14,15\}$, and all possible decompositions of $\Sigma$ into $G$-orbits are described in  Corollary~\ref{corollary:1-2-points}. Let $E_{1},\ldots,E_{N}$ be the $\pi$-exceptional surfaces, and we set $\overline{E}_i=\phi(E_i)$ for every $i$. Then each $E_i\simeq\PP^2$, each $\overline{E}_i$ is a plane in $\overline{X}$, and the induced morphism $E_i\to\overline{E}_i$ is an isomorphism.
Note that  if $\overline{E}_i\cap \overline{E}_j\neq \varnothing$ for $i\neq j$, then
$\overline{E}_i\cap \overline{E}_j$ is a singular point of $\overline{X}$.

\begin{corollary}
\label{corollary:2-planes-through-point}
Let $P$ be a point in $\overline{X}$. Then at most two planes among $\overline{E}_1,\ldots,\overline{E}_N$ contain~$P$.
\end{corollary}

\begin{proof}
Suppose that three planes among $\overline{E}_1,\ldots,\overline{E}_N$ contains~$P$. Let us seek for a contradiction. Without loss of generality, we may assume that $P=\overline{E}_1\cap\overline{E}_2\cap \overline{E}_3$. Let $T_P\subset\PP^N$ be the embedded tangent space to $\overline{X}$ 
at~$P$. Then $T_P\simeq\PP^4$, and $\overline{E}_1\cup\overline{E}_2\cup\overline{E}_3\subset T_P\cap\overline{X}$. Set $Z=T_P\cap\overline{X}$. Since $\overline{X}$ is an intersection of quadrics, $Z$ is an intersection of quadrics as well, and dimensions of its irreducible components are at most two. Now, intersecting $Z$ with a general hyperplane $H\subset T_P$, we obtain three skew lines $H\cap \overline{E}_1$, $H\cap \overline{E}_2$, $H\cap \overline{E}_3$ in $H\simeq\PP^3$ such that they are contained in the intersection of quadrics $H\cap Z$, whose irreducible components are at most one-dimensional. This is impossible, since three skew lines in $\PP^3$ are contained in a unique quadric.
\end{proof}

By Theorem~\ref{theorem:PSL27-real-Gorenstein}, $X$ is not $G$-birational to a real $G\mathbb{Q}$-Fano \mbox{3-fold} with terminal Gorenstein singularities. However, $X$ could be $G$-birational to another real non-Gorenstein $G\mathbb{Q}$-Fano \mbox{3-fold}, but everything we proved so far for $X$ would be true for any real non-Gorenstein $G\mathbb{Q}$-Fano \mbox{3-fold} that is $G$-birational to $X$. In particular, we know  from Lemma~\ref{lemma:g-6-8-9-etc} that the dimension of its anticanonical linear system is bounded above by $38$.

\begin{proposition}
\label{proposition:reduction}
Suppose that 
\begin{equation}
\label{equation:dim-K}
\dim\big(|-K_X|\big)\geqslant \dim\big(|-K_{X^\prime}|\big)
\end{equation}
for any real non-Gorenstein $G\mathbb{Q}$-Fano 3-fold $X^\prime$ such that $X^\prime$ is $G$-birational to $X$.  Then $\phi$ is not an isomorphism, $g\in\{6,10,11,12\}$, $N\in\{7,14\}$, and one of the following cases holds:
\begin{itemize}
\item $N=7$, and $\Sigma$ is the $G$-orbit of a real point,
\item $N=14$, and $\Sigma$ is union of two complex conjugate $G$-orbits of length $7$,
\item $N=14$, and $\Sigma$ is the $G$-orbit of a (possibly complex) point.
\end{itemize}
\end{proposition}

In the remaining part of the section, we prove Proposition~\ref{proposition:reduction}.
Replacing (if necessary) $X$ by a real non-Gorenstein $G\mathbb{Q}$-Fano 3-fold that is $G$-birational to $X$, we may assume that \eqref{equation:dim-K} holds. Let us show that $X$ satisfies all conditions of Proposition~\ref{proposition:reduction}. We start with

\begin{lemma}
\label{lemma:phi}
The map $\phi$ is not an isomorphism.
\end{lemma}

\begin{proof}
Without loss of generality, we may assume that the surface $E_1+\cdots+E_{k}$ is $G$-irreducible for some $k\in\{2,7,14\}$ such that $k\leqslant N$. Let $f\colon\widetilde{X}\to X$ be the blow up of the points $\pi(E_1),\ldots,\pi(E_k)$, and let $\widetilde{E}_1,\ldots,\widetilde{E}_k$ be the $f$-exceptional surfaces that are mapped to $\pi(E_1),\ldots,\pi(E_k)$, respectively. Then $f$ is $G$-equivariant, $\widetilde{X}$ is smooth near $\widetilde{E}_1\cup\cdots\cup\widetilde{E}_k$, and there exists a $G$-equivariant birational morphism $g\colon Y\to\widetilde{X}$ such that $\pi=f\circ g$. If $k=N$, then $\widetilde{X}=Y$, and $g$ is just an identity map. If $k<N$, then $g$ is a blow up of $N-k$ non-Gorenstein singular points of the 3-fold $\widetilde{X}$.

Suppose $\phi$ is an isomorphism. Then $-K_Y$ is ample, so $-K_{\widetilde{X}}$ is also ample. The $G$-invariant Mori cone of $\widetilde{X}$ is two-dimensional, and one of its generators is the ray contracted by $f$. Let $h\colon \widetilde{X}\to X^\prime$ be the $G$-equivariant contraction of another ray. Then $h$ is birational by Lemma~\ref{lemma:fibration}, so either $h$ is small or $h$ contracts a $G$-irreducible surface.

Suppose that $h$ is small. Let $\widetilde{C}$ be an irreducible curve contracted by $h$. Then it follows from \cite[Theorem 4.2]{KM:92} and \cite[Corollary 4.4.5]{KM:92} that $-K_{\widetilde{X}}\cdot\widetilde{C}=\frac{1}{2}$, so $k<N$, and $\widetilde{C}$ contains at least one non-Gorenstein singular point of $\widetilde{X}$. Then the strict transform on $Y$ of the curve $\widetilde{C}$ has non-positive intersection with $-K_Y$, which is impossible, since $-K_Y$ is \sout{not} ample. 

We see that  $h$ contracts a $G$-irreducible surface $F$. Then $X^\prime$ is a $G\mathbb{Q}$-Fano 3-fold, and it follows from Theorem~\ref{theorem:PSL27-real-Gorenstein} that $X^\prime$ is non-Gorenstein. We know that all non-Gorenstein points of $X^\prime_{\CC}$ are cyclic quotient singularities of type $\frac{1}{2}(1,1,1)$, and they form the base locus of the non-empty mobile linear system $|-K_{X^\prime}|$. If $h(F)$ is a curve, then this curve is contained in the base locus of the strict transform on $X^\prime$ of the linear system $|-K_Y|$, which is contained in $|-K_{X^\prime}|$. Thus, we have $\dim(|-K_X|)=\dim(|-K_Y|)<\dim(|-K_{X^\prime}|)$, which contradicts \eqref{equation:dim-K}. Hence, we see that $h(F)$ is the $G$-orbit of a point. But $F\cap \widetilde{E}_1\ne\varnothing$, and $F\cap \widetilde{E}_1$ is a curve, because $\widetilde{E}_1$ is contained in the smooth locus of $\widetilde{X}$. This is a contradiction, since $\widetilde{E}_1\simeq\PP^2$ has no contractible curves.
\end{proof}

Since $\phi$ is small, two distinct planes among $\overline{E}_1,\ldots,\overline{E}_N$ are either disjoint or intersect by a point. Moreover, if two such planes intersect by a point, this point is a singular point of the 3-fold $\overline{X}$.

\begin{lemma}
\label{lemma:G-fixed-points-again}
Suppose that $E_i$ is $G$-invariant. Then $E_i$ is disjoint from $\mathrm{Exc}(\phi)$
\end{lemma}

\begin{proof}
We may assume that $E_1$ is $G$-invariant, and there exists a curve $C\subset Y$ such that $E_1\cap C\ne\varnothing$, but $\phi(C)$ is a point in $\overline{E}_1$. Let us seek for a contradiction.

By Lemma~\ref{lemma:Klein-action-on-surfaces}, the (complex) plane $\overline{E}_1$ does not contains $G$-orbits of length less than $21$. On the other hand, by Corollary~\ref{corollary:1-2-points}, the $G$-orbit of each plane $\overline{E}_i$ consists of $1$, $7$ or $14$ planes. Thus, we see the plane $\overline{E}_1$ is disjoint from the remaining planes $\overline{E}_2,\ldots,\overline{E}_N$.

Now we observe that $\phi(C)$ is a singular point of the 3-fold $\overline{X}$, and $E_1\cap C$ is a single point, because $\phi$ induces an isomorphism $E_i\to\overline{E}_i$. Therefore, the plane $\overline{E}_1$ contains at least $21$ singular points of the 3-fold $\overline{X}$. Since the plane $\overline{E}_1$ is not real, its complex conjugate is disjoint $\overline{E}_1$ and also contains at least $21$ singular points of the 3-fold $\overline{X}$, which contradicts Corollary~\ref{corollary:Sing}.
\end{proof}

It follows from Lemma~\ref{lemma:G-fixed-points-again} that $X_\CC$ has non-Gorenstein singular points that are not $G$-fixed. 

\begin{lemma}
\label{lemma:G-reduction-1}
The 3-fold $X_{\CC}$ has no $G$-fixed non-Gorenstein singular points.
\end{lemma}

\begin{proof}
Suppose that $G$ fixes a non-Gorenstein singular point $P_1\in X_{\CC}$. Then $P_1$ is a not real, and we let $P_2$ be its complex conjugate, which is also fixed by the group $G$. Without loss of generality, we may assume that $P_1=\pi(E_1)$ and $P_2=\pi(E_2)$.

Let $f\colon\widetilde{X}\to X$ be the blow up of the points $P_1$ and $P_2$, and let $\widetilde{E}_1$ and $\widetilde{E}_2$ be the $f$-exceptional surfaces that are mapped to $P_1$ and $P_2$, respectively. Then $f$ is $G$-equivariant and defined over $\mathbb{R}$, the 3-fold $\widetilde{X}$ is smooth near $\widetilde{E}_1\cup\widetilde{E}_2$, and there exists a $G$-birational morphism $g\colon Y\to\widetilde{X}$ that contracts $E_3,\ldots,E_N$ to non-Gorenstein singular points of the 3-fold $\widetilde{X}$.  We claim that $\widetilde{X}$ is a Fano 3-fold. Indeed, if there exists a curve $\widetilde{C}\subset \widetilde{X}$ such that $-K_{\widetilde{X}}\cdot \widetilde{C}\leqslant 0$, then $\widetilde{C}\cap (\widetilde{E}_1\cup \widetilde{E}_2)\ne \varnothing$, so its strict transform on $Y$ has non-positive intersection with $-K_Y$, which implies that it is one of the curves contracted by the birational morphism $\phi$, but these curves are disjoint from $E_1\cup E_2$ by Lemma~\ref{lemma:G-fixed-points-again}, so $\widetilde{C}$ is disjoint from $\widetilde{E}_1\cup\widetilde{E}_2$, which is a contradiction.

Since the $G$-invariant part of the $\mathrm{Pic}(\widetilde{X})$ has rank $2$, it follows from Lemma~\ref{lemma:fibration} that there exists a $G$-equivariant birational morphism $h\colon \widetilde{X}\to X^\prime$ such that either $h$ is small, or $h$ contracts a $G$-irreducible surface different from $\widetilde{E}_1+\widetilde{E}_2$, and $X^\prime$ is a $G\mathbb{Q}$-Fano 3-fold. However, it follows from Lemma~\ref{lemma:G-fixed-points-again} that $h$ is not small. Indeed, if $h$ is small, then it follows from \cite[Theorem 4.2]{KM:92} that $h$ contracts a curve $\widetilde{C}\subset\widetilde{X}$ such that $-K_{\widetilde{X}}\cdot \widetilde{C}=\frac{1}{2}$, which implies that the strict transform of this curve on $Y$ has trivial intersection with the divisor $-K_Y$ and, therefore,  must be $\phi$-exceptional, which is impossible by Lemma~\ref{lemma:G-fixed-points-again}, because  $\widetilde{C}\cap (\widetilde{E}_1\cup \widetilde{E}_2)\ne \varnothing$.

Thus, we see that $h$ contracts a $G$-irreducible surface, and $X^\prime$ is a $G\mathbb{Q}$-Fano 3-fold. Now, arguing as in the proof of Lemma~\ref{lemma:phi}, we see that $h$ contracts a $G$-irreducible surface to a curve in $X^\prime$, which implies that $\dim(|-K_X|)<\dim(|-K_{X^\prime}|)$. This contradicts our assumption \eqref{equation:dim-K}.
\end{proof}

Thus, by Lemma~\ref{lemma:G-reduction-1} and Corollary~\ref{corollary:1-2-points}, $X_\CC$ has $7$ or $14$ non-Gorenstein singular points, which form one or two $G$-orbits. In our notations, $N=7$ or $N=14$.

\begin{lemma}
\label{lemma:21}
The 3-fold $\overline{X}_{\CC}$ has at least $21$ singular points.
\end{lemma}

\begin{proof}
Let $k=|\mathrm{Sing}(\overline{X}_{\CC})\cap\overline{E}_1|$, and let $G_1$ be the stabilizer of $\overline{E}_1$ in $G$. Then $\overline{E}_1$ does not contain $G_1$-orbits of length $1$ and $2$. Thus,  $k\geqslant 3$. On the other hand, at most two planes among $\overline{E}_1,\ldots,\overline{E}_{N}$ can pass through a singular point of $\overline{X}$ by Corollary~\ref{corollary:2-planes-through-point}. Thus, the union $\overline{E}_1\cup\ldots\cup\overline{E}_{N}$ contains at least $(N\cdot k)/2$ singular points of the 3-fold $\overline{X}$. So, we are done in the case $N=14$.

Suppose that $N=7$. If  $\overline{E}_1,\ldots,\overline{E}_{7}$ are disjoint, then $\overline{E}_1\cup\ldots\cup\overline{E}_{7}$ contains $7\cdot k\geqslant 21$ singular points of the 3-fold $\overline{X}$. If $\overline{E}_1,\ldots,\overline{E}_{7}$ are not disjoint, then, since $G$ acts doubly transitively on them, we have $k=6$, and $\overline{E}_1\cup\ldots\cup\overline{E}_{7}$ contains $21$ singular points of the 3-fold $\overline{X}$.
\end{proof}

\begin{corollary}
\label{corollary:g}
One has $g\in\{6,10,11,12\}$.
\end{corollary}

\begin{proof}
By \eqref{equation:Sing} and Lemma~\ref{lemma:21}, we have $\rho(V)+1\leqslant h^{1,2}(V)$, where $V$ is a smooth complex Fano 3-fold of Fano index $1$ such that $(-K_V)^3=2g-2$. By Lemma~\ref{lemma:g-6-8-9-etc},  $g\in\{6,10,11,\ldots,33\}$. Now, going through the list of smooth Fano 3-folds of index $1$, we get $g\leqslant 12$.
\end{proof}

Thus, to finish the proof of Proposition~\ref{proposition:reduction} it is enough to prove the following result. 

\begin{lemma}
\label{lemma:G-reduction-2}
Suppose that $N=14$. Then $\Sigma$ is not a union of two real $G$-orbits of length $7$.
\end{lemma}

\begin{proof}
Suppose that $\Sigma$ is a union of two real $G$-orbits of length $7$.
We may assume that one of these two orbits consists of the points $\pi(E_1),\ldots,\pi(E_7)$. Let $f\colon\widetilde{X}\to X$ be the blow up of these points, and let $\widetilde{E}_1,\ldots,\widetilde{E}_7$ be the $f$-exceptional surfaces that are mapped to $\pi(E_1),\ldots,\pi(E_7)$, respectively. Then $f$ is $G$-equivariant and defined over $\mathbb{R}$, the 3-fold $\widetilde{X}$ is smooth near $\widetilde{E}_1\cup\cdots\cup\widetilde{E}_7$, and there exists a $G$-equivariant birational morphism $g\colon Y\to\widetilde{X}$ such that $\pi=f\circ g$, which blows up all non-Gorenstein singular points of the 3-fold $\widetilde{X}$. Note that $-K_{\widetilde{X}}$ is nef and big.

First, we consider the case when $-K_{\widetilde{X}}$ is ample. Arguing as in the proof of Lemma~\ref{lemma:phi}, we see that there exists a $G$-equivariant birational morphism $h\colon \widetilde{X}\to X^\prime$ such that either $h$ is small, or $h$ contracts a $G$-irreducible surface different from $\widetilde{E}_1+\cdots+\widetilde{E}_7$, and $X^\prime$ is a $G\mathbb{Q}$-Fano 3-fold. In the latter case, it also follow from the proof of Lemma~\ref{lemma:phi} that $h$ contracts this surface to a curve, so $\dim(|-K_X|)<\dim(|-K_{X^\prime}|)$, which contradicts \eqref{equation:dim-K}. Therefore, $h$ must be small. Let $\widetilde{C}$ be an irreducible curve in $\widetilde{X}_{\CC}$ that is contracted by $h$. Then, by \cite[Theorem 4.2]{KM:92}, $\widetilde{C}$ is smooth, and it contains one non-Gorenstein singular point of $\widetilde{X}_{\CC}$. Without loss of generality, we may assume that $\g(E_8)\in \widetilde{C}$. Let $C$ be the strict transform of this curve on $Y_{\CC}$. Then $C\cap E_8$ is a single point. On the other hand, the stabilizer of  $E_8$ in $G$ does not fix points in $E_8$, so the $G$-orbit of the curve $\widetilde{C}$ contains another curve $\widetilde{C}^\prime$ such that $\g(E_8)\in\widetilde{C}\cap\widetilde{C}^\prime$ and both of these curves are contracted by $h$, which is impossible by \cite[Theorem 4.2]{KM:92} and \cite[Corollary 4.4.5]{KM:92}. 

Hence, we see that $-K_{\widetilde{X}}$ is not ample. Now, using Proposition~\ref{proposition:terminal-singularities}, we see that the linear system $|m(-K_{\widetilde{X}})|$ gives a small birational morphism. Thus, it follows from Lemma~\ref{lemma:fibration} that there exists the following $G$-Sarkisov link:
$$
\xymatrix{
X&&\widetilde{X}\ar[ll]_{f}\ar@{-->}[rr]^{\theta}&&\widetilde{X}^\prime\ar[rr]^{f^\prime}&&X^\prime}
$$
such that $\theta$ is a small non-biregular birational map that is a composition of flops, and
\begin{itemize}
\item either $f^\prime$ is a small birational contraction (flip),
\item or $f^\prime$ is a divisorial birational contraction, and $X^\prime$ is a real $G\mathbb{Q}$-Fano 3-fold.
\end{itemize}
Note that the curves flopped by $\theta$ are disjoint from the non-Gorenstein singular points of the $\widetilde{X}_{\CC}$, because otherwise their strict transforms on $Y_{\CC}$ would have a negative intersection with $-K_Y$. Therefore, arguing as above, we see that $f^\prime$ is not small, so $f^\prime$ contracts a divisor, and $X^\prime$ is a real $G\mathbb{Q}$-Fano 3-fold. By Theorem~\ref{theorem:PSL27-real-Gorenstein}, $X^\prime$ is non-Gorenstein, so all non-Gorenstein singular points of $X^\prime_{\CC}$ are cyclic quotient singularities of type $\frac{1}{2}(1,1,1)$, and the union of these points is the base locus of the non-empty mobile linear system $|-K_{X^\prime}|$. Moreover, it follows from \eqref{equation:dim-K} that $|-K_{X^\prime}|$ is the strict transform  of the linear system $|-K_X|$.

Let $\widehat{E}^\prime$ be the $G$-irreducible real surface contracted by $f^\prime$. Then $f^\prime(\widehat{E}^\prime)$ is contained in the base locus of the linear system $|-K_{X^\prime}|$, which implies that either $f^\prime(\widehat{E}^\prime)$ is the $G$-orbit of a non-Gorenstein singular point of $X^\prime_{\CC}$, or $f^\prime(\widehat{E}^\prime)$ is a union of two $G$-orbits of two non-Gorenstein complex conjugated singular points of $X^\prime_{\CC}$. In both cases, it follows from \cite[Theorem 8.3.1]{P:G-MMP} that $f^\prime$ is a blow up of these non-Gorenstein singular points of $X^\prime$ (see also \cite{Kawamata1994}).

Let $N^\prime$ be the number of geometrically irreducible components of $\widehat{E}^\prime_{\CC}$,
let $\widetilde{E}=\widetilde{E}_1+\cdots+\widetilde{E}_7$, and let $\widetilde{E}^\prime$ be the strict transform on $\widetilde{X}$ of the surface $\widehat{E}^\prime$. Then $N^\prime\in\{2,7,14\}$ by Corollary~\ref{corollary:1-2-points}. We have $\widetilde{E}^\prime\sim_{\mathbb{Q}}a(-K_{\widetilde{X}})-b\widetilde{E}$ for some $a,b\in\mathbb{Q}$. Moreover, $a>0$, since $\widetilde{E}\ne\widetilde{E}^\prime$. Now, since $\theta$ is a composition of flops, we have
\begin{eqnarray*}
N^\prime&=&\widehat{E}^\prime\cdot(-K_{\widehat{X}^\prime})^2=\big(a(-K_{\widetilde{X}})-b\widetilde{E}\big)\cdot(-K_{\widetilde{X}})^2=\Big(2g+\frac{3}{2}\Big)a-7b,
\\
-2N^\prime&=&\big(\widehat{E}^\prime\big)^2\cdot(-K_{\widehat{X}^\prime})=\big(a(-K_{\widetilde{X}})-b\widetilde{E}\big)^2\cdot(-K_{\widetilde{X}})=\Big(2g+\frac{3}{2}\Big)a^2-14b^2-14ab,
\end{eqnarray*}
where $g\in\{6,10,11,12\}$ by Corollary~\ref{corollary:g}. Solving these equations with rational $a>0$ and~$b$, we see that $N^\prime=7$, $b=1$ and
$a=\frac{28}{4g+3}$. Set $E^\prime=f(\widetilde{E}^\prime)$. Then $E^\prime\sim_{\mathbb{Q}} a(-K_{X})$. On the other hand, both $2E^\prime$ and $2K_X$ are  Cartier divisors, which gives $(4g+3)(2E^\prime)\sim 28(2(-K_{X}))$, because the group $\mathrm{Pic}(X)$ is torsion free. Since $4g+3$ and $28$ are coprime, we see that $2K_{X}$ is divisible by $4g+3\geqslant 27$ in $\mathrm{Pic}(X)$, which is impossible by \cite{Sano}.
\end{proof}

\subsection{Equivariant Sarkisov link}
\label{subsection:link}

We continue our analysis in Section~\ref{subsection:anticanonical-map}. We may assume that \eqref{equation:dim-K} holds. Then, by Proposition~\ref{proposition:reduction}, we may assume that $\phi$ is not an isomorphism, $g\in\{6,10,11,12\}$, and one of the following holds:
\begin{itemize}
\item $N=7$, and $\Sigma$ is a $G$-orbit of length $7$ consisting or real points;
\item $N=14$, and $\Sigma$ is union of two complex conjugate $G$-orbits of length $7$;
\item $N=14$, and $\Sigma$ is a $G$-orbit of length $14$.
\end{itemize}
Thus, we have $\rho^G(Y)=2$. Hence, it follows from Lemma~\ref{lemma:fibration} and \cite[Theorem 4.2]{KM:92} that the commutative diagram in Lemma~\ref{lemma:base-points} can be expanded to the following $G$-Sarkisov link:
$$
\xymatrix{
&&Y\ar[dll]_{\pi}\ar[rrd]_{\phi}\ar@{-->}[rrrr]^{\chi}&&&&Y^\prime\ar[drr]^{\pi^\prime}\ar[lld]^{\phi^\prime}&&\\
X\ar@{-->}[rrrr]_{\psi}&&&&\overline{X}&&&&X^\prime}
$$
where $\chi$ is a composition of flops, $Y^\prime$ is an almost Fano 3-fold, $\phi^\prime$ is a small contraction, $\pi^\prime$ is a divisorial birational contraction, and $X^\prime$ is a real $G\mathbb{Q}$-Fano 3-fold. Then $X^\prime$ is non-Gorenstein by Theorem~\ref{theorem:PSL27-real-Gorenstein}. We know that all non-Gorenstein singular points of $X^\prime_{\CC}$ are cyclic quotient singularities of type $\frac{1}{2}(1,1,1)$. Moreover, arguing as in the proof of Lemma~\ref{lemma:G-reduction-2}, we see that $\pi^\prime$ is the blow up of all non-Gorenstein singular points of $X^\prime$. 

Let $N^\prime$  be the number of non-Gorenstein singular points of the 3-fold $X^\prime$. Then $N^\prime\in\{2,7,14\}$ by Corollary~\ref{corollary:1-2-points}. Let $E$ and $E^\prime$ be the $G$-irreducible real surfaces contracted by $\pi$ and $\pi^\prime$, respectively. Denote by $E^\prime_{Y}$ be the strict transform in $Y$ of the surface $E^\prime$. Then 
$$
E^\prime_{Y}\sim_{\mathbb{Q}}a(-K_{Y})-bE
$$ 
for some rational numbers $a$ and $b$ such that $a>0$, because $E^\prime_{Y}\ne E$. Moreover, we have
\begin{eqnarray*}
N^\prime&=&E^\prime_{Y}\cdot (-K_{Y^\prime})^2=\big(a(-K_{Y})-bE\big)\cdot(-K_{Y})^2=\Big(2g-2\Big)a-Nb,
\\
-2N^\prime&=&\big(E^\prime_{Y}\big)^2\cdot (-K_{Y^\prime})=\big(a(-K_{Y})-bE\big)^2\cdot(-K_{Y})=\Big(2g-2\Big)a^2-2b^2N-2abN.
\end{eqnarray*}
Solving these equations with rational $a>0$ and $b$ for every possible $N\in\{7,8\}$, $N^\prime\in\{2,7,14\}$ and $g\in\{6,10,11,12\}$, we get $N=N^\prime$, $b=1$ and $a=\frac{N}{g-1}$. Set $\overline{E}=\phi(E)$ and $\overline{E}^\prime=\phi(E^\prime_{Y})$. Then
$$
\overline{E}+\overline{E}^\prime\sim_{\mathbb{Q}} \frac{N}{g-1}(-K_{\overline{X}}),
$$
which implies that  $\overline{E}+\overline{E}^\prime$ is Cartier. Hence, since $\mathrm{Pic}(\overline{X})$ has no torsion, we get 
$$
(g-1)(\overline{E}+\overline{E}^\prime)\sim N(-K_{\overline{X}}).
$$ 
If $g-1$ and $N$ are co-prime, then $-K_{\overline{X}}$ is divisible by $g-1\geqslant 5$ in $\mathrm{Pic}(\overline{X})$, which is impossible, since $\iota(\overline{X})\leqslant 4$. Thus, $N=14$ and $g=11$, so $5(\overline{E}+\overline{E}^\prime)\sim 7(-K_{\overline{X}})$. Then $-K_{\overline{X}}$ is divisible by $5$ in $\mathrm{Pic}(\overline{X})$, which is impossible. The obtained contradiction completes the proof of Theorem~\ref{theorem:PSL27-real-non-Gorenstein}.

\bibliography{cr}

@Article{Alexeev:ge,
  author     = {Alexeev, Valery},
  title      = {General elephants of {$\mathbf{Q}$}-{F}ano 3-folds},
  journal    = {Compositio Math.},
  year       = {1994},
  volume     = {91},
  number     = {1},
  pages      = {91--116},
  issn       = {0010-437X},
  fjournal   = {Compositio Mathematica},
  language   = {english},
  mrclass    = {14E30 (14E35 14J45)},
  mrnumber   = {1273928},
  mrreviewer = {Yujiro Kawamata},
  url        = {http://www.numdam.org/item?id=CM_1994__91_1_91_0},
}

@Article{Alzati-Bertolini-1992a,
  Title                    = {{On the rationality of {F}ano 3-folds with {$B\sb 2{\ge}2$}}},
  Author                   = {Alberto Alzati and Marina Bertolini},
  Journal                  = {Matematiche},
  Year                     = {1992},
  Number                   = {1},
  Pages                    = {63--74},
  Volume                   = {47},

  Fjournal                 = {{Le Matematiche}},
  ISSN                     = {0373-3505; 2037-5298/e},
  Language                 = {english},
  Msc2010                  = {14J30 14M20 14J45},
  Publisher                = {Universit{\`a} di Catania, Dipartimento di Matematica e Informatica, Catania},
  Zbl                      = {0787.14023}
}

@InCollection{Reid:YPG,
  author     = {Reid, Miles},
  title      = {Young person's guide to canonical singularities},
  booktitle  = {Algebraic geometry, {B}owdoin, 1985 ({B}runswick, {M}aine, 1985)},
  publisher  = {Amer. Math. Soc., Providence, RI},
  year       = {1987},
  volume     = {46},
  series     = {Proc. Sympos. Pure Math.},
  pages      = {345--414},
  language   = {English},
  mrclass    = {14E30 (14B05 14E05 14J10)},
  mrnumber   = {927963},
  mrreviewer = {Eckart Viehweg},
}

@InCollection{Reid:MM,
  author    = {Reid, Miles},
  title     = {Minimal models of canonical {$3$}-folds},
  booktitle = {Algebraic varieties and analytic varieties ({T}okyo, 1981)},
  publisher = {North-Holland, Amsterdam},
  year      = {1983},
  volume    = {1},
  series    = {Adv. Stud. Pure Math.},
  pages     = {131--180},
  doi       = {10.2969/aspm/00110131},
  language  = {English},
  mrclass   = {14E30 (14J30)},
  mrnumber  = {715649},
  url       = {https://doi.org/10.2969/aspm/00110131},
}

@Article{Hayakawa-Takeuchi-1987,
  Title                    = {{On canonical singularities of dimension three}},
  Author                   = {Takayuki Hayakawa and Kiyohiko Takeuchi},
  Journal                  = {Japan. J. Math. (N.S.)},
  Year                     = {1987},
  Number                   = {1},
  Pages                    = {1--46},
  Volume                   = {13},

  Coden                    = {JJMAAK},
  Fjournal                 = {Japanese Journal of Mathematics. New Series},
  ISSN                     = {0289-2316},
  Language                 = {english},
  Mrclass                  = {14B05 (14L30 32B30)},
  Mrnumber                 = {MR914313 (88j:14003)},
  Mrreviewer               = {David R. Morrison}
}

@manual{GAP4,
    key          = "GAP",
    organization = "The GAP~Group",
    title        = "{GAP -- Groups, Algorithms, and Programming,
                    Version 4.4.12}",
    year         = 2008,
    url          = "\verb+(http://www.gap-system.org)+",
    }

@book{Breuer,
 author = {Breuer, Thomas},
 title = {Characters and automorphism groups of compact {Riemann} surfaces},
 fseries = {London Mathematical Society Lecture Note Series},
 series = {Lond. Math. Soc. Lect. Note Ser.},
 issn = {0076-0552},
 volume = {280},
 isbn = {0-521-79809-4},
 year = {2000},
 publisher = {Cambridge: Cambridge University Press},
 language = {English},
 zbMATH = {1505376},
 Zbl = {0952.30001}
}

@article{Kawamata-88,
 author = {Kawamata, Yujiro},
 title = {Crepant blowing-up of 3-dimensional canonical singularities and its application to degenerations of surfaces},
 fjournal = {Annals of Mathematics. Second Series},
 journal = {Ann. Math. (2)},
 issn = {0003-486X},
 volume = {127},
 number = {1},
 pages = {93--163},
 year = {1988},
 language = {English},
 doi = {10.2307/1971417},
 zbMATH = {4061379},
 Zbl = {0651.14005}
}

@misc{Maschke,
 author = {Maschke, H.},
 title = {Invariants of a group of 2.168 linear quaternary substitutions.},
 year = {1896},
 language = {English},
 howpublished = {Chicago {Congr}. {Papers}, 175-186 (1896).},
 zbMATH = {2672849},
 JFM = {28.0116.01}
}

@article{SashaYura-del-Pezzo,
 author = {Kuznetsov, Alexander G. and Prokhorov, Yuri G.},
 title = {On higher-dimensional del {Pezzo} varieties},
 fjournal = {Izvestiya: Mathematics},
 journal = {Izv. Math.},
 issn = {1064-5632},
 volume = {87},
 number = {3},
 pages = {488--561},
 year = {2023},
 language = {English},
 doi = {10.4213/im9385e},
 zbMATH = {7745501},
 Zbl = {1540.14083}
}

@InProceedings{Kawamata:bF,
  author     = {Kawamata, Yujiro},
  title      = {Boundedness of {$\mathbf{Q}$}-{F}ano threefolds},
  booktitle  = {Proceedings of the {I}nternational {C}onference on {A}lgebra, {P}art 3 ({N}ovosibirsk, 1989)},
  year       = {1992},
  volume     = {131},
  series     = {Contemp. Math.},
  pages      = {439--445},
  publisher  = {Amer. Math. Soc., Providence, RI},
  language   = {english},
  mrclass    = {14J45 (14E30)},
  mrnumber   = {1175897},
  mrreviewer = {Jaros\l aw A. Wi\'{s}niewski},
}

@Article{P:JAG:simple,
  author    = {Yu. Prokhorov},
  title     = {Simple finite subgroups of the {C}remona group of rank $3$},
  journal   = {J. Algebraic Geom.},
  year      = {2012},
  volume    = {21},
  number    = {3},
  pages     = {563--600},
  doi       = {10.1090/S1056-3911-2011-00586-9},
  language  = {english},
}

@Misc{GroupNames,
  author = {Tim Dokchitser},
  title  = {Group Names (webpage)},
  url    = {https://people.maths.bris.ac.uk/~matyd/GroupNames/index.html},
 howpublished = {avalable at \url {https://people.maths.bris.ac.uk/~matyd/GroupNames/index.html}},
}

@Misc{GRDB,
  author = {Altinok, Selma and Brown, Gavin and Iano-Fletcher, Anthony and Kasprzyk, Alexander and Reid, Miles and Suzuki, Kaori},
  title  = {Graded Ring Database (online database)},
  url    = {http://www.grdb.co.uk/},
 howpublished = {avalable at \url {http://www.grdb.co.uk/}},
}

@article{Tyurin,
 author = {Tyurin, A. N.},
 title = {Five lectures on three-dimensional varieties},
 fjournal = {Russian Mathematical Surveys},
 journal = {Russ. Math. Surv.},
 issn = {0036-0279},
 volume = {27},
 number = {5},
 pages = {1--53},
 year = {1973},
 language = {English},
 doi = {10.1070/rm1972v027n05ABEH001384},
 zbMATH = {3414436},
 Zbl = {0263.14012}
}

@article{Takagi1,
 author = {Takagi, Hiromichi},
 title = {On classification of $\mathbb{Q}$-{Fano} {{\(3\)}}-folds of {Gorenstein} index {{\(2\)}}. {I}.},
 fjournal = {Nagoya Mathematical Journal},
 journal = {Nagoya Math. J.},
 issn = {0027-7630},
 volume = {167},
 pages = {117--155},
 year = {2002},
 language = {English},
 doi = {10.1017/S0027763000025460},
 zbMATH = {1995490},
 Zbl = {1048.14022}
}

@article{Takagi2,
 author = {Takagi, Hiromichi},
 title = {On classification of $\mathbb{Q}$-{Fano} {{\(3\)}}-folds of {Gorenstein} index {{\(2\)}}. {II}.},
 fjournal = {Nagoya Mathematical Journal},
 journal = {Nagoya Math. J.},
 issn = {0027-7630},
 volume = {167},
 pages = {157--216},
 year = {2002},
 language = {English},
 doi = {10.1017/S0027763000025472},
 zbMATH = {1995491},
 Zbl = {1048.14023}
}

@Article{KMMT-2000,
  Title                    = {Boundedness of canonical {$\mathbf{Q}$}-{F}ano 3-folds},
  Author                   = {J{\'a}nos Koll{\'a}r and Yoichi Miyaoka and Shigefumi Mori and Hiromichi Takagi},
  Journal                  = {Proc. Japan Acad. Ser. A Math. Sci.},
  Year                     = {2000},
  Number                   = {5},
  Pages                    = {73--77},
  Volume                   = {76},

  Coden                    = {PJAADT},
  Fjournal                 = {Japan Academy. Proceedings. Series A. Mathematical Sciences},
  ISSN                     = {0386-2194},
  Mrclass                  = {14J45 (14E30)},
  Mrnumber                 = {MR1771144 (2001h:14053)},
  Mrreviewer               = {Yuri G. Prokhorov}
}

@article {Prokhorov-G-Fanos,
    AUTHOR = {Prokhorov, Yu. G.},
     TITLE = {On {$G$}-{F}ano threefolds},
   JOURNAL = {Izv. Ross. Akad. Nauk Ser. Mat.},
  FJOURNAL = {Izvestiya Rossiiskoi Akademii Nauk. Seriya Matematicheskaya},
    VOLUME = {79},
      YEAR = {2015},
    NUMBER = {4},
     PAGES = {159--174},
      ISSN = {1607-0046,2587-5906},
   MRCLASS = {14J45 (14E30 14J30)},
  MRNUMBER = {3397422},
MRREVIEWER = {Jaros\l aw\ A.\ Wi\'sniewski},
       DOI = {10.4213/im8349},
       URL = {https://doi.org/10.4213/im8349},
}

@article{Voisin,
 author = {Voisin, Claire},
 title = {Sur la jacobienne interm{\'e}diaire du double solide d'indice deux. ({On} the intermediate {Jacobian} of the double solid of index two)},
 fjournal = {Duke Mathematical Journal},
 journal = {Duke Math. J.},
 issn = {0012-7094},
 volume = {57},
 number = {2},
 pages = {629--646},
 year = {1988},
 language = {French},
 doi = {10.1215/S0012-7094-88-05728-6},
 zbMATH = {4144155},
 Zbl = {0698.14049}
}

@article{Grinenko1,
 author = {Grinenko, M. M.},
 title = {On the double cone over the {Veronese} surface.},
 fjournal = {Izvestiya: Mathematics},
 journal = {Izv. Math.},
 issn = {1064-5632},
 volume = {67},
 number = {3},
 pages = {421--438},
 year = {2003},
 language = {English},
 doi = {10.1070/IM2003v067n03ABEH000433},
 zbMATH = {2168350},
 Zbl = {1082.14015}
}

@incollection{Grinenko2,
 author = {Grinenko, M. M.},
 title = {Mori structures on a {Fano} threefold of index 2 and degree 1},
 booktitle = {Algebraic geometry. Methods, relations, and applications. Collected papers. Dedicated to the memory of Andrei Nikolaevich Tyurin.},
 pages = {103--128},
 year = {2004},
 publisher = {Moscow: Maik Nauka/Interperiodica},
 language = {English},
 zbMATH = {5071239},
 Zbl = {1120.14031}
}

@Article{P:GFano1,
  author    = {Yuri Prokhorov},
  title     = {{G-{F}ano threefolds, {I}}},
  journal   = {Adv. Geom.},
  year      = {2013},
  volume    = {13},
  number    = {3},
  pages     = {389--418},
  issn      = {1615-715X},
  coden     = {AGDEA3},
  doi       = {10.1515/advgeom-2013-0008},
  fjournal  = {Advances in Geometry},
  mrclass   = {14-XX},
  mrnumber  = {3100917},
  url       = {http://dx.doi.org/10.1515/advgeom-2013-0008},
}

@Article{P:GFano2,
  author    = {Yuri Prokhorov},
  title     = {{G-{F}ano threefolds, {II}}},
  journal   = {Adv. Geom.},
  year      = {2013},
  volume    = {13},
  number    = {3},
  pages     = {419--434},
  issn      = {1615-715X},
  coden     = {AGDEA3},
  doi       = {10.1515/advgeom-2013-0009},
  fjournal  = {Advances in Geometry},
  mrclass   = {14-XX},
  mrnumber  = {3100918},
  url       = {http://dx.doi.org/10.1515/advgeom-2013-0009},
}

@Book{IP99,
  title      = {Fano varieties. {A}lgebraic geometry {V}},
  publisher  = {Springer},
  year       = {1999},
  author     = {V.~A. Iskovskikh and Yu. Prokhorov},
  volume     = {47},
  series     = {{Encyclopaedia Math. Sci.}},
  address    = {Berlin},
  language   = {english},
  mrclass    = {14J45 (14E07 14F22 14K30)},
  mrnumber   = {1668579},
  mrreviewer = {Takao Fujita},
}

@incollection {Kawamata1994,
    AUTHOR = {Kawamata, Yujiro},
     TITLE = {Divisorial contractions to {$3$}-dimensional terminal quotient
              singularities},
 BOOKTITLE = {Higher-dimensional complex varieties ({T}rento, 1994)},
     PAGES = {241--246},
 PUBLISHER = {de Gruyter, Berlin},
      YEAR = {1996},
      ISBN = {3-11-014503-0},
   MRCLASS = {14B05 (14J30)},
  MRNUMBER = {1463182},
MRREVIEWER = {Tomasz\ Szemberg},
}

@incollection {KollarPairs,
    AUTHOR = {Koll\'ar, J\'anos},
     TITLE = {Singularities of pairs},
 BOOKTITLE = {Algebraic geometry---{S}anta {C}ruz 1995},
    SERIES = {Proc. Sympos. Pure Math.},
    VOLUME = {62, Part 1},
     PAGES = {221--287},
 PUBLISHER = {Amer. Math. Soc., Providence, RI},
      YEAR = {1997},
      ISBN = {0-8218-0894-X; 0-8218-0493-6},
   MRCLASS = {14E30 (14E15 32S40)},
  MRNUMBER = {1492525},
MRREVIEWER = {Alessio\ Corti},
       DOI = {10.1090/pspum/062.1/1492525},
       URL = {https://doi.org/10.1090/pspum/062.1/1492525},
}

@incollection {IgorVasya,
    AUTHOR = {Dolgachev, Igor V. and Iskovskikh, Vasily A.},
     TITLE = {Finite subgroups of the plane {C}remona group},
 BOOKTITLE = {Algebra, arithmetic, and geometry: in honor of {Y}u. {I}.
              {M}anin. {V}ol. {I}},
    SERIES = {Progr. Math.},
    VOLUME = {269},
     PAGES = {443--548},
 PUBLISHER = {Birkh\"auser Boston, Boston, MA},
      YEAR = {2009},
      ISBN = {978-0-8176-4744-5},
   MRCLASS = {14E07 (14J26)},
  MRNUMBER = {2641179},
MRREVIEWER = {Julie\ D\'eserti},
       DOI = {10.1007/978-0-8176-4745-2\_11},
       URL = {https://doi.org/10.1007/978-0-8176-4745-2_11},
}

@Article{P:G-MMP,
  author    = {Yuri Prokhorov},
  title     = {Equivariant minimal model program},
  journal   = {Russian Math. Surv.},
  year      = {2021},
  volume    = {76},
  number    = {3},
  pages     = {461--542},
  doi       = {10.1070/rm9990},
  language  = {english},
  publisher = {{IOP} Publishing},
  url       = {https://doi.org/10.1070/rm9990},
}

@Article{KP:1nodal,
  author    = {Alexander Kuznetsov and Yuri Prokhorov},
  title     = {$1$-nodal {F}ano threefolds with {P}icard number $1$},
  journal   = {Izvestiya: Mathematics},
  year      = {2025},
  volume    = {89},
  number    = {3},
  pages     = {495--594},
  issn      = {1468-4810},
  doi       = {10.4213/im9585e},
  publisher = {Steklov Mathematical Institute},
  url       = {https://doi.org/10.4213/im9585e},
}

@article{Janos-Nick,
 author = {Koll{\'a}r, J. and Shepherd-Barron, N. I.},
 title = {Threefolds and deformations of surface singularities},
 fjournal = {Inventiones Mathematicae},
 journal = {Invent. Math.},
 issn = {0020-9910},
 volume = {91},
 number = {2},
 pages = {299--338},
 year = {1988},
 language = {English},
 doi = {10.1007/BF01389370},
 url = {https://eudml.org/doc/143542},
 zbMATH = {4045919},
 Zbl = {0642.14008}
}

@misc{KMM,
 author = {Kawamata, Yujiro and Matsuda, Katsumi and Matsuki, Kenji},
 title = {Introduction to the minimal model problem},
 year = {1987},
 language = {English},
 howpublished = {Algebraic geometry, {Proc}. {Symp}., {Sendai}/{Jap}. 1985, {Adv}. {Stud}. {Pure} {Math}. 10, 283-360 (1987).},
 zbMATH = {4099460},
 Zbl = {0672.14006}
}

@article{Loginov,
 author = {Loginov, Konstantin},
 title = {A note on 3-subgroups in the space {Cremona} group},
 fjournal = {Communications in Algebra},
 journal = {Commun. Algebra},
 issn = {0092-7872},
 volume = {50},
 number = {9},
 pages = {3704--3714},
 year = {2022},
 language = {English},
 doi = {10.1080/00927872.2022.2042545},
 zbMATH = {7548116},
 Zbl = {1490.14063}
}

@article{Kuznetsova,
 author = {Kuznetsova, A. A.},
 title = {Finite 3-subgroups in the {Cremona} group of rank 3},
 fjournal = {Mathematical Notes},
 journal = {Math. Notes},
 issn = {0001-4346},
 volume = {108},
 number = {5},
 pages = {697--715},
 year = {2020},
 language = {English},
 doi = {10.1134/S0001434620110085},
 zbMATH = {7289062},
 Zbl = {1469.14028}
}

@misc{pointless,
 author = {Abban, Hamid and Cheltsov, Ivan and Kishimoto, Takashi and Mangolte, Frederic},
 title = {K-stability of pointless del {Pezzo} surfaces and {Fano} 3-folds},
 year = {2024},
 howpublished = {Preprint, {arXiv}:2411.00767 [math.{AG}] (2024)},
 url = {https://arxiv.org/abs/2411.00767},
 arXiv = {arXiv:2411.00767}
}

@article{BuBa,
 author = {Cheltsov, Ivan and Przyjalkowski, Victor and Shramov, Constantin},
 title = {Burkhardt quartic, {Barth} sextic, and the icosahedron},
 fjournal = {IMRN. International Mathematics Research Notices},
 journal = {Int. Math. Res. Not.},
 issn = {1073-7928},
 volume = {2019},
 number = {12},
 pages = {3683--3703},
 year = {2019},
 language = {English},
 doi = {10.1093/imrn/rnx204},
 zbMATH = {7130832},
 Zbl = {1454.14036}
}

@incollection{YuraS6,
 author = {Prokhorov, Yuri},
 title = {Embeddings of the symmetric groups to the space {Cremona} group},
 booktitle = {Birational geometry, K\"ahler-Einstein metrics and degenerations. Proceedings of the conferences, Moscow, Russia, April 8--13, 2019, Shanghai, China, June 10--14, 2019, Pohang, South Korea, November 18--22, 2019},
 isbn = {978-3-031-17858-0; 978-3-031-17861-0; 978-3-031-17859-7},
 pages = {749--762},
 year = {2023},
 publisher = {Cham: Springer},
 language = {English},
 doi = {10.1007/978-3-031-17859-7_38},
 zbMATH = {7828773},
 Zbl = {1547.14022}
}

@article{Edge1947,
 author = {Edge, W. L.},
 title = {The {Klein} group in three dimensions},
 fjournal = {Acta Mathematica},
 journal = {Acta Math.},
 issn = {0001-5962},
 volume = {79},
 pages = {153--223},
 year = {1947},
 language = {English},
 doi = {10.1007/BF02404696},
 zbMATH = {3046465},
 Zbl = {0029.41201}
}

@article{Mukai,
 author = {Mukai, Shigeru},
 title = {Finite groups of automorphisms of {K3} surfaces and the {Mathieu} group},
 fjournal = {Inventiones Mathematicae},
 journal = {Invent. Math.},
 issn = {0020-9910},
 volume = {94},
 number = {1},
 pages = {183--221},
 year = {1988},
 language = {English},
 doi = {10.1007/BF01394352},
 url = {https://eudml.org/doc/143625},
 zbMATH = {4156651},
 Zbl = {0705.14045}
}

@article{CheltsovShramovKlein,
 author = {Cheltsov, Ivan and Shramov, Constantin},
 title = {Three embeddings of the {Klein} simple group into the {Cremona} group of rank three},
 fjournal = {Transformation Groups},
 journal = {Transform. Groups},
 issn = {1083-4362},
 volume = {17},
 number = {2},
 pages = {303--350},
 year = {2012},
 language = {English},
 doi = {10.1007/s00031-012-9183-8},
 url = {www.pure.ed.ac.uk/ws/files/17675371/Three_embeddings_of_the_Klein_simple_group_into_the_Cremona_group_of_rank_three.pdf},
 zbMATH = {6113118},
 Zbl = {1272.14013}
}

@Article{KM:92,
  author     = {J{\'a}nos Koll{\'a}r and Shigefumi Mori},
  title      = {Classification of three-dimensional flips},
  journal    = {J. Amer. Math. Soc.},
  year       = {1992},
  volume     = {5},
  number     = {3},
  pages      = {533--703},
  issn       = {0894-0347},
  doi        = {10.2307/2152704},
  fjournal   = {Journal of the American Mathematical Society},
  language   = {english},
  mrclass    = {14E30 (14B07 14E05 14E35 14J30)},
  mrnumber   = {1149195},
  mrreviewer = {Ulf Persson},
  url        = {https://doi.org/10.2307/2152704},
}

@InCollection{EisenbudHarris:VMD,
  author    = {David Eisenbud and Joe Harris},
  title     = {On varieties of minimal degree. {(A centennial account)}},
  booktitle = {{Algebraic geometry, Bowdoin, 1985 (Brunswick, Maine, 1985), part 1}},
  publisher = {Amer. Math. Soc.},
  year      = {1987},
  volume    = {46},
  series    = {{Proc. Sympos. Pure Math.}},
  pages     = {3--13},
  address   = {Providence, RI},
  language  = {english},
  msc2010   = {14N05 14J99 14-03 01A60},
  zbl       = {0646.14036},
}

@Article{Shin1989,
  Title                    = {{$3$}-dimensional {F}ano varieties with canonical singularities},
  Author                   = {Kil-Ho Shin},
  Journal                  = {Tokyo J. Math.},
  Year                     = {1989},
  Number                   = {2},
  Pages                    = {375--385},
  Volume                   = {12},

  Fjournal                 = {Tokyo Journal of Mathematics},
  ISSN                     = {0387-3870},
  Mrclass                  = {14J30 (14E35)},
  Mrnumber                 = {MR1030501 (90j:14050)},
  Mrreviewer               = {Jaros{\l}aw A. Wi{\'s}niewski}
}

@article {Xiao,
    AUTHOR = {Xiao, Gang},
     TITLE = {Galois covers between {$K3$} surfaces},
   JOURNAL = {Ann. Inst. Fourier (Grenoble)},
  FJOURNAL = {Universit\'e{} de Grenoble. Annales de l'Institut Fourier},
    VOLUME = {46},
      YEAR = {1996},
    NUMBER = {1},
     PAGES = {73--88},
      ISSN = {0373-0956,1777-5310},
   MRCLASS = {14J28 (14J50)},
  MRNUMBER = {1385511},
MRREVIEWER = {Shigeyuki\ Kondo},
       DOI = {10.5802/aif.1507},
       URL = {https://doi.org/10.5802/aif.1507},
}

@article {Kawakita,
    AUTHOR = {Kawakita, Masayuki},
     TITLE = {Inversion of adjunction on log canonicity},
   JOURNAL = {Invent. Math.},
  FJOURNAL = {Inventiones Mathematicae},
    VOLUME = {167},
      YEAR = {2007},
    NUMBER = {1},
     PAGES = {129--133},
      ISSN = {0020-9910,1432-1297},
   MRCLASS = {14E30 (14N30)},
  MRNUMBER = {2264806},
MRREVIEWER = {Carla\ Novelli},
       DOI = {10.1007/s00222-006-0008-z},
       URL = {https://doi.org/10.1007/s00222-006-0008-z},
}

@article{Florin,
 author = {Ambro, F.},
 title = {Ladders on {Fano} varieties},
 fjournal = {Journal of Mathematical Sciences (New York)},
 journal = {J. Math. Sci., New York},
 issn = {1072-3374},
 volume = {94},
 number = {1},
 pages = {1126--1135},
 year = {1999},
 language = {English},
 doi = {10.1007/BF02367253},
 zbMATH = {1315243},
 Zbl = {0948.14033}
}

@article {PriskaIvo,
    AUTHOR = {Jahnke, Priska and Radloff, Ivo},
     TITLE = {Gorenstein {F}ano threefolds with base points in the
              anticanonical system},
   JOURNAL = {Compos. Math.},
  FJOURNAL = {Compositio Mathematica},
    VOLUME = {142},
      YEAR = {2006},
    NUMBER = {2},
     PAGES = {422--432},
      ISSN = {0010-437X,1570-5846},
   MRCLASS = {14J45 (14J30)},
  MRNUMBER = {2218903},
MRREVIEWER = {Adrian\ Langer},
       DOI = {10.1112/S0010437X05001673},
       URL = {https://doi.org/10.1112/S0010437X05001673},
}

@article {HT,
    AUTHOR = {Cheltsov, I. A. and Przhiyalkovskii, V. V. and Shramov, K. A.},
     TITLE = {Hyperelliptic and trigonal {F}ano threefolds},
   JOURNAL = {Izv. Ross. Akad. Nauk Ser. Mat.},
  FJOURNAL = {Izvestiya Rossiiskoi Akademii Nauk. Seriya Matematicheskaya},
    VOLUME = {69},
      YEAR = {2005},
    NUMBER = {2},
     PAGES = {145--204},
      ISSN = {1607-0046,2587-5906},
   MRCLASS = {14J45 (14E15 14J30)},
  MRNUMBER = {2136260},
MRREVIEWER = {Jaros\l aw\ A.\ Wi\'sniewski},
       DOI = {10.1070/IM2005v069n02ABEH000533},
       URL = {https://doi.org/10.1070/IM2005v069n02ABEH000533},
}

@article {Fujino,
    AUTHOR = {Fujino, Osamu},
     TITLE = {Abundance theorem for semi log canonical threefolds},
   JOURNAL = {Duke Math. J.},
  FJOURNAL = {Duke Mathematical Journal},
    VOLUME = {102},
      YEAR = {2000},
    NUMBER = {3},
     PAGES = {513--532},
      ISSN = {0012-7094,1547-7398},
   MRCLASS = {14E30 (14C20 14E07)},
  MRNUMBER = {1756108},
MRREVIEWER = {Tomasz\ Szemberg},
       DOI = {10.1215/S0012-7094-00-10237-2},
       URL = {https://doi.org/10.1215/S0012-7094-00-10237-2},
}

@article{Krylov,
 author = {Krylov, Igor},
 title = {Birational geometry of del {Pezzo} fibrations with terminal quotient singularities},
 fjournal = {Journal of the London Mathematical Society. Second Series},
 journal = {J. Lond. Math. Soc., II. Ser.},
 issn = {0024-6107},
 volume = {97},
 number = {2},
 pages = {222--246},
 year = {2018},
 language = {English},
 doi = {10.1112/jlms.12105},
 zbMATH = {6866356},
 Zbl = {1401.14075}
}

@incollection{BeauvillePGL,
 author = {Beauville, Arnaud},
 title = {Finite subgroups of {{\(\mathrm{PGL}_2(K)\)}}.},
 booktitle = {Vector bundles and complex geometry. Conference on vector bundles in honor of S. Ramanan on the occasion of his 70th birthday, Madrid, Spain, June 16--20, 2008.},
 isbn = {978-0-8218-4750-3},
 pages = {23--29},
 year = {2010},
 publisher = {Providence, RI: American Mathematical Society (AMS)},
 language = {English},
 zbMATH = {5899347},
 Zbl = {1218.20030}
}

@article{Hu,
 author = {Hu, Yijue},
 title = {Jordan constant of {{\(\mathrm{PGL}_3(K)\)}}},
 fjournal = {European Journal of Mathematics},
 journal = {Eur. J. Math.},
 issn = {2199-675X},
 volume = {11},
 number = {1},
 pages = {16},
 note = {Id/No 7},
 year = {2025},
 language = {English},
 doi = {10.1007/s40879-024-00792-8},
 zbMATH = {7981896},
 Zbl = {1557.20049}
}

@book {GilleSzamuely,
    AUTHOR = {Gille, Philippe and Szamuely, Tam\'as},
     TITLE = {Central simple algebras and {G}alois cohomology},
    SERIES = {Cambridge Studies in Advanced Mathematics},
    VOLUME = {165},
   EDITION = {Second},
 PUBLISHER = {Cambridge University Press, Cambridge},
      YEAR = {2017},
     PAGES = {xi+417},
      ISBN = {978-1-316-60988-0; 978-1-107-15637-1},
   MRCLASS = {16K20 (14C35 14F22 19C30)},
  MRNUMBER = {3727161},
}

@misc{KollarSB,
 author = {Koll{\'a}r, J{\'a}nos},
 title = {{S}everi--{B}rauer varieties; a geometric treatment},
 year = {2016},
 howpublished = {Preprint, {arXiv}:1606.04368 [math.{AG}] (2016)},
 url = {https://arxiv.org/abs/1606.04368},
 arXiv = {arXiv:1606.04368}
}

@article{GrossPopescu,
 author = {Gross, Mark and Popescu, Sorin},
 title = {Calabi-Yau three-folds and moduli of abelian surfaces. {II}},
 fjournal = {Transactions of the American Mathematical Society},
 journal = {Trans. Am. Math. Soc.},
 issn = {0002-9947},
 volume = {363},
 number = {7},
 pages = {3573--3599},
 year = {2011},
 language = {English},
 doi = {10.1090/S0002-9947-2011-05179-2},
 zbMATH = {5934524},
 Zbl = {1228.14039}
}

@article{Belousov,
 author = {Belousov, Grigory},
 title = {Log del {Pezzo} surfaces with simple automorphism groups},
 fjournal = {Proceedings of the Edinburgh Mathematical Society. Series II},
 journal = {Proc. Edinb. Math. Soc., II. Ser.},
 issn = {0013-0915},
 volume = {58},
 number = {1},
 pages = {33--52},
 year = {2015},
 language = {English},
 doi = {10.1017/S0013091514000054},
 zbMATH = {6398976},
 Zbl = {1323.14025}
}

@article {Shokurov92,
    AUTHOR = {Shokurov, V. V.},
     TITLE = {Three-dimensional log perestroikas},
      NOTE = {With an appendix in English by Yujiro Kawamata},
   JOURNAL = {Izv. Ross. Akad. Nauk Ser. Mat.},
  FJOURNAL = {Izvestiya Rossiiskoi Akademii Nauk. Seriya Matematicheskaya},
    VOLUME = {56},
      YEAR = {1992},
    NUMBER = {1},
     PAGES = {105--203},
      ISSN = {1607-0046,2587-5906},
   MRCLASS = {14E05 (14E35)},
  MRNUMBER = {1162635},
MRREVIEWER = {A.\ S.\ Tikhomirov},
       DOI = {10.1070/IM1993v040n01ABEH001862},
       URL = {https://doi.org/10.1070/IM1993v040n01ABEH001862},
}

@article{BlancLamy,
 author = {Blanc, J{\'e}r{\'e}my and Lamy, St{\'e}phane},
 title = {Weak {Fano} threefolds obtained by blowing-up a space curve and construction of {Sarkisov} links},
 fjournal = {Proceedings of the London Mathematical Society. Third Series},
 journal = {Proc. Lond. Math. Soc. (3)},
 issn = {0024-6115},
 volume = {105},
 number = {5},
 pages = {1047--1075},
 year = {2012},
 language = {English},
 doi = {10.1112/plms/pds023},
 zbMATH = {6111539},
 Zbl = {1258.14015}
}

@article{DebarreKuznetsov,
 author = {Debarre, Olivier and Kuznetsov, Alexander},
 title = {Gushel-Mukai varieties: classification and birationalities},
 fjournal = {Algebraic Geometry},
 journal = {Algebr. Geom.},
 issn = {2313-1691},
 volume = {5},
 number = {1},
 pages = {15--76},
 year = {2018},
 language = {English},
 doi = {10.14231/AG-2018-002},
 zbMATH = {6999231},
 Zbl = {1408.14053}
}

@incollection {Jen,
    AUTHOR = {Paulhus, Jennifer},
     TITLE = {A database of group actions on {R}iemann surfaces},
 BOOKTITLE = {Birational geometry, {K}\"ahler-{E}instein metrics and
              degenerations},
    SERIES = {Springer Proc. Math. Stat.},
    VOLUME = {409},
     PAGES = {693--708},
 PUBLISHER = {Springer},
      YEAR = {2023},
      ISBN = {978-3-031-17858-0; 978-3-031-17859-7},
   MRCLASS = {30F35 (14H37 20H10 30-04)},
  MRNUMBER = {4606662},
       DOI = {10.1007/978-3-031-17859-7\_35},
       URL = {https://doi.org/10.1007/978-3-031-17859-7_35},
}

@misc{ArendSashaEmanuele,
 author = {Bayer, Arend and Kuznetsov, Alexander and Macr{\`{\i}}, Emanuele},
 title = {Mukai models of {Fano} varieties},
 year = {2025},
 howpublished = {Preprint, {arXiv}:2501.16157 [math.{AG}] (2025)},
 url = {https://arxiv.org/abs/2501.16157},
 arXiv = {arXiv:2501.16157}
}

@article {RavindraSrinivas,
    AUTHOR = {Ravindra, G. V. and Srinivas, V.},
     TITLE = {The {N}oether-{L}efschetz theorem for the divisor class group},
   JOURNAL = {J. Algebra},
  FJOURNAL = {Journal of Algebra},
    VOLUME = {322},
      YEAR = {2009},
    NUMBER = {9},
     PAGES = {3373--3391},
      ISSN = {0021-8693,1090-266X},
   MRCLASS = {13C20 (14C20)},
  MRNUMBER = {2567426},
MRREVIEWER = {Abdeslam\ Mimouni},
       DOI = {10.1016/j.jalgebra.2008.09.003},
       URL = {https://doi.org/10.1016/j.jalgebra.2008.09.003},
}

@article {Umezu,
    AUTHOR = {Umezu, Yumiko},
     TITLE = {On normal projective surfaces with trivial dualizing sheaf},
   JOURNAL = {Tokyo J. Math.},
  FJOURNAL = {Tokyo Journal of Mathematics},
    VOLUME = {4},
      YEAR = {1981},
    NUMBER = {2},
     PAGES = {343--354},
      ISSN = {0387-3870},
   MRCLASS = {14J15},
  MRNUMBER = {646044},
MRREVIEWER = {Jonathan\ M.\ Wahl},
       DOI = {10.3836/tjm/1270215159},
       URL = {https://doi.org/10.3836/tjm/1270215159},
}

@incollection {Prokhorov-2-groups,
    AUTHOR = {Prokhorov, Yuri},
     TITLE = {2-elementary subgroups of the space {C}remona group},
 BOOKTITLE = {Automorphisms in birational and affine geometry},
    SERIES = {Springer Proc. Math. Stat.},
    VOLUME = {79},
     PAGES = {215--229},
 PUBLISHER = {Springer, Cham},
      YEAR = {2014},
      ISBN = {978-3-319-05681-4; 978-3-319-05680-7},
   MRCLASS = {14E07 (14M22)},
  MRNUMBER = {3229353},
MRREVIEWER = {J\'er\'emy\ Blanc},
       DOI = {10.1007/978-3-319-05681-4\_12},
       URL = {https://doi.org/10.1007/978-3-319-05681-4_12},
}

@article{Prokhorov2007,
 author = {Prokhorov, Yu. G.},
 title = {The degree of $\mathbb{Q}$-{Fano} threefolds},
 fjournal = {Sbornik: Mathematics},
 journal = {Sb. Math.},
 issn = {1064-5616},
 volume = {198},
 number = {11},
 pages = {1683--1702},
 year = {2007},
 language = {English},
 doi = {10.1070/SM2007v198n11ABEH003901},
 zbMATH = {5272618},
 Zbl = {1139.14033}
}

@incollection {topologists,
    AUTHOR = {Reni, Marco and Zimmermann, Bruno},
     TITLE = {Finite simple groups acting on 3-manifolds and homology
              spheres},
      NOTE = {Dedicated to the memory of Marco Reni},
   JOURNAL = {Rend. Istit. Mat. Univ. Trieste},
  FJOURNAL = {Rendiconti dell'Istituto di Matematica dell'Universit\`a{} di
              Trieste. An International Journal of Mathematics},
    VOLUME = {32},
      YEAR = {2001},
     PAGES = {305--315},
      ISSN = {0049-4704,2464-8728},
   MRCLASS = {57M60 (57S17)},
  MRNUMBER = {1893403},
MRREVIEWER = {Michael\ Heusener},
}

@article {Egor-simple-groups,
    AUTHOR = {Yasinsky, Egor},
     TITLE = {Automorphisms of real del {P}ezzo surfaces and the real plane
              {C}remona group},
   JOURNAL = {Ann. Inst. Fourier (Grenoble)},
  FJOURNAL = {Universit\'e{} de Grenoble. Annales de l'Institut Fourier},
    VOLUME = {72},
      YEAR = {2022},
    NUMBER = {2},
     PAGES = {831--899},
      ISSN = {0373-0956,1777-5310},
   MRCLASS = {14E07 (14E05 14J26 14J50 14P05 20D99)},
  MRNUMBER = {4448610},
MRREVIEWER = {Tatiana\ M.\ Bandman},
       DOI = {10.5802/aif.3460},
       URL = {https://doi.org/10.5802/aif.3460},
}

@incollection {Prokhorov-p-groups,
    AUTHOR = {Prokhorov, Yuri},
     TITLE = {{$p$}-elementary subgroups of the {C}remona group of rank 3},
 BOOKTITLE = {Classification of algebraic varieties},
    SERIES = {EMS Ser. Congr. Rep.},
     PAGES = {327--338},
 PUBLISHER = {Eur. Math. Soc., Z\"urich},
      YEAR = {2011},
      ISBN = {978-3-03719-007-4},
   MRCLASS = {14E07 (14H37 14J45)},
  MRNUMBER = {2779480},
MRREVIEWER = {I.\ Dolgachev},
       DOI = {10.4171/007-1/16},
       URL = {https://doi.org/10.4171/007-1/16},
}

@article {YuraCostya-p-groups,
    AUTHOR = {Prokhorov, Yuri and Shramov, Constantin},
     TITLE = {{$p$}-subgroups in the space {C}remona group},
   JOURNAL = {Math. Nachr.},
  FJOURNAL = {Mathematische Nachrichten},
    VOLUME = {291},
      YEAR = {2018},
    NUMBER = {8-9},
     PAGES = {1374--1389},
      ISSN = {0025-584X,1522-2616},
   MRCLASS = {14E07},
  MRNUMBER = {3817323},
MRREVIEWER = {Alex\ Massarenti},
       DOI = {10.1002/mana.201700030},
       URL = {https://doi.org/10.1002/mana.201700030},
}

@article{Bayle,
 author = {Bayle, Lionel},
 title = {Classification of the projective complex threefolds whose general hyperplane section is an {Enrique} surface},
 fjournal = {Journal f{\"u}r die Reine und Angewandte Mathematik},
 journal = {J. Reine Angew. Math.},
 issn = {0075-4102},
 volume = {449},
 pages = {9--63},
 year = {1994},
 language = {French},
 doi = {10.1515/crll.1994.449.9},
 url = {https://eudml.org/doc/153604},
 zbMATH = {530111},
 Zbl = {0808.14028}
}

@article{SanoEnriques,
 author = {Sano, Takeshi},
 title = {On classification of non-{Gorenstein} $\mathbb{Q}$-{Fano} 3-folds of {Fano} index 1},
 fjournal = {Journal of the Mathematical Society of Japan},
 journal = {J. Math. Soc. Japan},
 issn = {0025-5645},
 volume = {47},
 number = {2},
 pages = {369--380},
 year = {1995},
 language = {English},
 doi = {10.2969/jmsj/04720369},
 zbMATH = {844458},
 Zbl = {0837.14031}
}

@article {Sano,
    AUTHOR = {Sano, Takeshi},
     TITLE = {Classification of non-{G}orenstein {${\bf Q}$}-{F}ano
              {$d$}-folds of {F}ano index greater than {$d-2$}},
   JOURNAL = {Nagoya Math. J.},
  FJOURNAL = {Nagoya Mathematical Journal},
    VOLUME = {142},
      YEAR = {1996},
     PAGES = {133--143},
      ISSN = {0027-7630,2152-6842},
   MRCLASS = {14J45},
  MRNUMBER = {1399470},
MRREVIEWER = {Yuri\ G.\ Prokhorov},
       DOI = {10.1017/S0027763000005663},
       URL = {https://doi.org/10.1017/S0027763000005663},
}

@Article{Takeuchi:DP,
  author   = {Takeuchi, Kiyohiko},
  title    = {Weak {Fano} threefolds with del {Pezzo} fibration},
  journal  = {Eur. J. Math.},
  year     = {2022},
  volume   = {8},
  number   = {3},
  pages    = {1225--1290},
  issn     = {2199-675X},
  doi      = {10.1007/s40879-022-00571-3},
  fjournal = {European Journal of Mathematics},
  language = {English},
  zbl      = {1505.14032},
  zbmath   = {7604912},
}

@Article{Fukuoka:DP6,
  author   = {Fukuoka, Takeru},
  journal  = {Math. Nachr.},
  title    = {On the existence of almost {Fano} threefolds with del {Pezzo} fibrations},
  year     = {2017},
  issn     = {0025-584X},
  number   = {8-9},
  pages    = {1281--1302},
  volume   = {290},
  doi      = {10.1002/mana.201600207},
  fjournal = {Mathematische Nachrichten},
  language = {English},
  zbl      = {1370.14034},
  zbmath   = {6745510},
}

@Article{Fukuoka:DP:Ref,
  author       = {Fukuoka, Takeru},
  title        = {Refinement of the classification of weak {Fano} threefolds with sextic del {Pezzo} fibrations},
  year         = {2019},
  arxiv        = {arXiv:1903.06872},
  howpublished = {Preprint, {arXiv}:1903.06872 [math.{AG}] (2019)},
  url          = {https://arxiv.org/abs/1903.06872},
}

@article {Leningrad,
    AUTHOR = {Degtyarev, Alex and Itenberg, Ilia and Kharlamov, Viatcheslav},
     TITLE = {Finiteness and quasi-simplicity for symmetric {$K3$}-surfaces},
   JOURNAL = {Duke Math. J.},
  FJOURNAL = {Duke Mathematical Journal},
    VOLUME = {122},
      YEAR = {2004},
    NUMBER = {1},
     PAGES = {1--49},
      ISSN = {0012-7094,1547-7398},
   MRCLASS = {14J28 (14J50 14L30 14P25 32J15)},
  MRNUMBER = {2046806},
MRREVIEWER = {I.\ Dolgachev},
       DOI = {10.1215/S0012-7094-04-12211-8},
       URL = {https://doi.org/10.1215/S0012-7094-04-12211-8},
}

@article {CampanaFlenner,
    AUTHOR = {Campana, F. and Flenner, H.},
     TITLE = {Projective threefolds containing a smooth rational surface
              with ample normal bundle},
   JOURNAL = {J. Reine Angew. Math.},
  FJOURNAL = {Journal f\"ur die Reine und Angewandte Mathematik. [Crelle's
              Journal]},
    VOLUME = {440},
      YEAR = {1993},
     PAGES = {77--98},
      ISSN = {0075-4102,1435-5345},
   MRCLASS = {14J30 (14J45)},
  MRNUMBER = {1225958},
MRREVIEWER = {Valery\ Alexeev},
       DOI = {10.1515/crll.1993.440.77},
       URL = {https://doi.org/10.1515/crll.1993.440.77},
}

@article {Mori-singularities,
    AUTHOR = {Mori, Shigefumi},
     TITLE = {On {$3$}-dimensional terminal singularities},
   JOURNAL = {Nagoya Math. J.},
  FJOURNAL = {Nagoya Mathematical Journal},
    VOLUME = {98},
      YEAR = {1985},
     PAGES = {43--66},
      ISSN = {0027-7630,2152-6842},
   MRCLASS = {14B05 (14B07 32B30)},
  MRNUMBER = {792770},
MRREVIEWER = {David\ R.\ Morrison},
       DOI = {10.1017/S0027763000021358},
       URL = {https://doi.org/10.1017/S0027763000021358},
}

@article{CheltsovSarikyan,
 author = {Cheltsov, Ivan and Sarikyan, Arman},
 title = {Equivariant pliability of the projective space},
 fjournal = {Selecta Mathematica. New Series},
 journal = {Sel. Math., New Ser.},
 issn = {1022-1824},
 volume = {29},
 number = {5},
 pages = {84},
 note = {Id/No 71},
 year = {2023},
 language = {English},
 doi = {10.1007/s00029-023-00869-4},
 zbMATH = {7746700},
 Zbl = {1531.14019}
}

@incollection{CheltsovSarikyanZhuang,
 author = {Cheltsov, Ivan and Sarikyan, Arman and Zhuang, Ziquan},
 title = {Birational rigidity and alpha invariants of {Fano} varieties},
 booktitle = {Higher dimensional algebraic geometry. A volume in honor of V. V. Shokurov to his 70th birthday. Based on the Japan-US Mathematics Institute (JAMI) conference, Baltimore, MD, USA May 3--8, 2022},
 isbn = {978-1-00-939624-0; 978-1-00-939623-3},
 pages = {286--318},
 year = {2025},
 publisher = {Cambridge: Cambridge University Press},
 language = {English},
 doi = {10.1017/9781009396233.018},
 zbMATH = {8085510}
}

@Article{Kawamata-1992-e-app,
  author   = {Yujiro Kawamata},
  title    = {The minimal discrepancy coefficients of terminal singularities in dimension three. {Appendix to V.~V.~Shokurov's paper ``$3$-fold log flips''}},
  journal  = {Russ. Acad. Sci., Izv., Math.},
  year     = {1993},
  volume   = {40},
  number   = {1},
  pages    = {193--195},
  language = {english},
}

@article{PiontkowskiVandeVen,
 author = {Piontkowski, J. and Van de Ven, A.},
 title = {The automorphism group of linear sections of the {Grassmannians} $\mathbb{G}(1,{N})$},
 fjournal = {Documenta Mathematica},
 journal = {Doc. Math.},
 issn = {1431-0635},
 volume = {4},
 pages = {623--664},
 year = {1999},
 language = {English},
 doi = {10.4171/dm/70},
 url = {https://eudml.org/doc/233797},
 zbMATH = {1377563},
 Zbl = {0934.14032}
}

@article {SmoothingPic,
    AUTHOR = {Jahnke, Priska and Radloff, Ivo},
     TITLE = {Terminal {F}ano threefolds and their smoothings},
   JOURNAL = {Math. Z.},
  FJOURNAL = {Mathematische Zeitschrift},
    VOLUME = {269},
      YEAR = {2011},
    NUMBER = {3-4},
     PAGES = {1129--1136},
      ISSN = {0025-5874,1432-1823},
   MRCLASS = {14J45 (14J30)},
  MRNUMBER = {2860279},
MRREVIEWER = {Jaros\l aw\ A.\ Wi\'sniewski},
       DOI = {10.1007/s00209-010-0780-8},
       URL = {https://doi.org/10.1007/s00209-010-0780-8},
}

@article {Namikawa,
    AUTHOR = {Namikawa, Yoshinori},
     TITLE = {Smoothing {F}ano {$3$}-folds},
   JOURNAL = {J. Algebraic Geom.},
  FJOURNAL = {Journal of Algebraic Geometry},
    VOLUME = {6},
      YEAR = {1997},
    NUMBER = {2},
     PAGES = {307--324},
      ISSN = {1056-3911,1534-7486},
   MRCLASS = {14J45 (14B07 14B12 14D15)},
  MRNUMBER = {1489117},
MRREVIEWER = {Yuri\ G.\ Prokhorov},
}

@article {Prokhorov-72,
    AUTHOR = {Prokhorov, Yu. G.},
     TITLE = {The degree of {F}ano threefolds with canonical {G}orenstein
              singularities},
   JOURNAL = {Mat. Sb.},
  FJOURNAL = {Matematicheski\u i\ Sbornik},
    VOLUME = {196},
      YEAR = {2005},
    NUMBER = {1},
     PAGES = {81--122},
      ISSN = {0368-8666,2305-2783},
   MRCLASS = {14J45 (14E15)},
  MRNUMBER = {2141325},
MRREVIEWER = {Jaros\l aw\ A.\ Wi\'sniewski},
       DOI = {10.1070/SM2005v196n01ABEH000873},
       URL = {https://doi.org/10.1070/SM2005v196n01ABEH000873},
}

@article {Kawamata-1,
    AUTHOR = {Kawamata, Yujiro},
     TITLE = {On {F}ujita's freeness conjecture for {$3$}-folds and
              {$4$}-folds},
   JOURNAL = {Math. Ann.},
  FJOURNAL = {Mathematische Annalen},
    VOLUME = {308},
      YEAR = {1997},
    NUMBER = {3},
     PAGES = {491--505},
      ISSN = {0025-5831,1432-1807},
   MRCLASS = {14C20 (14J30 14J35)},
  MRNUMBER = {1457742},
       DOI = {10.1007/s002080050085},
       URL = {https://doi.org/10.1007/s002080050085},
}

@article {Kawamata-2,
    AUTHOR = {Kawamata, Yujiro},
     TITLE = {Subadjunction of log canonical divisors. {II}},
   JOURNAL = {Amer. J. Math.},
  FJOURNAL = {American Journal of Mathematics},
    VOLUME = {120},
      YEAR = {1998},
    NUMBER = {5},
     PAGES = {893--899},
      ISSN = {0002-9327,1080-6377},
   MRCLASS = {14E30 (14J10)},
  MRNUMBER = {1646046},
MRREVIEWER = {Alessio\ Corti},
       URL =
              {http://muse.jhu.edu/journals/american_journal_of_mathematics/v120/120.5kawamata.pdf},
}

@article {KawachiMasec,
    AUTHOR = {Kawachi, Takeshi and Ma\c{s}ek, Vladimir},
     TITLE = {Reider-type theorems on normal surfaces},
   JOURNAL = {J. Algebraic Geom.},
  FJOURNAL = {Journal of Algebraic Geometry},
    VOLUME = {7},
      YEAR = {1998},
    NUMBER = {2},
     PAGES = {239--249},
      ISSN = {1056-3911,1534-7486},
   MRCLASS = {14J17 (14C20 14E05)},
  MRNUMBER = {1620102},
MRREVIEWER = {Marco\ Andreatta},
}

@book {Icosahedron,
    AUTHOR = {Cheltsov, Ivan and Shramov, Constantin},
     TITLE = {Cremona groups and the icosahedron},
    SERIES = {Monographs and Research Notes in Mathematics},
 PUBLISHER = {CRC Press, Boca Raton, FL},
      YEAR = {2016},
     PAGES = {xxi+504},
      ISBN = {978-1-4822-5159-3},
   MRCLASS = {14E07 (14J45 14J50 20B25)},
  MRNUMBER = {3444095},
MRREVIEWER = {I.\ Dolgachev},
}

@incollection{Corti,
 author = {Corti, Alessio},
 title = {Singularities of linear systems and 3-fold birational geometry},
 booktitle = {Explicit birational geometry of 3-folds},
 isbn = {0-521-63641-8},
 pages = {259--312},
 year = {2000},
 publisher = {Cambridge: Cambridge University Press},
 language = {English},
 zbMATH = {1513902},
 Zbl = {0960.14017}
}

@article{CheltsovUMN,
 author = {Cheltsov, I. A.},
 title = {Birationally rigid {Fano} varieties},
 fjournal = {Russian Mathematical Surveys},
 journal = {Russ. Math. Surv.},
 issn = {0036-0279},
 volume = {60},
 number = {5},
 pages = {875--965},
 year = {2005},
 language = {English},
 doi = {10.1070/RM2005v060n05ABEH003736},
 zbMATH = {5268230},
 Zbl = {1145.14032}
}

@Book{Utah,
  Title                    = {Flips and abundance for algebraic threefolds},
  Editor                   = {J{\'a}nos Koll{\'a}r},
  Publisher                = {Soci{\'e}t{\'e} Math{\'e}matique de France},
  Year                     = {1992},

  Address                  = {Paris},
  Note                     = {Papers from the Second Summer Seminar on Algebraic Geometry held at the University of Utah, Salt Lake City, Utah, August 1991, Ast{\'e}risque No. 211 (1992)},

  ISSN                     = {0303-1179},
  Language                 = {english},
  Mrclass                  = {14E30 (14E35 14M10)},
  Mrnumber                 = {MR1225842 (94f:14013)},
  Mrreviewer               = {Mark Gross},
  Pages                    = {1--258}
}

@article{Barth,
 author = {Barth, W.},
 title = {Two projective surfaces with many nodes, admitting the symmetries of the icosahedron},
 fjournal = {Journal of Algebraic Geometry},
 journal = {J. Algebr. Geom.},
 issn = {1056-3911},
 volume = {5},
 number = {1},
 pages = {173--186},
 year = {1996},
 language = {English},
 zbMATH = {859841},
 Zbl = {0860.14032}
}

@article {CheltsovShramovSelecta,
    AUTHOR = {Cheltsov, Ivan and Shramov, Constantin},
     TITLE = {Finite collineation groups and birational rigidity},
   JOURNAL = {Selecta Math. (N.S.)},
  FJOURNAL = {Selecta Mathematica. New Series},
    VOLUME = {25},
      YEAR = {2019},
    NUMBER = {5},
     PAGES = {Paper No. 71, 68},
      ISSN = {1022-1824,1420-9020},
   MRCLASS = {14E07 (14J50)},
  MRNUMBER = {4036497},
MRREVIEWER = {Jin-Xing\ Cai},
       DOI = {10.1007/s00029-019-0516-5},
       URL = {https://doi.org/10.1007/s00029-019-0516-5},
}

@Article{C58,
  Title                    = {Questions de rationalit\'{e} des diviseurs en g\'{e}om\'{e}trie
 alg\'{e}brique},
  Author                   = {Cartier, Pierre},
  Journal                  = {Bull. Soc. Math. France},
  Year                     = {1958},
  Pages                    = {177--251},
  Volume                   = {86},

  Fjournal                 = {Bulletin de la Soci\'{e}t\'{e} Math\'{e}matique de France},
  ISSN                     = {0037-9484},
  Mrclass                  = {14.00},
  Mrnumber                 = {0106223},
  Mrreviewer               = {M. Nagata},
  Url                      = {http://www.numdam.org/item?id=BSMF_1958__86__177_0}
}

@article{VanyaYuraZhijia,
 author = {Cheltsov, Ivan and Tschinkel, Yuri and Zhang, Zhijia},
 title = {Equivariant geometry of the {Segre} cubic and the {Burkhardt} quartic},
 fjournal = {Selecta Mathematica. New Series},
 journal = {Sel. Math., New Ser.},
 issn = {1022-1824},
 volume = {31},
 number = {1},
 pages = {36},
 note = {Id/No 7},
 year = {2025},
 language = {English},
 doi = {10.1007/s00029-024-01001-w},
 zbMATH = {7962041}
}

@article{KollarSzabo,
 author = {Reichstein, Zinovy and Youssin, Boris},
 title = {Essential dimensions of algebraic groups and a resolution theorem for {{\(G\)}}-varieties. ({With} an appendix by {J{\'a}nos} {Koll{\'a}r} and {Endre} {Szab{\'o}}: {Fixed} points of group actions and rational maps).},
 fjournal = {Canadian Journal of Mathematics},
 journal = {Can. J. Math.},
 issn = {0008-414X},
 volume = {52},
 number = {5},
 pages = {1018--1056},
 year = {2000},
 language = {English},
 doi = {10.4153/CJM-2000-043-5},
 zbMATH = {1545722},
 Zbl = {1044.14023}
}

@article{Avilov,
 author = {Avilov, Artem},
 title = {Automorphisms of singular three-dimensional cubic hypersurfaces},
 fjournal = {European Journal of Mathematics},
 journal = {Eur. J. Math.},
 issn = {2199-675X},
 volume = {4},
 number = {3},
 pages = {761--777},
 year = {2018},
 language = {English},
 doi = {10.1007/s40879-018-0253-x},
 zbMATH = {7073886},
 Zbl = {1423.14096}
}

@book{Blanc1,
 author = {Blanc, J{\'e}r{\'e}my},
 title = {Finite abelian subgroups of the {Cremona} group of the plane},
 year = {2006},
 publisher = {Gen{\`e}ve: Univ. de Gen{\`e}ve, Facult{\'e} des Sciences (Dissertation)},
 language = {English},
 zbMATH = {5315128},
 Zbl = {1141.14306}
}

@article{Blanc2,
 author = {Blanc, J{\'e}r{\'e}my},
 title = {Finite abelian subgroups of the {Cremona} group of the plane},
 fjournal = {Comptes Rendus. Math{\'e}matique. Acad{\'e}mie des Sciences, Paris},
 journal = {C. R., Math., Acad. Sci. Paris},
 issn = {1631-073X},
 volume = {344},
 number = {1},
 pages = {21--26},
 year = {2007},
 language = {English},
 doi = {10.1016/j.crma.2006.11.026},
 zbMATH = {5126418},
 Zbl = {1111.14003}
}

@article{IwaiJiangLiu,
 author = {Iwai, Masataka and Jiang, Chen and Liu, Haidong},
 title = {Miyaoka-type inequalities for terminal threefolds with nef anti-canonical divisors},
 fjournal = {Science China. Mathematics},
 journal = {Sci. China, Math.},
 issn = {1674-7283},
 volume = {68},
 number = {1},
 pages = {1--18},
 year = {2025},
 language = {English},
 doi = {10.1007/s11425-023-2230-6},
 zbMATH = {7965979}
}

@article{LiuLiu1,
 author = {Liu, Haidong and Liu, Jie},
 title = {{K}awamata-{M}iyaoka type inequality for {{\(\mathbb{Q} \)}}-{Fano} varieties with canonical singularities},
 fjournal = {Journal f{\"u}r die Reine und Angewandte Mathematik},
 journal = {J. Reine Angew. Math.},
 issn = {0075-4102},
 volume = {819},
 pages = {265--281},
 year = {2025},
 language = {English},
 doi = {10.1515/crelle-2024-0087},
 zbMATH = {7982965},
 Zbl = {1561.14067}
}

@article{LiuLiu2,
 author = {Liu, Haidong and Liu, Jie},
 title = {{K}awamata-{M}iyaoka-type inequality for {{\(\mathbb{Q}\)}}-{Fano} varieties with canonical singularities. {II}: {Terminal} {{\(\mathbb{Q}\)}}-{Fano} threefolds},
 fjournal = {{\'E}pijournal de G{\'e}om{\'e}trie Alg{\'e}brique. EPIGA},
 journal = {{\'E}pijournal de G{\'e}om. Alg{\'e}br., EPIGA},
 issn = {2491-6765},
 volume = {9},
 pages = {21},
 note = {Id/No 12},
 year = {2025},
 language = {English},
 doi = {10.46298/epiga.2025.13167},
 zbMATH = {8068448}
}

@misc{PinardingZhang,
 author = {Pinardin, Antoine and Zhang, Zhijia},
 title = {$\mathfrak{A}_5$-equivariant geometry of quadric threefolds},
 year = {2025},
 howpublished = {Preprint, {arXiv}:2508.11496 [math.{AG}] (2025)},
 url = {https://arxiv.org/abs/2508.11496},
 arXiv = {arXiv:2508.11496}
}

@misc{RonanSusanna,
 author = {Terpereau, Ronan and Zimmermann, Susanna},
 title = {Real forms of {Mori} fiber spaces with many symmetries},
 year = {2024},
 howpublished = {Preprint, {arXiv}:2403.14493 [math.{AG}] (2024)},
 url = {https://arxiv.org/abs/2403.14493},
 arXiv = {arXiv:2403.14493}
}

@misc{2-12,
 author = {Cheltsov, Ivan and Li, Oliver and Ma'u, Sione and Pinardin, Antoine},
 title = {K-stability and space sextic curves of genus three},
 year = {2024},
 howpublished = {Preprint, {arXiv}:2404.07803 [math.{AG}] (2024)},
 url = {https://arxiv.org/abs/2404.07803},
 arXiv = {arXiv:2404.07803}
}

@article{cheltsov2014five,
  title={Five embeddings of one simple group},
  author={Cheltsov, Ivan and Shramov, Constantin},
  journal={Transactions of the American Mathematical Society},
  volume={366},
  number={3},
  pages={1289--1331},
  year={2014}
}

@article {Prokhorov-g-12,
    AUTHOR = {Prokhorov, Yu. G.},
     TITLE = {Singular {F}ano manifolds of genus 12},
   JOURNAL = {Mat. Sb.},
  FJOURNAL = {Matematicheski\u i\ Sbornik},
    VOLUME = {207},
      YEAR = {2016},
    NUMBER = {7},
     PAGES = {101--130},
      ISSN = {0368-8666,2305-2783},
   MRCLASS = {14J45 (14E07 14E30)},
  MRNUMBER = {3535377},
MRREVIEWER = {Jaros\l aw\ A.\ Wi\'sniewski},
       DOI = {10.4213/sm8585},
       URL = {https://doi.org/10.4213/sm8585},
}

@article {SashaYuraCostya,
    AUTHOR = {Kuznetsov, Alexander G. and Prokhorov, Yuri G. and Shramov,
              Constantin A.},
     TITLE = {Hilbert schemes of lines and conics and automorphism groups of
              {F}ano threefolds},
   JOURNAL = {Jpn. J. Math.},
  FJOURNAL = {Japanese Journal of Mathematics},
    VOLUME = {13},
      YEAR = {2018},
    NUMBER = {1},
     PAGES = {109--185},
      ISSN = {0289-2316,1861-3624},
   MRCLASS = {14J45 (14C05 14J30 14J50)},
  MRNUMBER = {3776469},
MRREVIEWER = {Alexandr\ V.\ Pukhlikov},
       DOI = {10.1007/s11537-017-1714-6},
       URL = {https://doi.org/10.1007/s11537-017-1714-6},
}

@book{CalabiBook,
  title={The Calabi problem for Fano threefolds},
  author={Araujo, Carolina and Castravet, Ana-Maria and Cheltsov, Ivan and Fujita, Kento and Kaloghiros, Anne-Sophie and Martinez-Garcia, Jesus and Shramov, Constantin and S{\"u}{\ss}, Hendrik and Viswanathan, Nivedita},
  volume={485},
  year={2023},
  publisher={Cambridge University Press}
}

@article{Tian,
  title={On {K}{\"a}hler-{E}instein metrics on certain {K}{\"a}hler manifolds with $c_1(\uppercase{M})>0$},
  author={G. Tian},
  journal={Inventiones mathematicae},
  volume={89},
  number={2},
  pages={225--246},
  year={1987},
  publisher={Springer-Verlag Berlin/Heidelberg}
}

@article {cheltsov2011exceptional,
    AUTHOR = {Cheltsov, I. and Shramov, C.},
     TITLE = {On exceptional quotient singularities},
   JOURNAL = {Geom. Topol.},
  FJOURNAL = {Geometry \& Topology},
    VOLUME = {15},
      YEAR = {2011},
    NUMBER = {4},
     PAGES = {1843--1882},
      ISSN = {1465-3060},
   MRCLASS = {53C55 (32Q25 32S99)},
  MRNUMBER = {2860982},
MRREVIEWER = {Yanir A. Rubinstein},
       DOI = {10.2140/gt.2011.15.1843},
       URL = {https://doi.org/10.2140/gt.2011.15.1843},
}

@article{beauville2014finite,
  title={Finite simple groups of small essential dimension},
  author={Beauville, Arnaud},
  journal={Trends in contemporary mathematics},
  pages={221--228},
  year={2014},
  publisher={Springer}
}

@article{Joe2024,
 author = {Malbon, Joseph},
 title = {Automorphisms of {Fano} threefolds of rank 2 and degree 28},
 fjournal = {Annali dell'Universit{\`a} di Ferrara. Sezione VII. Scienze Matematiche},
 journal = {Ann. Univ. Ferrara, Sez. VII, Sci. Mat.},
 issn = {0430-3202},
 volume = {70},
 number = {3},
 pages = {1083--1092},
 year = {2024},
 language = {English},
 doi = {10.1007/s11565-024-00525-5},
 zbMATH = {7901298},
 Zbl = {1556.14050}
}

@article {Cutkosky,
    AUTHOR = {Cutkosky, Steven},
     TITLE = {Elementary contractions of {G}orenstein threefolds},
   JOURNAL = {Math. Ann.},
  FJOURNAL = {Mathematische Annalen},
    VOLUME = {280},
      YEAR = {1988},
    NUMBER = {3},
     PAGES = {521--525},
      ISSN = {0025-5831,1432-1807},
   MRCLASS = {14J30 (13F15 14E35)},
  MRNUMBER = {936328},
MRREVIEWER = {Harry\ D'Souza},
       DOI = {10.1007/BF01456342},
       URL = {https://doi.org/10.1007/BF01456342},
}

@incollection{BeauvilleA7,
 author = {Beauville, Arnaud},
 title = {Non-rationality of the symmetric sextic {Fano} threefold},
 booktitle = {Geometry and arithmetic. Based on the conference, Island of Schiermonnikoog, Netherlands, September 2010},
 isbn = {978-3-03719-119-4},
 pages = {57--60},
 year = {2012},
 publisher = {Z{\"u}rich: European Mathematical Society (EMS)},
 language = {English},
 doi = {10.4171/119-1/3},
 zbMATH = {6464175},
 Zbl = {1317.14033}
}

@article{hidalgo2018fricke,
  title={About the {F}ricke--{M}acbeath curve},
  author={Hidalgo, Ruben A},
  journal={European Journal of Mathematics},
  volume={4},
  number={1},
  pages={313--325},
  year={2018},
  publisher={Springer}
}

@book {lazarsfeld2003positivity,
    AUTHOR = {Lazarsfeld, R.},
     TITLE = {Positivity in algebraic geometry. {II}},
    SERIES = {Ergebnisse der Mathematik und ihrer Grenzgebiete. 3. Folge. A
              Series of Modern Surveys in Mathematics [Results in
              Mathematics and Related Areas. 3rd Series. A Series of Modern
              Surveys in Mathematics]},
    VOLUME = {49},
      NOTE = {Positivity for vector bundles, and multiplier ideals},
 PUBLISHER = {Springer-Verlag, Berlin},
      YEAR = {2004},
     PAGES = {xviii+385},
      ISBN = {3-540-22534-X},
   MRCLASS = {14-02 (14C20 14F05 14F17)},
  MRNUMBER = {2095472},
MRREVIEWER = {Mihnea\ Popa},
       DOI = {10.1007/978-3-642-18808-4},
       URL = {https://doi.org/10.1007/978-3-642-18808-4},
}

@misc{lmfdb,
  shorthand    = {LMFDB},
  author       = {The {LMFDB Collaboration}},
  title        = {The {L}-functions and modular forms database},
  howpublished = {\url{https://www.lmfdb.org}},
  year         = {2025},
}

@misc{atlas,
  author       = {Robert Wilson and Peter Walsh and Jonathan Tripp and Ibrahim Suleiman and Richard Parker and Simon Norton and Simon Nickerson and Steve Linton and John Bray and Rachel Abbott},
  title        = {ATLAS of Finite Group Representations, Version 3},
  howpublished = {\url{https://brauer.maths.qmul.ac.uk/Atlas/v3/}},
  year         = {2025},
  note         = {Accessed: 2025-09-20}
}

@Book{Dolgachev-ClassicalAlgGeom,
  title     = {Classical algebraic geometry},
  publisher = {Cambridge University Press},
  year      = {2012},
  author    = {Igor V. Dolgachev},
  address   = {Cambridge},
  isbn      = {978-1-107-01765-8},
  doi       = {10.1017/CBO9781139084437},
  language  = {english},
  mrclass   = {14-02},
  mrnumber  = {2964027},
  pages     = {xii+639},
  url       = {http://dx.doi.org/10.1017/CBO9781139084437},
}
 \bibliographystyle{alpha}
\end{document}